\documentclass[12pt]{amsart}
\usepackage{amsmath,amssymb,bbm,ifthen,a4wide,
  wrapfig,stmaryrd,url,textcomp,wasysym,calc}
\usepackage{enumerate}
\usepackage{graphicx}
\usepackage[colorlinks,final,backref=page,hyperindex]{hyperref}

\newcommand{\NN}{\mathbb{N}}

\newcommand{\QQ}{\mathbb{Q}}
\newcommand{\RR}{\mathbb{R}}
\newcommand{\CC}{\mathbb{C}}

\newcommand{\der}{\partial}

\renewcommand{\ker}{\operatorname{Ker}}
\newcommand{\im}{\operatorname{Im}}

\newcommand{\id}{\mathrm{id}}

\def\<{\langle}
\def\>{\rangle}

\newcommand{\trans}[1]{
  \mathrel{\vbox{\offinterlineskip\ialign{%
    \hfil##\hfil\cr
    $\scriptscriptstyle*$\cr
    \noalign{\kern0.3ex}
    #1\cr
}}}}

\newcommand{\cum}{{\textstyle \varint}}

\newcommand{\mmt}[2]{d_{#1,#2}}
\newcommand{\nonzero}[1]{#1^\times}
\newcommand{\fafac}[2]{#1^{\underline{#2}}}

\newcommand{\matrices}{\mathcal{M}}
\newcommand{\matspc}[3]{#1^{#2}_{#3}}
\newcommand{\matmon}[2]{\matrices_{#2}(#1)}
\newcommand{\matfix}[2]{\tilde{\matrices}_{#2}(#1)}
\newcommand{\substmon}[1]{\matrices(#1)}

\newcommand{\matma}{K[\mathcal{M}]}
\newcommand{\matsa}{K[\nonzero{\mathcal{M}}]}

\newcommand{\galg}{\mathcal{F}}
\newcommand{\ogalg}{\mathcal{G}}
\newcommand{\balg}{\mathcal{H}}

\newcommand{\trp}[1]{#1^\intercal}

\newcommand{\evl}{\text{\scshape\texttt e}}
\newcommand{\indhier}[1]{\mathfrak{H}(#1)}

\newcommand{\BialgHom}{\mathrm{BialgHom}}
\newcommand{\AlgHom}{\mathrm{AlgHom}}

\newcommand{\counit}{\setbox0=\hbox{\text{$1$}}%
  \raisebox{\ht0}{\text{\scalebox{1}[-1]{$1$}}}}
\newcommand{\circplus}{\mathrel{+\mkern-11.5mu\circ}}

\newcommand{\smallmat}[4]{\big(\begin{smallmatrix}#1&#2\\
    #3&#4\end{smallmatrix}\big)}

\newcommand{\sbigotimes}{%
  \mathop{\mathchoice{\textstyle\bigotimes}{\bigotimes}{\bigotimes}{\bigotimes}}}
\newcommand{\sprod}{%
  \mathop{\mathchoice{\textstyle\prod}{\prod}{\prod}{\prod}}}

\let\phi\varphi
\let\epsilon\varepsilon
\let\mathbb\mathbbm

\newtheorem{theorem}{Theorem}[section]
\newtheorem{proposition}[theorem]{Proposition}
\newtheorem{lemma}[theorem]{Lemma}
\newtheorem{corollary}[theorem]{Corollary}
\newtheorem{conjecture}[theorem]{Conjecture}

\theoremstyle{definition}
\newtheorem{definition}[theorem]{Definition}
\newtheorem{example}[theorem]{Example}
\newtheorem{remark}[theorem]{Remark}

\title[Multivariable Integration and Linear Substitutions]{%
  An Algebraic Study of Multivariable\\ Integration and Linear
  Substitution\footnote{MCS: 47G10, 16W99, 16T10, 68W30; 12H05, 47G20,
    13N10, 16W70}}

\author{Markus Rosenkranz}
\address{
 	School of Mathematics, Statistics and Actuarial Science,
	University of Kent,
	Canterbury CT2 7NF, England}
\email{M.Rosenkranz@kent.ac.uk}
\author{Xing Gao}
\address{School of Mathematics and Statistics,
Key Laboratory of Applied Mathematics and Complex Systems,
Lanzhou University, Lanzhou, Gansu, 730000, P.R. China}
\email{gaoxing@lzu.edu.cn}
\author{Li Guo}
\address{
Department of Mathematics and Computer Science,
Rutgers University,
Newark, NJ 07102, US}
\email{liguo@rutgers.edu}

\date{\today}

\begin{document}
\maketitle
\tableofcontents

\begin{abstract}
  We set up an algebraic theory of multivariable integration, based on
  a hierarchy of Rota-Baxter operators and an action of the matrix
  monoid as linear substitutions. Given a suitable coefficient domain
  with a bialgebra structure, this allows us to build an operator ring
  that acts naturally on the given Rota-Baxter hierarchy. We
  conjecture that the operator relations are a noncommutative
  Gr\"obner basis for the ideal they generate.
\end{abstract}

%=====================================================================
\section{Introduction}
%=====================================================================

%---------------------------------------------------------------------
\subsection{Motivation}
%---------------------------------------------------------------------

The notion of \emph{integral operator} plays a
fundamental role in Analysis, Physics and Stochastics. Apart from some
early beginnings, the systematic study of integral operators and
integral equations started in earnest by the end of the nineteenth
century, at the hands of towering figures like Hilbert, Fredholm and
Volterra.

Despite their indubitable origin in Analysis, it is often profitable
to study integral operators in an algebraic setting, i.e.\@ to model
them by the crucial notion of \emph{Rota-Baxter
  algebras}~\cite{Guo2012}. The situation is similar to the case of
\emph{differential algebra}, defined to be an (associative) algebra
equipped with a derivation. Its study originated from the algebraic
study of differential equations by Ritt~\cite{Ritt1932,Ritt1966}
starting in the 1930s and has since developed into a vast area of
mathematics research with broad applications, including the mechanical
proof of geometric theorems~\cite{Kolchin1973,Wu1987}. An algebraic
structure that encodes integration first appeared as a special case of
a Rota-Baxter algebra, defined to be an algebra $R$ equipped with a
linear operator $P$ on $R$ satisfying the Rota-Baxter axiom
\begin{equation*}
  P(u)P(v)=P(uP(v))+P(P(u)v)+\lambda P(uv)
\end{equation*}
for all $u, v \in R$. Here $\lambda \in R$ is a prefixed constant
called the weight, which is zero in the case~$P = \int_0^x$ but
nonzero for its discrete analog (partial summation). Let us now name a
few important areas where this approach proved to be successful:
\begin{itemize}
\item The original formulation of Rota-Baxter algebra (first called
  Baxter algebras) was in the context of \emph{fluctuation theory},
  when Glen Baxter gave a new proof~\cite{Baxter1960} of the so-called
  Spitzer identity, based on the Rota-Baxter axiom. This line of
  research continues to this day; see the references~[12, 61, 121,
  163, 174] in the survey monograph~\cite{Guo2012}.
\item The use of Rota-Baxter algebra for problems in
  \emph{combinatorics} has been firmly established and widely
  popularized~\cite{Rota1969,Rota1995} by Gian-Carlo Rota (which is why
  they are partly named after him). For a typical example,
  see~\cite[\S3.3]{Guo2012} and the references therein.
\item One of the main applications of Rota-Baxter algebras is the
  operator form of the so-called \emph{classical Yang-Baxter
    equation}, named after the physicists Chen-Ning~Yang and Rodney
  Baxter (\emph{not} Glen Baxter). For an overview and current
  developments in this area, see the references~[14, 26, 46, 127, 164]
  in~\cite{Guo2012}.
\item Another important application in Physics is in the
  \emph{renormalization} of perturbative quantum field theory, in
  particular the algebraic Birkhoff decomposition established by
  Connes and Kreimer~\cite{ConnesKreimer2000}. A survey of current
  literature in this domain is given in~\cite[\S2.5]{Guo2012}.
\item The relation of Rota-Baxter algebras and \emph{multiple zeta
    values} and their generalizations is a fascinating
  topic~\cite{Hoffman1997,Hoffman2005,Zagier1994}. For more on this
  topic, the reader may read~\cite[\S3.2]{Guo2012} and the literature
  survey in~\cite[\S3.4]{Guo2012}.
\item Algorithmic tools for solving and factoring \emph{linear
    boundary problems}~\cite{Duffy2001,Stakgold1979} are
  developed~\cite{Rosenkranz2005,RosenkranzRegensburger2008a,RosenkranzRegensburgerTecBuchberger2012}
  based on integro-differential algebras (Rota-Baxter algebras with a
  suitable differential structure). The solution of the boundary
  problem is denoted by the Green's operator, an element in a ring of
  integral operators.
\end{itemize}

In this paper, we shall employ Rota-Baxter algebras for creating an
\emph{algebraic theory of multivariable integation} under the action
of linear substitutions. As far as we are aware, this is the first
time that integral operators and substitutions are studied from a
systematic algebraic perspective.

Our original motivation comes from the last of the application areas
listed above: One may build up an \emph{abstract theory of linear
  boundary problems} and their Green's
operators~\cite{RegensburgerRosenkranz2009}, ``abstract'' in the sense
that it abstracts from the ring-theoretic structures (derivations,
Rota-Baxter operators, characters), retaining only their linear
skeleton. The advantage of the abstract theory is that it encompasses
both linear ordinary differential equations (LODEs) and linear partial
differential equations (LPDEs), both scalar and vector problems, both
continuos (differential) and discrete (difference) equations. However,
unlike the integro-differential setting
of~\cite{RosenkranzRegensburger2008a}, it does not---and
cannot---provide algorithmic tools.

For developing an algebraic theory of boundary problems for LPDEs,
along with the required algorithmic tools, it is first of all
necessary to develop a suitable \emph{ring of partial
  integro-differential operators} capable of expressing the Green's
operators of (some of) these boundary problems. An adhoc example of
such an operator ring---together with applications to solving certain
simple LPDE boundary problems---can be found
in~\cite{RosenkranzPhisanbut2013}.

It must be emphasized that a suitable ring of partial
integro-differential operators is not just the tensor product of the
corresponding ordinary integro-differential operator rings. Such a
ring would be too weak for expressing the Green's operators of even
very simple LPDEs like the inhomogeneous wave equation~$u_{tt} -
u_{xx} = f$ mentioned in~\cite[\S7]{RegensburgerRosenkranz2009}, and
the problems treated in~\cite{RosenkranzPhisanbut2013}. The reason for
this is evident from a geometric perspective: The tensor product
provides only \emph{axis-parallel integrators} while most Green's
operators require integration over oblique domains (slanted lines,
triangular areas, skewed volumes). This contrasts with differential
operators, where any directional derivative can be represented---via
the chain rule---as a linear combination of partial
(``axis-parallel'') derivatives.

For rectifying this defect, it is sufficient to add \emph{linear
  substitutions} to the operator ring. Their algebraic interation with
the axis-parallel integrators is regulated by (a special case of) the
substitution rule of integration. The precise algebraic formulation
turns out to be surprisingly delicate, and the current axiomatization
(Definition~\ref{def:hier-Rota-Baxter}) should be seen as a first
attempt in this direction. Our central question can be stated thus:
``What is the algebraic structure of linear substitutions under
partial integral operators?'' See the Conclusion for some further
ideas on how to improve and generalize the current setup.

The presence of linear substitutions also makes the choice of the
coefficient algebra much more delicate (see the beginning of
Section~\ref{sec:rota-baxter-bialgebras} for a short explanation). To
incorporate closure under linear substitutions into the algebraic
setting, we require the coefficients to form a \emph{Hopf algebra}
with a suitably compatible Rota-Baxter operator; we call the resulting
structure a Rota-Baxter bialgebra
(Definition~\ref{def:Rota-Baxter-bialgebra}).

For simplifying the further treatment, we have decided to keep
differential operators out of the picture. In other words, we are
building up rings of \emph{partial integral operators} (with linear
substitutions and coefficient functions), rather than partial
integro-\emph{differential} operators. In view of the applications for
boundary problems, this is clearly not sufficient but it provides a
crucial intermediate stage. In fact, the adjunction of partial
differential operators (coming from derivations of the base algebra)
leads to new operator relations that are much simpler than the ones
treated here: Loosely speaking, differential operators ``march
through'' to the far right (unlike integral operators) since the chain
rule leads to a full decoupling, as indicated by our earlier remarks
above. We shall develop the resulting operator ring in a subsequent
paper, along with its application to expressing Green's operators of
LPDE boundary problems.

Apart from this deliberate restriction, the operator rings constructed
in this paper are a \emph{vast generalization} of the adhoc ring
created
in~\cite{RosenkranzPhisanbut2013,RosenkranzRegensburgerTecBuchberger2009}
for the purpose of illustrating certain techniques for solving LPDE
boundary problems. To wit, the current ring is not based on the
specific model of multivariate analytic functions as
in~\cite{RosenkranzPhisanbut2013}, and it allows for a wide class of
coefficient algebras (arbitrary Rota-Baxter bialgebras) rather than
just constant coefficients.

From the viewpoint of Analysis, one may also pose the following
question: Let~$R$ be the subring of the ring of linear operators on
the complex function space~$C^\infty(\RR^n)$ generated by the partial
integral operators, the multiplication operators induced by certain
coefficient functions, and all linear substitutions. If~$\Gamma$
denotes the set of all these generators, there is an evident
map~$\CC\langle\Gamma\rangle \to R$ that ``evaluates'' the
indeterminates by the corresponding generators. The
kernel~$\mathcal{R}$ of this map encodes the ideal of \emph{operator
  relations} satisfied in Analysis. If these are known, one has a full
algebraization in the sense that~$R \cong
\CC\langle\Gamma\rangle/\mathcal{R}$. While we have not reached this
goal in the present paper, the relations~$\mathcal{R}'$ used in our
formulation of the operator ring
(Definition~\ref{def:ring-partial-intops}) are clearly an important
subset of~$\mathcal{R}$, and there are good reasons to conjecture that
actually~$\mathcal{R} = \mathcal{R}'$. However, this will presumably
need an analytic proof that has to be approached elsewhere.

%---------------------------------------------------------------------
\subsection{Structure of the Paper}
%---------------------------------------------------------------------

In Section~\ref{sec:hier-intdiffalg} we introduce the main structures
needed for the algebraic description of multivariable integration and
linear substitutions. After providing some motivation, we start in
Subsection~\ref{ssec:background-rbh} with an axiomatization of
function algebras over a field~$K$ with a contravariant action of the
matrix monoid~$(\matmon{K}{n}, \cdot)$. These structures, subsequently
named $K$-\emph{hierarchies} (Definition~\ref{def:hierarchy}), serve
as the basic building blocks for the partial integral operator
rings---the latter act on $K$-hierarchies, and they take their
coefficients from among them. Every $K$-hierarchy induces a lattice of
subalgebras, which provide an algebraic formulation for functions
depending on certain prescribed variables. In
Subsection~\ref{ssec:ex-prop-rbh} we add a suitable collection of
Rota-Baxter operators to a $K$-hierarchy to arrive at the pivotal
notion of a \emph{Rota-Baxter hierarchy}
(Definition~\ref{def:hier-Rota-Baxter}). Its core ingredient is an
algebraic formulation of those cases of the substitution rule of
integration that are needed for obtaining normal forms in the operator
ring. The basic properties of Rota-Baxter hierarchies are then
developed (Lemmas~\ref{lem:simple-prop}, \ref{lem:gen-transvec},
\ref{lem:slack}). The central example of a Rota-Baxter hierarchy is
the collection of all smooth multivariate functions
(Example~\ref{ex:classical}).

The introduction of coefficients is studied in
Section~\ref{sec:rota-baxter-bialgebras}. Ignoring first the
Rota-Baxter structure, we start in
Subsection~\ref{ssec:lin-subst-bialg} by describing an algebraic
structure that captures ``separated'' multivariate functions under
linear substitutions. While the coproduct of the bialgebra provides
the basic linear substitution~$x \mapsto x+y$, proper linear
combinations like~$x \mapsto x+\lambda y$ can only be achieved via an
additional ``scaling'' structure, thus leading to the notion of scaled
bialgebra (Definition~\ref{def:scaled-bialgebra}). Intuitively, the
elements in such a scaled bialgebra are the univariate functions from
which one builds all available coefficient functions via linear
substitutions. This is made precise by the \emph{induced hierarchy}
(Definition~\ref{def:induced-hierarchy}), which is indeed a
$K$-hierarchy (Proposition~\ref{prop:induced-hierarchy}), and of
course the smooth algebra---along with certain subalgebras---provides
the prototypical model for this situation
(Example~\ref{ex:scaled-bialg}). Next we combine scaled bialgebras
with Rota-Baxter operators (Subsection~\ref{ssec:int-bialg}). This
requires a certain compatibility relation, which is in fact a special
case of the substitution rule
(Definition~\ref{def:Rota-Baxter-bialgebra}); the induced hierarchy of
such a \emph{Rota-Baxter bialgebra} is then a Rota-Baxter hierarchy
(Theorem~\ref{thm:induced-hierarchy}), and this holds in particular
for the smooth algebra (Example~\ref{ex:Rota-Baxter-bialgebra}). They
form the appropriate coefficient algebras for the operator ring
(Definition~\ref{def:separated-algebra}).

Section~\ref{sec:pios} is devoted to the construction of the operator
ring, focusing first on the identification and verification of
fundamental \emph{operator relations}
(Subsection~\ref{ssec:operator-relations}). From the perspective of
rewriting, the effect of these relations can be described intuitively
as follows: In the first step (Proposition~\ref{prop:1d-subst-rule}),
one \emph{normalizes one-dimensional integrators}---meaning integral
operators with subsequent linear substitution and coefficient---by
ensuring that their associated linear substitution is ``minimal'' (for
integration along the $x_i$-axis: identity matrix except beneath the
$i$-th diagonal element). Then the normalized one-dimensional
integrators are put in ascending \emph{order}
(Proposition~\ref{prop:order-int}), and \emph{repetitions are
  eliminated} (Proposition~\ref{prop:coalesc}). The monomials in the
resulting normal form will therefore consist of a strictly ascending
list of one-dimensional integrators, as will be shown later
(Theorem~\ref{thm:confluence}). In the final
Subsection~\ref{ssec:const-quot} we define the desired \emph{ring of
  partial integral operators} and \emph{linear substitutions}
(Definition~\ref{def:ring-partial-intops}). By its very defintion, it
is clear that it has a \emph{natural action} on the given Rota-Baxter
hierarchy (Proposition~\ref{prop:action}). The corresponding rewrite
system is shown to be \emph{Noetherian}
(Theorem~\ref{thm:termination}), meaning the computation of normal
forms will always terminate. We conclude by conjecturing that the
rewrite system is moreover \emph{confluent}
(Conjecture~\ref{conj:confluence}), meaning the normal forms are in
fact unique.

In the Conclusion we provide some pointers to possible \emph{future
  work}. On the one hand, we propose some thoughts on the next obvious
steps towards building up an algebraic and algorithmic theory of LPDE
boundary problems. On the more speculative side, we discuss also some
ideas about potential generalizations and extensions of our theory.

%---------------------------------------------------------------------
\subsection{Notation}
%---------------------------------------------------------------------

The set~$\NN$ of \emph{natural numbers} is assumed to
include~$0$. If~$X$ is any set such that~$0 \in X$, we
write~$\nonzero{X} = \{ x \in X \mid x \neq 0\}$ for its nonzero
elements.

The vector space of~$m \times n$ \emph{matrices} over a field~$K$ is
denoted by~$\matspc{K}{m}{n}$, where~$m=1$ or~$n=1$ is omitted. The
unit vectors of~$K^n$ or~$K_n$ are written as~$e_i$; it will be clear
from the context what~$n$ is, and whether rows or columns are
intended. We denote the corresponding matrix rings by~$\matmon{K}{n} =
\matspc{K}{n}{n}$. The~$n \times n$ identity matrix is denoted
by~$I_n$. We write~$e_{ij} \in \matmon{K}{n}$ for the matrix units,
with the only nonzero entry~$1$ in row~$i$ and column~$j$. Similarly,
$\mmt{i}{\lambda} = I_n + (\lambda-1)e_{ii}$ denotes the scaling
matrix, which is equal to~$I_n$ except that its $i$-th diagonal entry
is~$\lambda$. The vertical composition of two matrices~$M \in
\matspc{K}{r}{n}, N \in \matspc{K}{s}{n}$ is denoted by~$M \oplus N
\in \matspc{K}{r+s}{n}$.

Given a~$K$-\emph{algebra}~$\galg$, we write~$\AlgHom_K(\galg)$ for
the monoid of $K$-algebra endomorphisms on~$\galg$. The opposite of a
monoid or algebra~$\galg$ is denoted by~$\galg^*$. If~$S$ is a
semigroup, we use the notation~$K[S]$ to refer to the semigroup
algebra over~$K$. In this paper, all (co)algebras are assumed to be
(co)commutative, so in particular bialgebras and Hopf algebras are
both commutative and cocommutative. In contract, the word ring shall
designate unital rings that are not necessarily commutative. We apply
the usual notation~$[S,T] = ST - TS$ for the commutator of ring
elements~$S$ and~$T$.

In this paper, all Rota-Baxter operators and derivations are of weight
zero. Since we think of them as inspired from analysis, we write them
as~$\cum$ and~$\der$. These symbols (along with their embellished
variants) will always be used in \emph{operator notatation}. For
example, we write~$(\der_x f) (\cum^y g)$ rather than~$\der_x(f)
\cum^y(g)$. Products---usually denoted by juxtaposition---have
precedence over operators, so~$\cum^y fg$ is to be parsed as~$\cum^y
(fg)$ and~$\cum^y f \cum^y g$ as~$\cum^y (f \cum^y g)$. Note that
derivations like~$\der_x$ are indexed below to indicate their origin
from~$\tfrac{\partial}{\partial x}$; in contrast, integrators
like~$\cum^y$ are indexed above as a reminder of~$\cum_0^y$.

In the sequel, we will use the \emph{standard variables}~$x_1, x_2,
x_3, \dots$ for defining functions of arbitrarily (but finitely) many
variables. So a definition like~$f(x_2) := x_2$ should be carefully
distinguished from~$f(x_1) := x_1$ since the first denotes the
$x_2$-projection $(x_1, x_2, \dots) \mapsto x_2$ but the second the
$x_1$-projection~$(x_1, x_2, \dots) \mapsto x_1$. Sometimes it will be
convenient to use the abbreviations~$x \equiv x_1, y \equiv x_2, z
\equiv x_3$. Since we do not use~$x$ as a shorthand for the
sequence~$(x_1, x_2, x_3, \dots)$, this will create no confusion.  In
the scope of this paper we will only deal with linear substitutions,
hence we will usually drop the qualification ``linear''.

%=====================================================================
\section{Rota-Baxter Hierarchies}
\label{sec:hier-intdiffalg}
%=====================================================================

We start by building up the basic operational domains---the algebras
which the prospective ring of partial integro-differential operators
is to operate on. Such domains will be called \emph{Rota-Baxter
  hierarchies} since they encode a notion of multivariable integration
in conjunction with an action of the corresponding substitution monoid.

%---------------------------------------------------------------------
\subsection{Background to the Concept of Rota-Baxter Hierarchy}
\label{ssec:background-rbh}
%---------------------------------------------------------------------

As a motivation to the general definition, let us first look at the
classical setting of \emph{multivariate smooth functions}. In the
latter case, we are thinking of functions~$f\colon \RR^n \to \RR$, for
any arity~$n \ge 0$, with the nullary ones denoting constants. To
simplify the book-keeping we pass to the direct limit
\begin{equation*}
  C^\infty(\RR^\infty) := \bigcup_{n \ge 0} C^\infty(\RR^n)
\end{equation*}
of functions depending on arbitrarily (but finitely) many real
variables.

Thus we think of~$\galg := C^\infty(\RR^\infty)$ as an algebra with
the ascending filtration~$\galg_n \subset \galg_{n+1}$ given
by~$\galg_n = C^\infty(\RR^n)$, so we can set up the action of
arbitrary-sized real matrices. An ascending sequence of
algebras~$(\galg_n)$ will be called an \emph{ascending algebra}.

For the rest of this section, let~$K$ be a field of characteristic
zero so that~$\QQ \subseteq K$. We write~$\substmon{K}$ for the monoid of
all $\infty \times \infty$ matrices~$M\colon \NN \times \NN \to K$
that can be written in the form~$M = I + \tilde{M}$, where~$I$ is the
$\infty \times \infty$ identity matrix and where~$\tilde{M}$ is any
row and column finite matrix (meaning all rows except finitely many
and all columns except finitely many are zero). We call~$\substmon{K}$
the (linear) \emph{substitution monoid} over~$K$. Note that it has the
natural ascending filtration
\begin{equation*}
  \substmon{K} = \bigcup_{n \ge 1} \, \big( I + \matmon{K}{n} \big),
\end{equation*}
where we use the embedding~$\matmon{K}{n} \hookrightarrow \matmon{K}{n+1}$
that sends~$M$ to~$\smallmat{M}{0}{0}{1}$. We will also identify
finite matrices~$M \in \matmon{K}{n}$ with their
embedding~$\smallmat{M}{0}{0}{I} \in \substmon{K}$. In particular, we
regard scalars~$\lambda \in K$ as~$\lambda e_{11} + e_{22} + \cdots
\in \substmon{K}$ rather than~$\lambda I$. In the sequel, we will also
need the descending chain
\begin{align*}
  \substmon{K} &= \matfix{K}{0} \supset \matfix{K}{1} \supset \matfix{K}{2}
  \supset \cdots, \qquad\text{where}\\[0.5ex]
  \matfix{K}{n} &= \Big\{
  \begin{pmatrix}
    I_n & 0\\
    0 & M
  \end{pmatrix}
  \Big|\, M \in \substmon{K} \Big\} \subseteq \matma
\end{align*}
is the subring of matrices acting trivially on~$K^n$. A rectangular
matrix~$M \in \matspc{K}{r}{s}$ with~$r < s$ is identified with the
corresponding square matrix~$\tilde{M} \in \matspc{K}{s}{s}$ obtained
from~$M$ by adding the unit vectors~$e_{r+1}, \dots, e_s$ as
additional rows. On the other hand, given~$M \in \matspc{K}{r}{s}$
with~$r > s$, we identify~$M$ with~$\tilde{M} \in \matspc{K}{r}{r}$ by
adding zero columns. In particular, any row~$v \in K_n$ may be viewed
as a matrix~$v \oplus e_2 \oplus \cdots \oplus e_n \in \matspc{K}{n}{n}$,
and adjoining a column vector~$w \in K^n$ to the identity matrix
yields the square matrix matrix~$(I_n,v) \oplus e_{n+1} \in
\matmon{K}{n+1}$. As mentioned above, all square matrices are further
embedded into~$\matma$ via~$\matmon{K}{n} \hookrightarrow
\matmon{K}{n+1}$.

In the classical setting~$K = \RR$, the canonical action of
monoid~$\substmon{\RR}$ on the ascending
$\RR$-algebra~$C^\infty(\RR^\infty)$ is defined as follows. For
given~$M \in \matmon{\RR}{k}$ and~$f \in C^\infty(\RR^l)$, we set~$n =
\max(k,l)$ so that we may take~$M \in \matmon{\RR}{n}$ and~$f \in
C^\infty(\RR^n)$ via the corresponding embeddings. Then we think
of~$M$ as effecting the \emph{change of variables}
\begin{align*}
  \bar{x}_1 &= m_{11} x_1 + \cdots + m_{1n} x_n,\\
  \vdots &\\
  \bar{x}_n &= m_{n1} x_1 + \cdots + m_{nn} x_n,
\end{align*}
and we define the action~$\matmon{\RR}{n} \times C^\infty(\RR^n) \to
C^\infty(\RR^n)$ by~$(M,f) \mapsto f \circ M$. Via the embeddings,
this yields the desired action~$\substmon{\RR} \times
C^\infty(\RR^\infty) \to C^\infty(\RR^\infty)$.

Let us now turn to the general case of a contravariant monoid
action~$\matma \times \galg \to \galg$, meaning a monoid
homomorphism~$\matma^* \to \AlgHom_K(\galg)$, where~$\matma^*$
denotes the opposite monoid of~$\matma$. We require the following
natural compatibility condition. For given~$M \in \substmon{K}$, we
write~$M^*$ for the induced mapping~$\galg \to \galg$. Moreover, we
write~$M_{\lrcorner n}$ for the $n$-th \emph{cut-off substitution},
whose first~$n$ rows coincide with those of~$M$ while the subsequent
ones are~$e_{n+1}, e_{n+2}, \dots$. In~$C^\infty(\RR^\infty)$ this
means~$M_{\lrcorner n}^*$ substitutes only in the first~$n$ variables
while leaving the remaining ones invariant. We call the action
\emph{straight} if~$M^* f = M_{\lrcorner n}^*f$ for all~$M \in
\substmon{K}$ and~$f \in \galg_n$. The canonical action
on~$C^\infty(\RR^\infty)$ is of course straight while for example
shifting the filtration to~$\galg_n := C^\infty(\RR^{n+1})$ leads to
an action that is not straight.

Another crucial property of the classical
example~$C^\infty(\RR^\infty)$ is that evaluation of a function~$f \in
C^\infty(\RR^n)$ at~$x_n = \xi \in \RR$ leaves a function
in~$C^\infty(\RR^{n-1})$. In the general case we shall require this
only for evalation at~$\xi = 0$, which can be described as the
action~$E_n^*$, where~$E_n := I - e_{nn} \in \matmon{K}{n}$ is the
$n$-th \emph{evaluation matrix} (at zero). Adding this requirement to
straightness, we arrive at the following axiomatization of
multivariate functions.

\begin{definition}
  \label{def:hierarchy}
  An ascending $K$-algebra~$(\galg_n)$ is called
  a~$K$-\emph{hierarchy} if it has a straight contravariant monoid
  action of~$\substmon{K}$ such that~$M^*(\galg_n) \subseteq \galg_n$ for
  all~$M \in \matmon{K}{n}$ and~$E_n^*(\galg_n) \subseteq
  \galg_{n-1}$. We write~$\galg$ for the direct limit
  of~$(\galg_n)$. By abuse of language, we refer also to~$\galg$ as a
  hierarchy.
\end{definition}

In detail, a contravariant action means each~$M^*\colon \galg \to
\galg$ is a homomorphism of $K$-algebras that restricts to a
homomorphism~$M^*\colon \galg_n \to \galg_n$ whenever~$M \in
\matmon{K}{n}$. Moreover, we have the usual action axioms~$I^* =
1_{\galg}$ and~$(M \tilde{M})^* = \tilde{M}^* M^*$. Note also that we
have assumed~$\galg_0 = K$, which implies that the action on~$\galg_0$
is trivial (since it fixes~$1 \in K$).

In a $K$-hierarchy~$(\galg_n)$, we can define the following
\emph{dependency lattice} for expressing that some functions depend
only on certain variables~$x_{\alpha_1}, \dots, x_{\alpha_k}$. For
convenience, let us identify strictly increasing sequences~$\alpha_1 <
\cdots < \alpha_k$ with finite subsets~$\{\alpha_1, \dots, \alpha_k\}
\subset \NN$; we will use the notation~$(\alpha_1, \dots, \alpha_k)$
for either of them. Furthermore, we shall identify permutations~$\pi
\in S_n$ with the permutation matrices~$(\pi(e_1), \dots, \pi(e_n))
\in \matmon{K}{n}$. Let~$S_\infty := \bigcup_{n \in \NN} S_n$ be the
group of all permutations with finite support (those fixing all but
finitely many elements of~$\NN$). Then we have an embedding~$S_\infty
\hookrightarrow \substmon{K}$. However, note that the action on column
vectors is accordingly
\begin{equation*}
  \pi \begin{pmatrix} v_1\\\vdots\\v_n \end{pmatrix}
  = \begin{pmatrix} v_{\tilde{\pi}(1)}\\\vdots\\v_{\tilde{\pi}(n)} \end{pmatrix}
\end{equation*}
where~$\tilde{\pi}$ is the inverse of~$\pi$. We introduce the
$K$-subalgebras
\begin{equation*}
  \galg_\alpha := \{ f \in \galg \mid \pi^* f \in \galg_k \}
\end{equation*}
where~$\pi := \pi_\alpha\colon \NN \to \NN$ is the unique permutation
with finite support sending~$j$ to~$\alpha_j$ for~$j = 1, \dots, k$
while restricting to an increasing map~$\NN\setminus\{1,\dots,k\} \to
\NN\setminus\alpha$. (In fact, the action of~$\pi$ outside
of~$\{1,\dots,k\}$ is immaterial because of the straightness
assumption.) By convention we set~$\galg_\emptyset = \galg_0$. One
checks immediately, using the straightness of the action, that~$\alpha
\mapsto \galg_\alpha$ is a monotonic map (in the sense that it
preserves inclusions). Hence we may extend it to arbitrary~$\alpha
\subseteq \NN$ by defining
\begin{equation*}
  \galg_\alpha = \bigcup_{n=1}^\infty \galg_{(\alpha_1, \dots, \alpha_n)}.
\end{equation*}
This yields a complete bounded lattice~$(\galg_\alpha)$
with~$\galg_\alpha \sqcup \galg_\beta = \galg_{\alpha \cup \beta}$
and~$\galg_\alpha \sqcap \galg_\beta = \galg_{\alpha \cap \beta}$,
with bottom element~$\galg_\emptyset = K$ and top element~$\galg_\NN =
\galg$. Moreover, the lattice is complemented with~$\galg_\alpha' =
\galg_{\NN\setminus\alpha}$. Intuitively, $\galg_\alpha$ captures
those functions that depend \emph{at most} on the variables specified
in~$\alpha$ and~$\galg_\alpha'$ those that do \emph{not} depend on these
variables.

The usual \emph{substitution notation}~$f(M_{11} x_1 + \cdots + M_{1n}
x_n, \dots, M_{n1} x_1 + \cdots + M_{nn} x_n)$ can be viewed as a
convenient shorthand for~$M^*(f)$, for a given substitution matrix
\begin{displaymath}
  M = \begin{pmatrix}
    M_{11} & \cdots & M_{1n}\\
    \vdots & \ddots & \vdots\\
    M_{n1} & \cdots & M_{nn}
  \end{pmatrix}
\end{displaymath}
and~$f \in \galg_n$. While we will not employ this notation in the
present paper (a more suitable notation is introduced in
Section~\ref{ssec:operator-relations}), it is certainly useful in a
computational context.

The next step is to add \emph{Rota-Baxter operators} and to describe
their interaction with substitutions. This will lead to an
algebraization of the well-known substitution rule for integrals. For
basic definitions and properties of Rota-Baxter algebras we refer
to~\cite{Guo2012,GuoKeigher2000}.

Let~$R$ be a ring containing~$\QQ$, and let~$(\galg, P)$ be a
Rota-Baxter algebra over~$R$. Then we call~$(\galg, P)$
\emph{ordinary} if~$P$ is injective and~$\im(P) \dotplus R = \galg$ as
$R$-modules. This is an algebraic way of describing~$P$ as an integral
operator on ``univariate functions''. In fact, we get an ordinary
integro-differential algebra~$(\galg, d, P)$, where~$d\colon \galg \to
\galg$ is the unique derivation that sends~$P(f) + c$ to~$f$, for
arbitrary~$f \in \galg$ and~$c \in R$. Hence~$1_\galg - P \circ d$ is
an algebra homomorphism~$\galg \to R$, which we call the
\emph{associated evaluation} of~$(\galg, d, P)$; it is the projector
corresponding to the direct sum~$\im(P) \dotplus R = \galg$.

Having an ordinary Rota-Baxter algebra has the added benefit of having
the \emph{polynomial ring} at our disposal. This holds for all
ordinary integro-differential
algebras~\cite[Prop.~3]{BuchbergerRosenkranz2012}, but we give an
independent proof here that does not make use of the derivation.

\begin{lemma}
  \label{lem:ord-rba-poly}
  Let~$(\galg, P)$ be an ordinary Rota-Baxter algebra over~$R$. Then~$x \mapsto
  P(1)$ defines an embedding~$(R[x], \cum_0^x) \hookrightarrow (\galg, P)$ of
  Rota-Baxter algebras.
\end{lemma}
\begin{proof}
  Since~$(R[x], \cum_0^x)$ is the initial object in the category of
  Rota-Baxter $R$-algebras~\cite[Cor.~3.2.4]{Guo2012}, there is a
  unique Rota-Baxter morphism~$\iota\colon R[x] \to \galg$, which
  clearly satisfies~$\iota(x) = \iota(\cum_0^x 1) = P(1)$. It remains
  to check that~$\iota$ is injective, so
  %
  % Writing the given mapping as~$\iota\colon R[x] \to \galg$, we prove first
  % that~$\iota$ respects the Rota-Baxter structure. It suffices to do so on the
  % $R$-basis~$x^n$. Using the well-known fact~\cite[Thm.~3.3.1]{Guo2012}
  % that~$P(1)^n = n! \, P^n(1)$, we obtain
  % %
  % \begin{align*}
  %   \iota(\cum_0^x x^n) &= (n+1)^{-1} \iota(x^{n+1}) = (n+1)^{-1} \, P(1)^{n+1} =
  %   n! \, P^{n+1}(1)\\ &= P(n! \, P^n(1)) = P(P(1)^n) = P(\iota(x^n))
  % \end{align*}
  % %
  % as required. For proving injectivity of~$\iota$, 
  %
  we show that~$\iota(p) = 0$ implies~$p=0$ for all polynomials~$p \in
  R[x]$. We use induction on the degree of~$p$. The induction base~$p
  \in R$ is trivial since by definition~$\iota$ acts as the identity
  on~$R$. Now assume the claim for all polynomials of degree less
  than~$k > 0$, and take~$p = p_0 + p_1 x + \cdots + p_k x^k$
  with~$\iota(p) = 0$. Using the property~$P(1)^i = i! \, P^i(1)$
  from~\cite[Thm.~3.3.1]{Guo2012}, we obtain
  \begin{align*}
    -p_0 = \sum_{i=1}^k p_i \, \iota(x)^i = \sum_{i=1}^k i! \, p_i \,
    P^i(1) = P\left(\sum_{i=0}^{k-1} (i+1)! \, p_{i+1} \,
      P^i(1)\right) \in \im(P),
  \end{align*}
  and~$\im(P) \dotplus R = \galg$ implies that~$p_0$ as well as the expression
  on the left-hand side above vanish. Since~$P$ is injective, this implies
  \begin{equation*}
    \iota \left( \sum_{i=0}^{k-1} (i+1) \, p_{i+1} x^i \right) =
    \sum_{i=0}^{k-1} (i+1) \, p_{i+1} P(1)^i =
    \sum_{i=0}^{k-1} (i+1)! \, p_{i+1} P^i(1) = 0,
  \end{equation*}
  and the induction hypothesis yields~$(i+1) \, p_i = 0$ for
  all~$i=0,\dots,k-1$, and hence~$p=0$.
\end{proof}

For an ascending algebra~$(\galg_n)$, it is natural to require an
infinite \emph{collection of Rota-Baxter operators} that we shall
write as~$\cum^{x_n}$.  Since we think of~$\galg_1$ as univariate
functions, we shall require that~$\cum^{x_1}$ is an ordinary
Rota-Baxter operator over~$K = \galg_0$. Analogous assumptions are
imposed for~$\cum^{x_n}$ so that~$\galg_n$ is an ordinary Rota-Baxter
algebra over~$\galg_{n-1}$. For the Rota-Baxter
operators~$\cum^{x_n}$ we shall now postulate the substitution rule
for integration, which we shall only need for certain particular
substitutions.

We introduce the following two special matrices. For~$i < n$ we define
the general \emph{transvection} (= horizontal shear) in the
$x_i$-direction as
\begin{equation}
  \label{eq:gen-transvec}
  T_i(v) =
  \begin{pmatrix}
    1\\
    & \ddots\\
    & & 1\\
    v_1 & \cdots & v_{i-1} & 1 & v_{i+1} & \cdots & v_n
    \rlap{\qquad\text{$\leftarrow$ row $i$,}}\\
    &&&& 1\\
    &&&&& \ddots\\
    &&&&&& 1
  \end{pmatrix}
\end{equation}
for a vector conveniently written as~$v = \trp{(v_1, \dots, v_{i-1},
  v_{i+1}, \dots, v_n)} \in K^{n-1}$. Similarly, we define the
\emph{eliminant} (= subdiagonal vertical shear)
\begin{equation}
  \label{eq:lower-tri}
  L_i(w) =
  \begin{pmatrix}
    1\\
    &\ddots\\
    &&1\\
    &&&1\\
    &&&w_{i+1} & 1 &&& \rlap{\qquad\text{$\leftarrow$ row $i+1$,}}\\
    &&&\vdots && \ddots\\
    &&&w_n &&& 1
  \end{pmatrix}
\end{equation}
for a vector written as~$w = \trp{(w_{i+1}, \dots, w_n)} \in
K^{n-i}$. When using the abbreviated variables~$(x,y) = (x_1, x_2)$,
we shall also write~$T_x(v)$ and~$T_y(v)$ for~$T_1(v)$ and~$T_2(v)$,
respectively, and similar abbreviations will be in force for the
eliminants. We note also the composition rule~$L_i(w) \,
L_i(\tilde{w}) = L_i(w+\tilde{w})$ so that~$L_i^{-1}(v) = L_i(-v)$ for
eliminants in the same direction. However, for distinct directions we
have the rules
\begin{align}
  \label{eq:lmat-asc}
  L_i(w) \, L_j(u) &=
  \begin{pmatrix}
    1\\
    &\ddots\\
    &&1\\
    &&&1\\
    &&&w_{i+1} & 1 &&&&&\rlap{\qquad\text{$\leftarrow$ row $i+1$}}\\
    &&&\vdots && \ddots\\
    &&&\vdots &&& 1\\
    &&&\vdots &&& u_{j+1} & 1 && \rlap{\qquad\text{$\leftarrow$ row $j+1$}}\\
    &&&\vdots &&& \vdots && \ddots\\
    &&&w_n &&& u_n    &&& 1
  \end{pmatrix}\kern3.5cm\\[1ex]
  \label{eq:lmat-desc}
  L_j(u) \, L_i(w) &= L_i(w') \, L_j(u)
\end{align}
with $i<j$ and~$w' := L_{j-i}(u) \, w \in K^{n-i}$. In particular, we
note that the~$L_i(w)$ and~$L_j(u)$ do not commute.

We can now introduce the key structure that we shall use as an
algebraic model of \emph{multivariate functions under integration and
  substitution}. Roughly speaking, we shall need the subsitution rule
for integration just for the following three matrix types: scalings,
transvections, and eliminants.

\begin{definition}
  \label{def:hier-Rota-Baxter}
  A \emph{Rota-Baxter hierarchy}~$(\galg_n, \cum^{x_n})_{n \in \NN}$ consists of
  a $K$-hierarchy~$(\galg_n)$ and commuting Rota-Baxter operators~$(\cum^{x_n})$
  that satisfy the following axioms:
  \begin{enumerate}[(a)]
  \item\label{it:mapping-prop} We have~$\cum^{x_n} \galg_m \subseteq \galg_m$
    and~$\cum^{x_n} \matfix{K}{m}^* = \matfix{K}{m}^* \, \cum^{x_n}$ for~$n \le m$.
  \item\label{it:ordinary-piece} Every~$(\galg_n, \cum^{x_n})$ is an ordinary
    Rota-Baxter algebra over~$\galg_{n-1}$ with evaluation~$E_n^*$.
  \item\label{it:perm-action} We have~$\tau^* \cum^{x_i} = \cum^{x_j} \tau^*$
    for the transposition~$\tau = (i \: j)$.
  \item\label{it:sub-rule} For $\lambda \in \nonzero{K}$ and~$v = (0,
    v') \in K \oplus K^{n-2}$ we require the substitution rules
    \begin{align}
      \label{eq:scaling-spc}
      &\cum^x \lambda^* = \lambda^{-1} \, \lambda^* \cum^x,\\
      \label{eq:transvec-spc}
      &\cum^x \, T_x(e_1)^* = (1-E_x^*) \, T_x(e_1)^* \cum^x,\\
      \label{eq:vert-subst-spc}
      &\cum^x L_x(e_1+v)^* \cum^x = L_y^{-1}(v')^* 
      \smash{\big[ L_x(e_1)^*, \cum^x \big] \cum^{y\,} \, L_y(v')^*}.
    \end{align}
  \end{enumerate}
  We call~\eqref{eq:scaling-spc}, \eqref{eq:transvec-spc},
  \eqref{eq:vert-subst-spc} the \emph{diagonal}, \emph{horizontal} and
  \emph{vertical} substitution rule, respectively.
\end{definition}

Note that~\eqref{eq:scaling-spc} describes the effect of scaling,
\eqref{eq:transvec-spc} is essentially a one-dimensional substitution
rule while~\eqref{eq:vert-subst-spc} is the significant part of the
two-dimensional substitution rule, which will be used for collapsing
multiple integrals along the same direction. It turns out that all
other instances of the substitution rule needed here can be inferred
from the three
instances~\eqref{eq:scaling-spc}--\eqref{eq:vert-subst-spc}. In
particular, we note immediately that they imply the slightly
\emph{more general cases} of the three substitution rules
\begin{align}
  \label{eq:scaling}
  &\cum^{x_i} \mmt{i}{\lambda} = \lambda^{-1} \, \mmt{i}{\lambda}
  \, \cum^{x_i}\\
  \label{eq:transvec}
  &\cum^{x_j} \, T_j(e_i)^* = (1-E_j^*) \, T_j(e_i)^* \cum^{x_j}\\
  \label{eq:vert-subst-without-mult}
  &\cum^{x_j} L_j(e_i+v)^* \cum^{x_j} = L_{i+1}^{-1}(v')^* 
  \smash{\big[ L_j(e_i)^*, \cum^{x_j} \big] \cum^{x_{i+1}} L_{i+1}(v')^*}
\end{align}
for any~$i, j>0$, as one may readily check using conjugation by
transpositions. In the vertical
rule~\eqref{eq:vert-subst-without-mult} we require~$v = (0, v') \in
K^{i-j+1} \oplus K^{n-i-1}$ and~$j \le i$.

%---------------------------------------------------------------------
\subsection{Examples and Properties of Rota-Baxter Hierarchies}
\label{ssec:ex-prop-rbh}
%---------------------------------------------------------------------

The \emph{vertical substitution rule}~\eqref{eq:vert-subst-spc} can be
formulated in the following equivalent way, which is more symmetric
and perhaps more natural (but less economical for our present
purposes). This can be useful for proving that something is a
Rota-Baxter hierarchy (as we will do in
Theorem~\ref{thm:induced-hierarchy}).

\begin{lemma}
  \label{lem:vert-subst-alt}
  Let~$\Lambda \subset \nonzero{\NN}$ be an arbitrary finite index set
  with minimal element~$\lambda \in \Lambda$ and complement~$\Lambda'
  = \Lambda \setminus \{ \lambda \}$. Then
  \begin{equation}
    \label{eq:vert-subst-alt}
    \cum^x L_x\Big( \sum_{i \in \Lambda} e_i \Big)^* \cum^x
    = L_{\lambda+1}^{-1} \Big(\sum_{i \in \Lambda'} e_{i-\lambda}
    \Big)^* \big[ L_x(e_\lambda)^*, \cum^x \big] \,
    \cum^{x_{\lambda+1}} L_{\lambda+1} \Big( \sum_{i \in
      \Lambda'} e_{i-\lambda} \Big)^*
  \end{equation}
  is equivalent to~\eqref{eq:vert-subst-spc}, assuming all other
  axioms of a Rota-Baxter hierarchy~$(\galg_n, \cum^{x_n})_{n \in
    \NN}$.
\end{lemma}

\begin{proof}
  It is easy to see that~\eqref{eq:vert-subst-alt} is necessary.
  Indeed, let~$n$ be the largest index of the set~$\Lambda$. We may
  assume that~$\lambda = 1$ since otherwise we can set up a
  permutation~$\pi \in S_n$ such that~$\pi(\lambda) = 1$; then
  conjugation of~\eqref{eq:vert-subst-alt} by~$\pi^*$ will ensure
  this. Now setting~$v' = \sum_{i \in \Lambda'} e_i$
  in~\eqref{eq:vert-subst-spc} immediately yields
  formula~\eqref{eq:vert-subst-alt}. Note that the appearance of
  the~$e_{i-\lambda} = e_{i-1}$ in the latter formula corresponds to
  the transition from~$v \in K^{n-1}$ to~$v' \in K^{n-2}$ in the
  former.

  For proving that~\eqref{eq:vert-subst-alt} is sufficient, let~$v =
  (0,v') \in K \oplus K^{n-2}$ be given, and let~$\Lambda' \subset
  \nonzero{\NN}$ denote the set of all those~$i$ with~$v_i' \neq
  0$. Furthermore, set~$\lambda = 1$ and~$\Lambda = \{1\} \cup
  \Lambda'$. We use the scaling matrix~$S = S_3 \cdots S_n \in K^{n
    \times n}$ with
  \begin{equation*}
    S_{i+2} = 
    \begin{cases}
      \mmt{i+2}{v'_i} & \text{if $v_i \neq 0$,}\\
      I_n& \text{otherwise,}
    \end{cases}
  \end{equation*}
  for~$1 \le i \le n-2$. Then we have~$L_x(e_1 + v) = S \, L_x(\sum_{i
    \in \Lambda} e_i) \, S^{-1}$ and hence
  \begin{equation*}
    L_x(e_1+v)^* = (S_3^{-1})^* \cdots (S_n^{-1})^* \,
    L_x\Big(\sum_{i \in \Lambda} e_i \Big)^* S_3^* \cdots
    S_n^*.
  \end{equation*}
  Since the~$S_j^*, (S_j^{-1})^* \in \matfix{K}{2}^*$ commute
  with~$\cum^x$, substitution into~\eqref{eq:vert-subst-alt} yields
  \begin{align*}
    \cum^x L_x(e_1+v)^* \cum^x &= (S^{-1})^* \, L_y^{-1} \Big(\sum_{i
      \in \Lambda'} e_{i-1} \Big)^* \big[ L_x(e_1)^*, \cum^x \big] \,
    \cum^y L_y \Big( \sum_{i \in \Lambda'} e_{i-1} \Big)^* S^*\\
    &= (S^{-1})^* \, L_y^{-1} \Big(\sum_{i \in \Lambda'} e_{i-1}
    \Big)^* S^* \big[ L_x(e_1)^*, \cum^x \big] \, \cum^y (S^{-1})^* 
    L_y \Big( \sum_{i \in \Lambda'} e_{i-1} \Big)^* S^*,
  \end{align*}
  where in the last step~$S^* (S^{-1})^* = 1$ was inserted
  after~$\cum^y$, and~$S^*$ was pushed left because the~$S_j^*$ also
  commute with~$\cum^y$ and~$L_x(e_1)^*$. Then we
  obtain~\eqref{eq:vert-subst-spc} since clearly~$S \, L_y^{-1}(\sum)
  \, S^{-1} = L_y^{-1}(v')$ as well as~$S \, L_y(\sum) \, S^{-1} =
  L_y(v')$.
\end{proof}

Let us now look at the most important example of a Rota-Baxter
hierarchy---the algebra of \emph{smooth functions in several
  variables}. This hierarchy contains also several important
subhierarchies, in particular the analytic functions.

\begin{example}
  \label{ex:classical}
  In the \emph{classical setting} we have~$\galg_n = C^\infty(\RR^n)$
  with Rota-Baxter operators
  \begin{equation*}
    \cum^{x_n}\colon \galg \to \galg, \qquad
    f \mapsto \int_0^{x_n} f(x_1, \dots, x_{n-1},
    \xi, x_{n+1}, \dots) \, d\xi,
  \end{equation*}
  which clearly satisfy the mapping properties required in
  Item~\eqref{it:mapping-prop} above. Moreover, it is clear that
  every~$\cum^{x_n}$ is injective, and we have~$\galg_n = \galg_{n-1} \dotplus
  \im(\cum^{x_n})$ since
  \begin{equation*}
    f(x_1, \dots, x_n) = f(x_1 \dots, x_{n-1}, 0) + \cum_0^{x_n} f'(x_1, \dots,
    x_{n-1}, \xi) \, d\xi,
  \end{equation*}
  and since~$0 = c(x_1, \dots, x_{n-1}) + \int_0^{x_n} f(x_1, \dots,
  x_{n-1}, \xi) \, d\xi$ implies~$f = 0$ upon differentiating with
  respect to~$x_n$. Thus every~$(\galg_n, \cum^{x_n})$ is an ordinary
  Rota-Baxter algebra over~$\galg_{n-1}$ with evaluation~$E_n^*\colon
  x_n \mapsto 0$, as required by Item~\eqref{it:ordinary-piece}. The
  transposition property of Item~\eqref{it:perm-action} is clear,
  while~\eqref{eq:scaling} follows by the substitution~$\bar{\xi} =
  \lambda \xi$ in the integral~$\int_0^{x_1} f(\lambda \xi, x_2,
  \dots) \, d\xi$.

  Now let us consider~\eqref{eq:transvec-spc}. Writing~$T \equiv
  T_x(e_1)$, we have
  \begin{align*}
    & \cum^{x_1} T^*\! f(x_1, x_2, x_3, \dots) = \int_0^{x_1} f(\xi +
    x_2, x_2, x_3, \dots) \, d\xi =
    \int_{x_2}^{x_1+x_2} f(\bar{\xi}, x_2, x_3, \dots) \, d\bar{\xi}\\
    &= \int_0^{x_1+x_2} f(\xi, x_2, x_3, \dots) \, d\xi - \int_0^{x_2}
    f(\xi, x_2, x_3, \dots) \, d\xi = (1-E_x^*) \, T^* \cum^{x_1}
    f(x_1, x_2, x_3, \dots)
  \end{align*}
  where the second equality employs the change of variables~$\bar{\xi}
  = \xi + x_2$.

  Finally, let us verify~\eqref{eq:vert-subst-spc}. Using the
  abbreviation~$z_{k\dots l} \equiv z_k, \dots, z_l \; (k \le l)$ for
  extracting and manipulating partial vectors with the obvious
  meaning, we have
  \begin{align*}
    &L_2(v')^* \cum^{x_1} \, L_x(e_{1}+v)^* \cum^{x_1} f(x_1, \dots,
    x_n, \dots)\\
    & \quad = L_2(v')^* \int_0^{x_1} \int_0^\eta f(\xi, x_2 + \eta,
    x_{3\dots n} + v_{3\dots n} \eta, \dots) \, d\xi \, d\eta\\
    & \quad = L_2(v')^* \int_0^{x_1} \int_{\xi+x_2}^{x_1+x_2} 
    f(\xi, \bar{\eta}, x_{3\dots n} + v_{3\dots n} (\bar{\eta}-x_2),
    \dots) \,
    d\bar{\eta} \, d\xi\\
    & \quad = \int_0^{x_1} \int_{\xi+x_2}^{x_1+x_2} f(\xi,
    \eta, x_{3\dots n} + v_{3\dots n} \eta, \dots) \, d\eta \, d\xi\\
    & \quad = \int_0^{x_1} \int_0^{x_1+x_2} \ldots \, d\eta \, d\xi -
    \int_0^{x_1} \int_0^{\xi+x_2} \ldots \, d\eta \, d\xi
  \end{align*}
  where in the second equality the integration sweeps are swapped and
  the substitution~$\bar{\eta} = x_2 + \eta$ is applied.  It is easy
  to see that the first summand is~$L_x(e_1)^* \cum^{x_1} \cum^{x_2}
  L_y(v')^* f$ and the second is~$\cum^{x_1} L_x(e_1)^* \cum^{x_2}
  L_y(v')^* f$.

  This concludes the proof that~$C^\infty(\RR^\infty)$ is a
  Rota-Baxter hierarchy over~$\RR$. A similar proof will also work for
  the \emph{analytic subhierarchy}~$C^\omega(\RR^\infty)$. This may be
  extended to complex variables as follows. Let~$\cum^x$ be the path
  integral from~$0$ to~$x \in \CC$. If~$f \in C^\omega(\CC^\infty)$ is
  a multivariate holomorphic function, we use the
  conjugates~$\cum^{x_i} = (1 \: i)^* \cum^x \, (1 \: i)^*$ for
  creating the hierarchy of Rota-Baxter operators. One sees
  immediately that~$C^\omega(\CC^\infty)$ is a Rota-Baxter hierarchy
  over~$\CC$. (There is also an intermediate case where one considers
  only complex-valued $C^\omega$ functions on~$\RR^n$, but allowing
  complex substitutions. For example, $e^{ix} = \cos(x) + i \,
  \sin(x)$ may be interpreted in that way.)
\end{example}

\begin{remark}
  \label{rem:geometric}
  The classical example provides a convenient \emph{geometrical
    interpretation} of the three substitution
  rules~\eqref{eq:scaling-spc}, \eqref{eq:transvec-spc}
  and~\eqref{eq:vert-subst-spc}. The diagonal
  rule~\eqref{eq:scaling-spc} describes the natural contravariant
  behavior when stretching or shrinking a coordinate axis.

  The horizontal rule~\eqref{eq:transvec-spc} says that integrating
  over the segment~$[0,x] \times \{y\}$ after horizontally shearing it
  to~$[y,x+y] \times \{y\}$ can be achieved by integrating over the
  whole sweep~$[0,x+y] \times \{y\}$ minus the surplus~$[0,y] \times
  \{y\}$. Its effect is that all axis-parallel line integrals may be
  started from the axis, generalizing the usual rule~$\cum_a^b \dots
  dx = \cum_0^b \dots dx - \cum_0^a \dots dx$ on the~$x$-axis to all
  parallel segments. 

  The vertical rule~\eqref{eq:vert-subst-spc} for~$v = 0$ may be seen
  to decompose an integral over the triangle $(0,y), (0,x+y), (x,x+y)$
  as an integral over the rectangle~$(0,0), (x,0), (x,x+y), (0,x+y)$
  minus an integral over the complement, namely the trapezoid~$(0,0),
  (x,0), (x,x+y), (0,y)$. Parametrizing the line segment from~$(0,y)$
  to~$(x,x+y)$ as~$\eta = s(\xi)$, the former integral is~$\cum_0^x
  \cum_0^y \dots \, d\eta \, d\xi$, the latter is~$\cum_0^x
  \cum_0^{s(\xi)} \dots \, d\eta \, d\xi$, so the effect is again to
  offset all integrals from the axes. The case~$v \neq 0$ is similar
  except that the triangle is now tilted against the $xy$-plane, but
  the same decomposition as before appears in the orthogonal
  projection.

  Finally, we should point out that the special case~$\lambda = -1$
  of~\eqref{eq:scaling-spc} means that~$(-1)^*$ and~$\cum^x$
  anti-commute. Thinking of Definition~\ref{def:hier-Rota-Baxter} as
  an axiomatization of the substitution rule of integration for the
  case of linear substitutions~$A \in \matmon{\RR}{n}$, the proper
  generalization to general ``spaces'' must be the \emph{signed
    integral} over oriented manifolds rather than the so-called area
  formula of measure
  theory~\cite[\S3.3]{EvansGariepy1992}\cite[Thm.~2.6]{Traynor1993}. The
  latter would introduce absolute values whose algebraic treatment
  would be considerably more awkward.
\end{remark}

Clearly, the notion of Rota-Baxter hierarchy gives rise to a category,
with the morphisms respecting the linear action as well as the
Rota-Baxter operators. In connection with this, several other notions
can be built up in a natural fashion but in the present context we
shall not need them. Let us only mention that a \emph{Rota-Baxter
  subhierarchy} means that corresponding filtered components are
Rota-Baxter subalgebras in the usual sense and that the linear action
of~$\matma$ coincides on them.

We proceed now by stating a few \emph{elementary consequences} of the
axioms. Though most of these are fairly obvious in the classical
setting, we have to prove them here on the basis of our axioms for
making sure that these include all the essential requirements for the
theory to be developed.

\begin{lemma}
  \label{lem:simple-prop}
  Every Rota-Baxter hierarchy~$(\galg_n, \cum^{x_n})$ satisfies the
  following properties.
  \begin{enumerate}[(a)]
  \item\label{it:poly-perm} For any~$\alpha = (\alpha_1, \dots, \alpha_k)$,
    there is an embedding
    \begin{align*}
      \iota_\alpha\colon K[X_{\alpha_1}, \dots, X_{\alpha_k}] & \hookrightarrow
      \galg_\alpha\\
      X_{\alpha_j} & \mapsto x_{\alpha_j} := \cum^{x_{\alpha_j}} 1,
    \end{align*}
    and we have~$\pi^* p(x_{\alpha_1}, \dots, x_{\alpha_k}) =
    p(x_{\pi(\alpha_1)}, \dots, x_{\pi(\alpha_k)})$ for all permutations~$\pi$
    of~$(\alpha_1, \dots, \alpha_k)$.
  \item\label{it:perm-int} For~$\pi \in S_n$ and~$i \le n$ we
    have~$\pi^* \cum^{x_i} = \cum^{x_j} \pi^*$ with~$j := \pi(i)$.  In
    particular, all~$\cum^{x_i}\colon \galg_{(i)} \to \galg_{(i)}$ are
    conjugates of~$\cum^{x_1}$ and hence ordinary Rota-Baxter
    operators.
  \item\label{it:trivial-int} We have~$\cum^{x_n} cf = c \, \cum^{x_n} f$
    for all~$c \in \galg'_{(n)}$ and~$f \in \galg$. In particular,
    $\cum^{x_n} c = cx_n$.
  \item\label{it:poly-rb-emb} The embedding~$\iota_\alpha$ of
    Item~\eqref{it:poly-perm} is a homomorphism of Rota-Baxter algebras in the
    sense that~$\iota_\alpha \circ \cum_0^{X_{\alpha_j}} = \cum^{x_{\alpha_j}}
    \circ \iota_\alpha$ for~$j = 1, \dots, k$.
  \item\label{it:zero-col} If~$M \in \matma$ vanishes in the~$i$-th column,
    then~$M^*(\galg) \subseteq \galg_{i}'$.
  \item\label{it:eval-van} We have~$E_i^* \, \cum^{x_i} = 0$ for all~$i > 0$.
  \end{enumerate}
\end{lemma}
\begin{proof}
  \eqref{it:poly-perm} It suffices to consider~$\alpha = (1, \dots, n)$ since
  restricting to a subset of~$\alpha$ induces another embedding, and
  permutations may be restricted accordingly (fixing the elements outside the
  subset).

  By definition, $\iota_n := \iota_\alpha$ is a homomorphism of
  $K$-algebras. We show that it is injective by induction on~$n$, with
  the base case~$n=1$ covered by
  Lemma~\ref{lem:ord-rba-poly}. Since~$\galg_\alpha = \galg_n$ is an
  ordinary Rota-Baxter algebra over~$\galg_{n-1}$, applying
  Lemma~\ref{lem:ord-rba-poly} again yields an embedding~$\iota\colon
  (\galg_{n-1}[X_n], \cum_0^{X_n}) \hookrightarrow (\galg_n,
  \cum^{x_n})$ defined by~$X_n \mapsto x_n$. By the induction
  hypothesis we also have an embedding~$\iota_{n-1}\colon K[X_1,
  \dots, X_{n-1}] \hookrightarrow \galg_{n-1}$ defined by~$X_i \mapsto
  x_i \; (i < n)$. By the universal property of
  polynomials~\cite[Lem.~2.15]{BeckerWeispfenning1993}, we obtain a
  unique $K$-linear map~$\iota_n\colon K[X_1, \dots, X_{n-1}][X_n] \to
  \galg_n$, which acts on coefficients via~$\iota_{n-1}$ and
  sends~$X_n$ to~$x_n$. By~\cite[Lem.~2.16]{BeckerWeispfenning1993}
  and the injectivity of~$\iota_{n-1}$, the map~$\iota_n$ is injective
  iff
  \begin{equation}
    \label{eq:lift-inj}
    \sum_{i=0}^k p_i(x_1, \dots, x_{n-1}) \, x_n^i = 0
    \quad\text{implies}\quad
    p_0, \dots, p_k = 0.
  \end{equation}
  For showing~\eqref{eq:lift-inj} note that the left-hand side is the
  image of~$\sum_i \iota_{n-1}(p_i) \, X_n^i$ under the
  embedding~$\iota$, so that~$p_0, \dots, p_k = 0$ follows from the
  injectivity of~$\iota$ and~$\iota_{n-1}$.

  Now let~$\pi \in S_n$ be an arbitrary permutation. Then
  Item~\eqref{it:perm-int} of this Lemma (whose proof below is
  independent) yields
  \begin{equation*}
    \pi^* p(x_1, \dots, x_n) = p(\pi^* \cum^{x_1} 1, \dots, \pi^* \cum^{x_n} 1)
    = p(\cum^{x_{\pi(1)}} 1, \dots, \cum^{x_{\pi(n)}} 1)
    = p(x_{\pi(1)}, \dots, x_{\pi(n)})
  \end{equation*}
  since clearly~$\pi^* 1 = 1$.

  \eqref{it:perm-int} The general conjugate relation follows from
  Item~\eqref{it:perm-action} Definition~\ref{def:hier-Rota-Baxter}
  since every~$\pi \in S_n$ is a product of transpositions. The last
  claim follows since~$\cum^{x_i} = \tau^* \cum^{x_1} \tau^*$
  with~$\tau=(1 \: i)$ is clearly injective and for given~$f \in
  \galg_{(i)}$ we have~$\tau^* f = f_0 + \cum^{x_1} f_1$ for some~$f_0
  \in K$ and~$f_1 \in \galg_{(1)}$, which implies~$f = f_0 +
  \cum^{x_i} f_i$ for~$f_i = \tau^* f_1 \in \galg_{(i)}$, and the
  decomposition is unique since the one for~$\galg_{(1)} = \galg_1$
  is.

  \eqref{it:trivial-int} For a sufficiently large~$k \ge n$ we have~$f
  \in \galg_k$ and~$c \in \galg_{(\alpha)}$ with~$\alpha := (1, \dots,
  k) \setminus (n)$. So if~$\pi$ is any permutation taking~$\alpha_i$
  to~$i$, we have~$\pi^* c \in \galg_{k-1}$. We choose a permutation
  with~$\pi(n) = k$ so that~$\pi^* f \in \galg_k$. Then
  Item~\eqref{it:ordinary-piece} of
  Definition~\ref{def:hier-Rota-Baxter} yields~$\cum^{x_k} (\pi^* c
  \cdot \pi^* \! f) = (\pi^* c) \, \cum^{x_k} \pi^* f$. If~$\tilde{\pi}$
  is the inverse of~$\pi$, left multipication by~$\tilde{\pi}^*$
  together with Item~\eqref{it:perm-int} of this Lemma gives the
  required identity.

  \eqref{it:poly-rb-emb} Again it suffices to consider~$\alpha = (1,
  \dots, n)$, so we show that~$\iota_n \circ \cum_0^{X_i} = \cum^{x_i}
  \circ \iota_n$ for fixed~$n$. In the case~$i < n$ we use finite
  induction on~$i$. The induction base~$i=1$ is then covered by
  Lemma~\ref{lem:ord-rba-poly} and Item~\eqref{it:trivial-int} above,
  so assume~$i>1$. Then we have
  \begin{align*}
    &\iota_n \circ \int_0^{X_i} \bigg( \sum_{j=0}^k p_j(X_1, \dots,
    X_{n-1}) \, X_n^j \bigg) = \iota_n \bigg( \sum_{j=0}^k X_n^j
    \int_0^{X_i} p_j(X_1, \dots, X_{n-1})
    \bigg)\\
    &\qquad = \sum_{j=0}^k x_n^j \, \iota_{n-1} \bigg( \int_0^{X_i}
    p_j(X_1, \dots, X_{n-1}) \bigg) = \sum_{j=0}^k x_n^j \, \int^{x_i}
    p_j(x_1, \dots,
    x_{n-1})\\
    &\qquad = \int^{x_i} \sum_{j=0}^k p_j(x_1, \dots, x_{n-1}) \,
    x_n^j = \int^{x_i} \circ \, \iota_n \, \bigg( \sum_{j=0}^k
    p_j(X_1, \dots, X_{n-1}) \, X_n^j \bigg)
  \end{align*}
  where we have used the induction hypothesis in the third and
  Item~\eqref{it:trivial-int} of this Lemma in the fourth equality. It
  remains to prove~$\iota_n \circ \cum_0^{X_n} = \cum^{x_n} \circ
  \iota_n$ on~$K[X_1, \dots, X_n] = K[X_1, \dots, X_{n-1}][X_n]$. To
  this end, recall that the embedding~$\iota\colon (\galg_{n-1}[X_n],
  \cum_0^{X_n}) \hookrightarrow (\galg_n, \cum^{x_n})$ from the above
  proof of Item~\eqref{it:poly-perm} is a Rota-Baxter homomorphism by
  Lemma~\ref{lem:ord-rba-poly}. Hence we obtain
  \begin{align*}
    &\iota_n \circ \int_0^{X_n} \bigg( \sum_{j=0}^k p_j(X_1, \dots,
    X_{n-1}) \, X_n^j \bigg)
    = \iota_n \bigg( \sum_{j=0}^k p_j(X_1, \dots, X_{n-1}) \, X_n^{j+1}/(j+1) \bigg)\\
    &\qquad = \sum_{j=0}^k p_j(x_1, \dots, x_{n-1}) \, x_n^{j+1}/(j+1)
    = \iota \circ \int_0^{X_n} \bigg( \sum_{j=0}^k
    p_j(x_1, \dots, x_{n-1}) X_n^j \bigg)\\
    &\qquad = \int^{x_n} \circ \, \iota \, \bigg( \sum_{j=0}^k
    p_j(x_1, \dots, x_{n-1}) X_n^j \bigg) = \int^{x_n} \circ \, \iota_n
    \, \bigg( \sum_{j=0}^k p_j(X_1, \dots, X_{n-1}) X_n^j \bigg)
  \end{align*}
  where the Rota-Baxter property of~$\iota$ has been employed in the
  fourth equality.

  \eqref{it:zero-col} We may assume~$M \in \matmon{K}{n}$ and~$f \in
  \galg_n$ for a sufficiently large~$n \ge i$. Setting~$\tau = (i \:
  n) \in S_n$, we see that~$M\tau$ has the last column zero so
  that~$M\tau = (M\tau) \, E_n$. Hence we have also~$\tau^* M^* \! f =
  E_n^* \, (M\tau)^* f \in \galg_{n-1}$ by the evaluation property of
  hierarchies in Definition~\ref{def:hierarchy}. Using the cyclic
  permutation
  \begin{equation*}
    \gamma = (i \mapsto i+1 \mapsto \cdots \mapsto n-1 \mapsto
    i) \in S_{n-1},
  \end{equation*}
  we see that~$\tau\gamma$ is a permutation that sends~$k$
  to~$\alpha_k$ for~$\alpha = (1, \dots, i-1, i+1, \dots, n)$, so by
  the definition of the dependency lattice we may infer~$M^* \! f \in
  \galg_\alpha \subseteq \galg_{(i)}'$ from~$(\tau\gamma)^* M^* f =
  \gamma^* \tau^* M^* f \in \galg_{n-1}$. But the latter follows
  immediately from~$\tau^* M^* \! f \in \galg_{n-1}$ since~$\gamma \in
  S_{n-1}$.

  \eqref{it:eval-van} Take~$f \in \galg_n$. If~$n \le i$ then~$f \in
  \galg_i$ and hence~$E_i^* \cum^{x_i} f = 0$ follows from
  Item~\eqref{it:ordinary-piece} of
  Definition~\ref{def:hier-Rota-Baxter}. Otherwise let~$\tau$ be the
  transposition~$(n \; i)$ so that~$E_n^* \, \cum^{x_n} \tau^* f = 0$
  by the same token. Composing this on the left by~$\tau^*$
  gives~$E_i^* \cum^{x_i} f = 0$ by Item~\eqref{it:perm-int} of this
  Lemma since~$E_i = \tau E_n \tau$.
\end{proof}

The horizontal substitution rule~\eqref{eq:transvec} can be
generalized to the following result about arbitrary
\emph{transvections}.

\begin{lemma}
  \label{lem:gen-transvec}
  Let~$(\galg_n, \cum^{x_n})_{n \in \NN}$ be a Rota-Baxter hierarchy
  over a field~$K$. If~$T = T_i(v)$ is any transvection along
  the~$x_i$-axis with~$(v_1, \dots, v_{i-1}, v_{i+1}, \dots, v_n) \in
  K_{n-1}$, then we have~$\cum^{x_i} \, T^* = (1-E_i^*) \, T^*
  \cum^{x_i}$.
\end{lemma}
\begin{proof}
  The general result follows from the case~$i=1$ by conjugation
  with~$(1 \: i) \in S_n$ and Item~\eqref{it:perm-action} of
  Definition~\ref{def:hier-Rota-Baxter}. Hence we may
  assume~$i=1$. If~$T_v$ denotes the
  transvection~\eqref{eq:gen-transvec} corresponding to~$v = (v_2,
  \dots, v_n) \in K_{n-1}$, we have the obvious relation~$T_v =
  T_{v_2e_2} \cdots T_{v_ne_n}$ for the composition, so it suffices to
  consider~$T_a := T_v$ for~$v = a e_j \; (j > 1)$. We may furthermore
  assume that~$a \neq 0$ since otherwise~$T_a = I_n$, and the result
  follows by Item~\eqref{it:eval-van} in
  Lemma~\ref{lem:simple-prop}. If~$S = \mmt{j}{a}$ denotes the
  corresponding scaling in~$x_j$-direction with inverse~$\tilde{S} =
  \mmt{j}{1/a}$, we have~$T_a = \tilde{S} T_1 S$ and hence
  \begin{align*}
    \cum^{x_1} \, T_a^* &= \cum^{x_1} S^* T_1^* \tilde{S}^* = S^* \cum^{x_1} T_1^*
    \tilde{S}^* = S^* (1-E_1^*) T_1^* \cum^{x_1} \tilde{S}^* = S^* (1-E_1^*) T_1^*
    \tilde{S}^* \cum^{x_1},\\
    &= (1 - E_1^*)\, T_a^* \, \cum^{x_1}
  \end{align*}
  where the second and fourth equality follows from Item~\eqref{it:mapping-prop}
  of Definition~\ref{def:hier-Rota-Baxter}, the third from~\eqref{eq:transvec},
  and the last from~$E_1 S = S E_1$.
\end{proof}

The vertical substitution rule can also be generalized in a way that
will become important in the next secion---allowing a coefficient
function within the integral. For handling this kind of situation we
use the technique of \emph{slack variables}. Before going through the
formal proof, it may be helpful to recall this technique from
Analysis. In fact, the verification of the vertical substitution rule
in Example~\ref{ex:classical} can be generalized as follows:
\begin{align*}
  &L_2(v')^* \cum^{x_1} g(x_1) \, L_x(e_{1}+v)^*
  \cum^{x_1} f(x_1, \dots,
  x_n)\\
  & \quad = L_2(v')^* \int_0^{x_1} g(\eta) \int_0^\eta f(\xi, x_2 +
  \eta, x_{3\dots n} + v_{3\dots n} \eta) \, d\xi \, d\eta\\
  & \quad = L_2(v')^* \int_0^{x_1} \int_{\xi+x_2}^{x_1+x_2}
  g(\bar{\eta}-x_2) \, f(\xi, \bar{\eta}, x_{3\dots n} + v_{3\dots n}
  (\bar{\eta}-x_2)) \, d\bar{\eta} \, d\xi\\
  & \quad = \int_0^{x_1} \int_{\xi+x_2}^{x_1+x_2}
  \bar{g}(x_1, \eta, x_3, \dots, x_n, x_2) \,
  f(\xi, \eta, x_{3\dots n} + v_{3\dots n} \eta) \, d\eta \,
  d\xi\\
  & \quad = \int_0^{x_1} \int_0^{x_1+x_2} \ldots \, d\eta \, d\xi -
  \int_0^{x_1} \int_0^{\xi+x_2} \ldots \, d\eta \, d\xi
\end{align*}
The auxiliary function~$\bar{g} \in \galg_{n+1}$ is defined
as~$\bar{g}(x_1, x_2, \dots, x_n, x_{n+1}) := g(x_2 - x_{n+1})$. Since
the substitution induced by~$J := I_n \oplus e_2$ acts as~$x_{n+1}
\mapsto x_2$, it is now easy to check that the first summand is given
by~$J^* L_x(e_1)^* \cum^{x_1} \cum^{x_2} \bar{g} \, L_y(v')^* f$ and
the second by~$J^* \cum^{x_1} L_x(e_1)^* \cum^{x_2} \bar{g} \,
L_y(v')^* f$. Thus we recover the same rule as before, apart from the
presence of the multipliers~$g, \bar{g}$ and the back substitution
effected by~$J^*$. The whole point is that the slack substitution~$x_2
\mapsto x_{n+1}$ allows us to temporarily ``freeze'' the
variable~$x_2$ so that it is not affected by integration. However,
there is a price to pay for this: The identity is now restricted
to~$\galg_n$ rather than being valid over all~$\galg$. This is
reflected by the change from~$f(x_1, \dots, x_n, \dots)$ in
Example~\ref{ex:classical} to~$f(x_1, \dots, x_n)$ in the above
verification. Practically speaking, this is not an essential
restriction: The slack variable~$x_{n+1}$ must be chosen large enough
to prevent any conflict with the substitutions or integrals. Let us
now prove the same result for general Rota-Baxter hierarchies.

\begin{lemma}
  \label{lem:slack}
  Let~$(\galg_n, \cum^{x_n})_{n \in \NN}$ be a Rota-Baxter hierarchy
  over a field~$K$, and let~$g \in \galg_1$ be any coefficient
  function. Then on~$\galg_n$ we have the identity
  \begin{equation}
    \label{eq:vert-subst}
    \cum^{x_j} g \, L_j(e_i+v)^* \cum^{x_j} = L_{i+1}^{-1}(v')^* (I_n \oplus e_{i+1})^*
    \smash{\big[ L_j(e_i)^*, \cum^{x_j} \big] \cum^{x_{i+1}}
      \bar{g} \, L_{i+1}(v')^*}
  \end{equation}
  for any~$0 < j \le i$ and~$v = (0, v') \in K^{i-j+1} \oplus
  K^{n-i-1}$. Here~$\bar{g} := (e_{i+1} - e_{n+1})^* g \in
  \galg_{n+1}$ is the transform of the coefficient function.
\end{lemma}
\begin{proof}
  As before it will suffice to prove the case~$i=j=1$ since the
  general case can then be recovered via conjugation by
  transpositions. Let~$u \in \galg_n$ be arbitrary but fixed. Since~$g
  \in \galg_1$, it is invariant under~$L_x(e_1+v)^*$ by the
  straightness of the action. Moroever, setting~$\tilde{g} :=
  (-e_{n+1})^* g \in \galg_{(n+1)}$ we have~$g = E_{n+1}^* L_x(-e_n)^*
  \tilde{g}$, by straightness again since the first row of the
  matrix~$\big(-e_{n+1} \oplus (0 \: I_n)\big) \, L_x(-e_n) \, E_{n+1}
  \in K^{(n+1) \times (n+1)}$ is~$e_1$. Using this factorization we
  obtain
  \begin{align*}
    \cum^x g \, L_x(e_1+v)^* \cum^x u &= \cum^x L_x(e_1+v)^*
    E_{n+1}^* L_x(-e_n)^* \tilde{g} \cum^x u\\
    &= \cum^x E_{n+1}^* \, L_x(e_1+v-e_n)^* \tilde{g} \, \cum^x u,
  \end{align*}
  where the first equality uses the multiplicativity of substitutions
  and the fact that~$\cum^x u \in \galg_n$ is invariant
  under~$E_{n+1}^*$ and~$L_x(-e_n)^*$ by straightness once again; the
  second equality follows from~$[L_x(e_1+v), E_{n+1}] = 0$
  and~\eqref{eq:lmat-asc}, where~$L_x(e_1+v) \in K^{n \times n}$ is
  embedded into~$K^{(n+1) \times (n+1)}$ via the filtration
  of~$\substmon{K}$. Since~$\tilde{g} \in \galg_{(n+1)}$, we may move it
  in and out of the inner~$\cum^x$ by Item~\eqref{it:trivial-int} of
  Lemma~\ref{lem:simple-prop}, while~$E_{n+1} \in \matfix{K}{n}$ and
  Item~\eqref{it:mapping-prop} of
  Definition~\ref{def:hier-Rota-Baxter} allows us to
  extract~$E_{n+1}^*$ from the outer~$\cum^x$. Then we
  apply~\eqref{eq:vert-subst-spc} to obtain
  \begin{align*}
    \cum^x g \, L_x(e_1+v)^* \cum^x u
    &= E_{n+1}^* \, \cum^x L_x(e_1+v-e_n)^* \cum^x \tilde{g} u\\
    &= E_{n+1}^* \, L_y^{-1}(v'-e_n)^* \big[ L_x(e_1)^*, \cum^x \big]
    \cum^y L_y(v'-e_n)^* \, \tilde{g} u,
  \end{align*}
  and we observe that~$L_y^{-1}(v'-e_n) \, E_{n+1} = (I_n \oplus e_2)
  \, L_y^{-1}(v')$ and~$L_y(v'-e_n)^* \tilde{g} u = \bar{g} \,
  L_y(v')^* u$, which may be verified by a short calculation (again
  using straightness for the latter).
\end{proof}

The technique of slack variables gives considerable power to the
notion of Rota-Baxter hierarchy. For example, choosing~$\galg =
\bigcup_{n\ge0} C^\infty(\RR_+^n)$, the \emph{convolution of
  univariate functions}~$\star\colon \galg_1 \times \galg_1 \to
\galg_1$ can be defined by
\begin{equation*}
  f \star g := (I_1 \oplus e_x)^* \cum^y (e_x - e_y)^*\! f \, e_y^* g,
\end{equation*}
which means~$(f \star g)(x) = \cum^x f(x-y) \, g(y) \, dy$, recovering
the classical definition of the Duhamel convolution. Adding
evaluations as in the ordinary
case~\cite[\S3]{RosenkranzRegensburger2008a}, and limits as
``evaluations at
infinity''~\cite[\S4]{AlbrecherConstantinescuPirsicEtAl2009}, one can
define and study integral transforms (e.g.\@ Fourier, Laplace) in this
algebraic framework. We will not pursue these topic in the present
paper. For us, the main use of slack variables is to carry around
coefficient functions within integral operators. This is a topic that
we shall now investigate in some detail.

%=====================================================================
\section{Rota-Baxter Bialgebras}
\label{sec:rota-baxter-bialgebras}
%=====================================================================

Before we can build up the operator rings for Rota-Baxter hierarchies,
we must first address the question of suitable coefficient domains. It
turns out that such domains not only have an algebra structure
(needed for composing multiplication operators) but also a
\emph{coalgebra structure} (needed for expressing basic linear
substitutions) and a scaling action (in conjunction with the coalgebra
this yields all linear substitutions).

%---------------------------------------------------------------------
\subsection{Conceptualizing Linear Substitutions via Scaled
  Bialgebras}
\label{ssec:lin-subst-bialg}
%---------------------------------------------------------------------

Again it is helpful to first look at the classical
example~$C^\infty(\RR^\infty)$. Intuitively, we would like to simplify
an integral operator like~$\cum^y f(x,y)$, acting as~$u(x,y) \mapsto
\cum^y f(x,y) u(x,y)$, by pulling out of the integral those parts
of~$f(x,y)$ that depend only on~$x$. For example, if~$f(x,y) =
(x+y)^2$ we would simplify
\begin{equation}
  \label{eq:expansion}
  \cum^y (x+y)^2 \, u(x,y) = x^2 \cum^y \, u(x,y) + 2x \, \cum^y y \,
  u(x,y) + \cum^y y^2 \, u(x,y).
\end{equation}
However, this kind of simplification is not possible for integrands
like~$f(x,y) = e^{xy}$ except if we are willing to use infinite
expansion like~$e^{xy} = \sum_k (1/k!) \, x^k y^k$. In the terminology
of integral equations, $f(x,y) = (x+y)^2$ is called a
\emph{separated}\footnote{In Analysis, the term \emph{degenerated} is
  most commonly used. From the viewpoint of Algebra, however, this
  term sounds too drastic so that we prefer the expression
  \emph{separated}.} kernel, as opposed to the non-separated
kernel~$f(x,y) = e^{xy}$. Since we would like to refrain from using
infinite sums (and hence topology), we will only allow separated
kernels as coefficients in this paper. This will be made precise in
Definition~\ref{def:separated-algebra} below.

The expansion step~\eqref{eq:expansion} can be understood as the
substitution~$x \mapsto x+y$ on~$g(x) = x^2$ to yield~$f(x,y) = x^2 +
2xy + y^2$. In other words, we have used the
\emph{coproduct}~$\Delta\colon K[x] \to K[x] \otimes K[x] \cong
K[x,y]$ defined by~$\Delta(x) = x \otimes 1 + 1 \otimes x$ so that~$f
= \Delta(g)$. Note that~$\Delta$ is an algebra homomorphism; in fact,
$K[x]$ has the structure of a bialgebra. For general properties of
bialgebras we refer to~\cite{Cartier2007,Downie2012,Sweedler1969} and
to~\cite[\S2]{Guo2012}. As we will make precise later
(Definition~\ref{def:Rota-Baxter-bialgebra}), the coproduct interacts
nicely with the Rota-Baxter structure.

Before we look at this interaction in more detail, it is apposite to
focus first on the substitution structure of those bialgebras that
provide ``separated kernels'' like the paradigmatic example~$K[x,y]$
above. It turns out that we can build a $K$-hierarchy from a given
bialgebra~$\balg$ from just one more ingredient, which we call
\emph{scaling}: an action of the ground field~$K$ that will be
extended to an action of the full matrix monoid~$\matma$ on the
tensor algebra over~$\balg$. In the classical example, this is the
action~$f(x) \mapsto f(\lambda x)$ for a function~$f \in
C^\infty(\RR)$ and a scalar~$\lambda \in \nonzero{\RR}$.

We formulate the basic properties of such an action in terms of the
\emph{convolution product} that we denote here by~$\circplus$. As
stated in the Introduction, all bialgebras are assumed to be
commutative and cocommutative. Recall~\cite[Thm.~2.3.4]{Guo2012} that
for a $K$-bialgebra~$\balg$ with product~$\nabla$, unit~$1$,
coproduct~$\Delta$ and counit~$\counit$ one defines the associative
and commutative operation~$\circplus$ on vector space endomorphism by
setting~$f \circplus g = \nabla \, (f \otimes g) \, \Delta$. If~$f$
and~$g$ are bialgebra endomorphisms, then both~$f \circplus g$ and~$f
\circ g$ are. Writing~$\BialgHom_K(\balg)$ for the set of bialgebra
$K$-endomorphisms, this yields two operations
\begin{equation*}
  \circplus, \circ\colon \BialgHom_K(\balg) \times \BialgHom_K(\balg)
  \to \BialgHom_K(\balg),
\end{equation*}
which are clearly also~$K$-linear. Moreover, one checks that~$(f
\circplus g) \circ h = (f \circ h) \circplus (g \circ h)$ and~$h \circ
(f \circplus g) = (h \circ f) \circplus (h \circ
g)$. Hence~$(\BialgHom_K(\balg), \circplus, \circ)$ is a commutative
unital semiring. The neutral element with respect to~$\circplus$ is
the composite~$\evl := 1 \circ \counit$ of the unit~$1\colon K \to
\balg$ and the counit~$\counit\colon \balg \to K$. Of course,
$\id_\balg\colon \balg \to \balg$ is the neutral element with respect
to~$\circ$.

Note that a (multiplicative) group action~$\nonzero{K} \times \balg
\to \balg$ is by definition a monoid homomorphism~$(\nonzero{K},
\cdot, 1) \to (\AlgHom_K(\balg), \circ, \id)$, which we shall
write~$\lambda \mapsto \lambda^*$. Clearly, we may extend it to a
monoid homomorphism~$K \to \AlgHom_K(\balg)$ by setting~$0^* :=
\evl$. If the~$\lambda^*$ are bialgebra homomorphisms, we refer
to~$\nonzero{K} \times \balg \to \BialgHom_K(\balg)$ as a \emph{group
  biaction}. For a scaling, we want this homomorphism to respect the
convolution product.

\begin{definition}
  \label{def:scaled-bialgebra}
  Let~$(\balg, \nabla, 1, \Delta, \counit)$ be a $K$-bialgebra. Then a
  group biaction~$\nonzero{K} \times \balg \to \balg$ is called a
  \emph{scaling} if the map~$(K, +, \cdot) \to (\BialgHom_K(\balg),
  \circplus, \circ)$, given by~$\lambda \mapsto \lambda^*$ is a
  semiring homomorphism. In this case, we call~$\balg$ a \emph{scaled
    bialgebra}.
\end{definition}

In the sequel, we shall suppress the notation~$\circ$ for the composition of
endomorphisms. From the definition, we have the additive law~$(\lambda_1 +
\lambda_2)^* = \lambda_1^* \circplus \lambda_2^* = \nabla (\lambda_1^* \otimes
\lambda_2^*) \Delta$. By induction, this generalizes immediately to
\begin{equation}
  \label{eq:additive-law}
  (\lambda_1 + \cdots + \lambda_n)^* = \nabla^n \, (\lambda_1^*
  \otimes \cdots \otimes \lambda_n^*) \, \Delta^n
\end{equation}
for all~$n > 0$, where~$\nabla^n$ and~$\Delta^n$ denote the obvious
iterations (see below for the definition). Observe that the image~$K^*
\subseteq \BialgHom_K(\balg)$ of the scaling homomorphism~$\lambda
\mapsto \lambda^*$ is automatically a \emph{field}. Hence we have a
field homomorphism~$K \to K^*$, and the group action~$\nonzero{K}
\times \balg \to \balg$ is faithful. Furthermore, note that~$\balg$ is
a \emph{Hopf algebra} with the antipode~$(-1)^*\colon \balg \to \balg$
since this is clearly the convolution inverse of~$\id_\balg = 1^*$.

We write~$\indhier{\balg} = \oplus_{i\ge 0} \balg_i$ for the tensor
algebra with grades~$\balg_0 := K$ and $\balg_i = \balg^{\otimes i}$
for $i>0$. The product on~$\balg$ is denoted by~$\nabla\colon \balg
\otimes \balg \to \balg$, its iterations by~$\nabla^n\colon \balg_n
\to \balg$. Likewise, we write~$\Delta^n\colon \balg \to \balg_n$ for
the \emph{iterated coproduct} defined by~$\Delta^1 = \id_{\balg}$ and
$\Delta^{n+1} = (\Delta^n \otimes \id_{\balg}) \, \Delta$. Iterating
coassociativity one obtains
\begin{equation}
  \label{eq:it-coassoc}
  (\Delta^{m_1} \otimes \cdots \otimes \Delta^{m_k}) \, \Delta^k
  = \Delta^{m_1 + \cdots + m_k}
\end{equation}
for all~$m_1, \dots, m_k \ge 0$ as
in~\cite[Lem.~1.1.9(2)]{Manetti2012} with slightly different
notation. Setting~$m_1 = 2, m_2 = \cdots = m_n = 1$ yields the
alternative recursion~$\Delta^{n+1} = (\Delta \otimes \id^{n-1}) \,
\Delta^n$ for computing the iterated coproduct. The counit axioms
imply
\begin{equation}
  \label{eq:it-counit}
  (\evl^{\otimes (i-1)} \otimes \id \otimes \evl^{\otimes (n-i)}) \,
  \Delta^n = 1^{\otimes (i-1)} \otimes \id \otimes 1^{\otimes (n-i)},
\end{equation}
where the right-hand side is the embedding~$\balg \to \balg_n$ defined
by~$f \mapsto 1^{\otimes (i-1)} \otimes f \otimes 1^{\otimes (n-i)}$.
One checks also the iterated scaling commutations~$\lambda^* \nabla^n
= \nabla^n (\lambda^* \otimes \cdots \otimes \lambda^*)$ for the
product and~$\Delta^n \lambda^* = (\lambda^* \otimes \cdots \otimes
\lambda^*) \Delta^n$ for the coproduct.

The \emph{induced maps} on the $n$-fold tensor products are
then~$\nabla^{n \otimes n} = (\nabla^n)^{\otimes n}\colon
\balg_n^{\otimes n} \to \balg_n$ and~$\Delta^{n \otimes n} =
(\Delta^n)^{\otimes n}\colon \balg_n \to \balg_n^{\otimes n}$. Note
that~$\balg_n$ is a bialgebra with (iterated) product $\nabla^{n
  \otimes n} \tau$ and (iterated) coproduct $\tau \Delta^{n \otimes
  n}$, where $\tau:=\tau_n\colon \balg_n^{\otimes n} \to
\balg_n^{\otimes n}$ is the transposition
\begin{equation*}
  (f_{11} \otimes
  \cdots \otimes f_{1n}) \otimes \cdots \otimes (f_{n1} \otimes \cdots
  \otimes f_{nn}) \mapsto (f_{11} \otimes \cdots \otimes f_{n1}) \otimes
  \cdots \otimes (f_{1n} \otimes \cdots \otimes f_{nn}).
\end{equation*}
If~$a = (a_1, \dots, a_n)$ is a column vector in~$K^n$ or a row vector
in~$K_n$ we write~$a^\otimes = a_1^* \otimes \cdots \otimes a_n^*$ for
the corresponding scaling map on~$\balg_n$. Likewise, for a matrix~$A
\in \matmon{K}{n}$ we write $A^\otimes = A_{1\bullet}^\otimes \otimes
\cdots \otimes A_{n\bullet}^\otimes$ for the scaling map
on~$\balg_n^{\otimes n}$. If~$\tilde{A}$ denotes the transpose matrix,
we have~$\tau A^\otimes = \tilde{A}^\otimes \tau$ or, in other words,
$\tau (A_{1\bullet}^\otimes \otimes \cdots \otimes
A_{n\bullet}^\otimes) = (A_{\bullet 1}^\otimes \otimes \cdots \otimes
A_{\bullet n}^\otimes) \tau$.

% By an easy induction proof one also obtains from the bialgebra axioms the
% general product-coproduct compatibility relation
% %
% \begin{equation}
%   \label{eq:prod-coprod}
%   \Delta^n \nabla^n = \nabla^{n \times n} \, \tau \, \Delta^{n \times n},
% \end{equation}
%

We can now define the matrix action in terms of the coproduct and the scaling
action. Note the appearance of a \emph{single}~$\tau$ in the expression
for~$M^*$ below: It allows us to interpret the left expression~$\nabla^{n
  \otimes n} \tau$ as the product map on~$\balg_n$ \emph{or}, via~$\tau
M^\otimes = \tilde{M} \tau$, the right expression~$\tau \Delta^{n \otimes n}$
on~$\balg_n$ as the corresponding coproduct map.

% This might be related to the \emph{contravariance} of the resulting
% monoid action.)

\begin{definition}
  \label{def:induced-hierarchy}
  Let~$(\balg, \nabla, 1, \Delta, \counit)$ be a scaled bialgebra.
  Then the \emph{induced matrix action} $\matmon{K}{n} \times \balg_n
  \to \balg_n$, written $(M, f) \mapsto M^* f$, is definded by~$M^* =
  \nabla^{n \otimes n} \tau M^{\otimes} \Delta^{n\otimes n}$. The
  algebra~$\indhier{\balg}$ together with the corresponding
  action~$\matma \times \indhier{\balg} \to \indhier{\balg}$ is
  called the \emph{induced hierarchy} for~$\balg$.
\end{definition}

If we define the action~$a^*\colon \balg \to \balg_n$ of a row~$a =
(a_1, \dots, a_n) \in \matspc{K}{}{n}$ by~$a^* = a^\otimes \Delta^n$, the
finite matrix action~$\matmon{K}{n} \times \balg_n \to \balg_n$ can be
written as
\begin{equation}
  \label{eq:ind-mat-action}
  M^* (f_1 \otimes \cdots \otimes f_n) = (M_{1\bullet}^* f_1) 
  \cdots  (M_{n\bullet}^* f_n),
\end{equation}
where juxtaposition on the right-hand side denotes the iterated
product in~$\balg_n$. The extension of the finite matrix action
to~$\matma \times \indhier{\balg} \to \indhier{\balg}$ is to be
understood as in Section~\ref{sec:hier-intdiffalg}: Every matrix
of~$\matma$ is of the block form~$\smallmat{A}{0}{0}{I}$ with~$A \in
\matmon{K}{r}$ and~$I$ the~$\infty \times \infty$ identity matrix, while
every~$f \in \indhier{\balg}$ is a finite sum of tensors of the
form~$f_1 \otimes \cdots \otimes f_s \otimes 1 \otimes 1 \otimes
\cdots$, so it suffices to choose~$n$ as the maximum of those~$r$
and~$s$ and then use the action~$\matmon{K}{n} \times \balg_n \to
\balg_n$. Let us now make sure that the induced hierarchy deserves its
name.

\begin{proposition}
  \label{prop:induced-hierarchy}
  Let~$(\balg, \Delta, \counit)$ be a scaled bialgebra. Then the
  induced hierarchy~$\indhier{\balg}$ is a $K$-hierarchy in the sense
  of Definition~\ref{def:hierarchy}.
\end{proposition}
\begin{proof}
  The tensor algebra~$\indhier{\balg}$ is clearly an ascending
  algebra~$(\balg_n)$ with~$\balg_n = \balg^{\otimes n}$ and direct
  limit~$\indhier{\balg}$. Moreover, it is clear from the definition
  of the induced matrix action that~$M^*(\balg_n) \subseteq \balg_n$
  for all~$M \in \matmon{K}{n}$.

  Let us show that~$E_n^*(\balg_n) \subseteq \balg_{n-1}$.  We
  write~$\iota_j := \evl \otimes \cdots \otimes \id_\balg \otimes
  \cdots \otimes \evl\colon \balg_n \to \balg_n$ with all
  entries~$\evl$ except for the~$j$-th, which
  is~$\id_\balg$. Likewise, we write~$\eta := \evl \otimes \cdots
  \otimes \evl\colon \balg_n \to \balg_n$ for the tensor map having
  only~$\evl$. Then~$E_n^{\otimes} = \iota_1 \otimes \cdots \otimes
  \iota_{n-1} \otimes \eta$, and~\eqref{eq:ind-mat-action} yields
  \begin{equation*}
    E_n^* (f_1 \otimes \cdots \otimes f_{n-1} \otimes f_n) = (\iota_1
    \Delta^n f_1) \cdots (\iota_{n-1} \Delta^n f_{n-1}) (\eta \, \Delta^n
    f_n).
  \end{equation*}
  Since~$\evl\colon \balg \to K \subseteq \balg$, we have~$\eta \,
  \Delta^n f_n \in K^{\otimes n} \cong K$, so it suffices to
  show~$\iota_j \Delta^n f_j \in \balg_{n-1}$ for all~$j <
  n$. Writing~$\tilde{\iota}_j\colon \balg_{n-1} \to \balg_{n-1}$ for
  the corresponding tensor maps with~$\evl$ everywhere except for
  the~$j$-th entry~$\id_{\balg}$, we have~$\iota_j = \tilde{\iota}_j
  \otimes \evl$ and therefore
  \begin{equation*}
    \iota_j \Delta^n = (\tilde{\iota}_j \otimes \evl)
    (\Delta^{n-1} \otimes \id_{\balg}) \Delta = (\tilde{\iota}_j
    \Delta^{n-1} \otimes \id_\balg) (\id_\balg \otimes \evl)
    \Delta = \tilde{\iota}_j \Delta^{n-1} \otimes 1,
  \end{equation*}
  where the last identity uses the defining property of the counit
  in~$\balg$. But this clearly implies that~$\iota_j \Delta^n f_j =
  \tilde{\iota}_j \Delta^{n-1} f_j \in \balg_{n-1}$ for all~$j < n$,
  as was required.

  Next we prove that~$\matmon{K}{n} \times \balg_n \to \balg_n$ is a
  contravariant monoid action. For any matrices~$M, \tilde{M} \in
  \matmon{K}{n}$ and any function~$f := f_1 \otimes \cdots \otimes f_n
  \in \balg_n$ we must show
  \begin{equation*}
    (M \mkern-1.5mu\tilde{M})^* (f_1 \otimes \cdots \otimes f_n) = 
    \tilde{M}^* M^* (f_1 \otimes \cdots \otimes f_n).
  \end{equation*}
  We start by computing~$M^* (f_1 \otimes \cdots \otimes f_n)$. Using sumless
  Sweedler notation, we have
  \begin{equation*}
    \Delta^{n \otimes n} f = (f_{1, (1)} \otimes
    \cdots \otimes f_{1,(n)}) \otimes \cdots \otimes (f_{n,(1)} \otimes \cdots
    \otimes f_{n,(n)}) = \sbigotimes_i \sbigotimes_j f_{i,(j)},
  \end{equation*}
  where in the last expression (as in the rest of this proof) all indices range
  over~$\{1, \dots, n\}$. We have then~$M^\otimes \Delta^{n \otimes n} f =
  \bigotimes_i \bigotimes_j M_{ij}^* \, f_{i,(j)}$ and further
  \begin{equation}
    \label{eq:matrix-action}
    M^* f = \sbigotimes_j \sprod_i M_{ij}^* \, f_{i,(j)},
  \end{equation}
  with~$\prod$ denoting the iterated product of~$\balg$. Using the fact
  that~$\Delta^n$ is a bialgebra morphism, this implies~$\Delta^{n \otimes n}
  M^* f = \bigotimes_j \prod_i (M_{ij}^*)^{\otimes n} \Delta^n f_{i,(j)}$. Now
  observe that the~$\tilde{M}_{1\bullet}^\otimes, \dots,
  \tilde{M}_{n\bullet}^\otimes$ are algebra morphisms so that
  \begin{align*}
    \tilde{M}^\otimes \Delta^{n \otimes n} M^* f &= \sbigotimes_j \sprod_i \bigg(
    \!  \sbigotimes_k \, (M_{ij} \tilde{M}_{jk})^* \!\bigg) \bigg( \! \Delta^n
    f_{i,(j)} \! \bigg) = \sprod_i \sbigotimes_j \bigg( \!  \sbigotimes_k \,
    (M_{ij} \tilde{M}_{jk})^*
    \!\bigg) \bigg( \! \sbigotimes_k f_{i,(j)(k)} \! \bigg),\\
    &= \sprod_i \bigg( \! \sbigotimes_j \sbigotimes_k \, (M_{ij} \tilde{M}_{jk})^*
    \!\bigg) \bigg( \! \sbigotimes_j \sbigotimes_k f_{i,(j)(k)} \! \bigg)
  \end{align*}
  where the second step uses the definition of the product on~$\balg_n^{\otimes
    n}$. Applying the trans\-position~$\tau$ to this equation will swap~$j
  \leftrightarrow k$ in the double tensor products. However, note
  that~$\smash{\bigotimes_j \bigotimes_k} \, f_{i,(j)(k)} = \Delta^{n \otimes n}
  \Delta^n f_i = \smash{\Delta^{n^2}} \! f_i$
  by~\cite[Lem.~1.1.9(2)]{Manetti2012}. Since~$\balg$ is cocommutative, we
  have~$\Delta^N = \pi \, \Delta^N$ for any~$N > 0$ and any
  permutation~$\pi\colon \balg_N \to \balg_N$ of the tensor factors. In the
  special case of~$\pi = \tau$ on~$\balg_{n^2} \cong \balg_n^{\otimes n}$, this
  yields~$\tau \, \Delta^{n \otimes n} \Delta^n = \Delta^{n \otimes n}
  \Delta^n$, so we can keep the argument~$\Delta^{n \otimes n} \Delta^n
  f_i$. Since~$\tau$ is furthermore a morphism of algebras, $\tilde{M}^* \! M^*
  f$ comes to
  \begin{align*}
    \sprod_i \nabla^{n \otimes n} & \bigg( \! \sbigotimes_k \sbigotimes_j \,
    (M_{ij} \tilde{M}_{jk})^* \!\bigg) \Big( \Delta^{n \otimes n} \Delta^n f_i
    \Big) = \sprod_i \Bigg( \! \sbigotimes_k \nabla^n \bigg( \! \sbigotimes_j \,
    (M_{ij} \tilde{M}_{jk})^* \! \bigg) \Delta^n \! \Bigg) \Delta^n f_i,\\
    & \qquad = \sprod_i \Bigg( \! \sbigotimes_k \bigg( \! \sum_j M_{ij}
    \tilde{M}_{jk} \! {\bigg)\!\!}^* \, \Bigg) \Delta^n f_i = \sprod_i
    \sbigotimes_k \, (M \!  \tilde{M})_{ik}^* \, f_{i,(k)},
  \end{align*}
  where we have used~\eqref{eq:additive-law} in the last but one step. Comparing
  this with~\eqref{eq:matrix-action} and using again the definition of the
  product on~$\balg_n^{\otimes n}$, the claim follows.
\end{proof}

Let us give some \emph{important examples} of scaled bialgebras, which will also
turn out to be admissible coefficient algebras for building a suitable ring of
Rota-Baxter operators.

\begin{example}
  \label{ex:scaled-bialg}
  The prototypical example of a scaled bialgebra is the \emph{polynomial
    ring}~$K[x]$. Its coproduct is given by~$\Delta(x) = x \otimes 1 + 1 \otimes
  x$, which implies~$\Delta(x^n) = \sum_{k=0}^n \binom{n}{k} \, x^k \otimes
  x^{n-k}$ on the canonical $K$-basis of~$K[x]$. In other words, we have~$\Delta
  p(x) = p(x+y)$ under the natural isomorphism~$K[x] \otimes K[x] \cong
  K[x,y]$. This is a well-known example of a Hopf algebra, sometimes also called
  the binomial bialgebra~\cite[Ex.~2.2.3.2]{Guo2012}. We use the scaling
  action~$\nonzero{K} \times K[x] \to K[x]$ given by the
  substitutions~$\lambda^* f(x) = f(\lambda x)$. It remains to check that the
  action~$\lambda \mapsto \lambda^*$ respects the convolution product. Indeed,
  on the $K$-basis element~$x^m$ we have
  \begin{align*}
    (\lambda^* \circplus \mu^*) (x^m) &= \nabla \, (\lambda^* \otimes \mu^*) \,
    \Delta (x^m) = \sum_{k=0}^m \binom{m}{k} \, \nabla (\lambda^* \otimes \mu^*)
    (x^k \otimes x^{m-k})\\[-0.25ex]
    &= \sum_{k=0}^m \binom{m}{k} \, \nabla \big( (\lambda x)^k \otimes (\mu
    x)^{m-k} \big) = \sum_{k=0}^m \binom{m}{k} \, (\lambda x)^k (\mu x)^{m-k}
    = (\lambda x + \mu x)^m\\[0.75ex]
    &= (\lambda+\mu)^* x^m,
  \end{align*}
  as is required for a scaling. Under the aforementioned
  isomorphism~$K[x]^{\otimes n} \cong K[x_1, \dots, x_n]$,
  Definition~\ref{def:induced-hierarchy} yields now the induced matrix
  action~$\matmon{K}{n} \times K[x_1, \dots, x_n] \to K[x_1, \dots,
  x_n]$ defined via~\eqref{eq:ind-mat-action} by~$M^*(x_1^{k_1} \cdots
  x_n^{k_n}) = M_{1\bullet}^*(x^{k_1}) \cdots M_{n\bullet}^*(x^{k_n})$
  for the matrix~$M \in \matmon{K}{n}$ and exponent vector~$(k_1, \dots,
  k_n) \in \NN^n$. Since~$\Delta^n f(x) = f(x_1 + \cdots + x_n)$, the
  row action
  \begin{align*}
    M_{i\bullet}^* x^{k_i} &= (M_{i1}^* \otimes \cdots \otimes M_{in}^*) \, (x_1
    + \cdots + x_n)^{k_i}\\
    &= \sum_{l_1 + \cdots + l_n = k_i} \binom{k_i}{l_1, \dots, l_n} \, (M_{i1}^*
    \otimes \cdots \otimes M_{in}^*) \, (x_1^{l_1} \cdots x_n^{l_n})\\
    &= \sum_{l_1 + \cdots + l_n = k_i} \binom{k_i}{l_1, \dots, l_n} \,
    (M_{i1} x_1)^{l_1} \cdots (M_{in} x_n)^{l_n} =
    (M_{i1} x_1 + \cdots M_{in} x_n)^{k_i},
  \end{align*}  
  induces the overall action~$M^* (x_1^{k_1} \cdots x_n^{k_n}) =
  \prod_i \big( \sum_j M_{ij} x_j \big)^{k_i}$ with the intended effect
  of a linear substitution~$x_i \mapsto \sum_j M_{ij} x_j$. In other
  words, the column~$(x_1, \dots, x_n)$ gets multiplied on the left by
  the matrix~$M \in \matmon{K}{n}$.

  Similar considerations apply to the larger ring of \emph{exponential
    polynomials}~$K[x,e^{Kx}]$ and its variants (e.g.\@ restricting
  the exponents to submonoids of~$K$, like replacing~$e^{Kx}$ by
  $e^{\NN x}$). Since the coproduct is an algebra morphism, it is
  sufficient to define it on the algebra generators~$x$
  and~$e^{Kx}$. Clearly, $K[x]$ should be a sub-bialgebra, so~$\Delta$
  coincides on~$x$. For the exponential generators, one puts~$\Delta
  e^{\alpha x} = e^{\alpha x} \otimes e^{\alpha x}$; so these are
  group-like elements unlike the primitive element~$x$. Under the
  isomorphism~$K[x, e^{Kx}]^{\otimes n} \cong K[x_1, e^{Kx_1}, \dots,
  x_n, e^{Kx_n}]$ we have again~$\Delta^n f(x) = f(x_1, \dots, x_n)$
  for any~$f \in K[x,e^{Kx}]$. Moroever, it is easy to check
  that~$(\lambda+\mu)^* = \lambda^* \circplus \mu^*$, so we have a
  scaled bialgebra. As in the case of the polynomials, one sees
  that~$M^*$ acts via the linear substitution~$x_i \mapsto \sum_j
  M_{ij} x_j$.
\end{example}

%---------------------------------------------------------------------
\subsection{Integration in Scaled Bialgebras}
\label{ssec:int-bialg}
%---------------------------------------------------------------------

Let us now turn to the interaction between the substitution structure
and \emph{Rota-Baxter operators}. Here and henceforth we shall
suppress the unit~$1$ and the product~$\nabla$ of a bialgebra~$(\balg,
\nabla, 1, \Delta, \counit)$.

\begin{definition}
  \label{def:Rota-Baxter-bialgebra}
  We call~$(\balg, \Delta, \counit, \cum)$ a \emph{Rota-Baxter
    bialgebra}\footnote{Note that this notion is distinct from the
    Rota-Baxter coalgebra in sense of~\cite{JianZhang2014}. There may
    be interesting relations between the two concepts but this
    investigation will have to wait for future work. At this point,
    let us just mention that the standard integral on polynomials is a
    Rota-Baxter bialgebra (see Example~\ref{ex:Rota-Baxter-bialgebra}
    below) but not a Rota-Baxter coalgebra in the sense
    of~\cite{JianZhang2014}. Note also that the
    axiom~\eqref{eq:hor-subst-rule} involves the counit, unlike the
    axiom in~\cite{JianZhang2014}.}  over~$K$ if~$(\balg, \cum)$ is an
  ordinary Rota-Baxter algebra over~$K$ and~$(\balg, \Delta, \counit)$
  is a bialgebra such that~$\evl := 1 \circ \counit$ is the projector
  associated to~$\im \cum \dotplus K = \balg$ and the \emph{horizontal
    substitution rule} in the form
  \begin{equation}
    \label{eq:hor-subst-rule}
    \Delta \cum = (\cum \otimes \id) \, \Delta + (\evl \otimes \id) \,
    \Delta \cum
  \end{equation}
  is satisfied. We call~$(\balg, \Delta, \counit, \cum)$ a
  \emph{scaled Rota-Baxter bialgebra} if it is further endowed with a
  scaling~$\nonzero{K} \times \balg \to \galg$, $\lambda \mapsto
  \lambda^*$ subject to the diagonal substitution rule~$\cum \lambda^*
  = \lambda^{-1} \lambda^* \cum$.
\end{definition}

The condition that~$\evl$ be the \emph{projector} associated to~$\im
\cum \dotplus K = \balg$ may also be expressed by~$1 \circ \counit
\circ \cum = 0$ and~$1 \circ \counit \circ 1 = 1$, and this implies
the direct sum. Hence we may also define a Rota-Baxter bialgebra as a
a bialgebra with a Rota-Baxter operator that satisfies these two
conditions along with~\eqref{eq:hor-subst-rule}.

As remarked above, a scaled bialgebra has the antipode~$S:=(-1)^*$,
which in the case of a scaled Rota-Baxter algebra satisfies~$\cum S +
S \cum = 0$ by the diagonal substitution rule. This betrays the
\emph{oriented} nature of this notion of integration: In typical
cases, like those described in Example~\ref{ex:scaled-bialg}, one
has~$S\colon f(x) \mapsto f(-x)$, so the integral picks up a sign
under reflection (see also the remark at the end of
Example~\ref{ex:classical}).

Since the bialgebra~$(\balg, \Delta, \counit)$ is cocommutative by
hypothesis, the horizontal substitution rule is also valid in its
\emph{symmetric variant}~$\Delta \cum = (\id \otimes \cum) \, \Delta +
(\id \otimes \evl) \, \Delta \cum$. Using~\eqref{eq:it-counit} we can
write both versions in the form
\begin{equation}
  \label{eq:hor-subst-rule-alt}
  \Delta \cum f = \cum^{x\!} \, \Delta f + (\cum f)(y) =
  \cum^{y\!} \, \Delta f + (\cum f)(x)
\end{equation}
for functions~$f \in \balg_1 \subset \indhier{\balg}$. Here we have
used the notation~$h(x) = h \otimes 1$ and~$h(y) = 1 \otimes h$ for
the two embeddings~$\balg \hookrightarrow \balg_2$. Moreover, the
horizontal substitution rule~\eqref{eq:hor-subst-rule} can be iterated
as follows.

\begin{lemma}
  If~$(\balg, \Delta, \counit, \cum)$ is a Rota-Baxter bialgebra
  over~$K$, we have
  \begin{equation}
    \label{eq:hor-subst-rule-it}
    \Delta^n \cum = (\cum \otimes \id^{\otimes (n-1)}) \, \Delta^n + (\evl
    \otimes \id^{\otimes (n-1)}) \, \Delta^n \cum
  \end{equation}
  for any~$n>0$. Using the operations of~$\indhier{\balg}$, this may
  be written as~$\Delta^n \cum = \cum^x \Delta^n + \evl_x \,
  \Delta^n$.
\end{lemma}
\begin{proof}
  Note that the case~$n=1$ is trivial since~$\Delta^1 = \id$ and~$\evl
  \cum = 0$. Hence we use induction over~$n$ with the base case~$n=2$
  for~$\Delta^2 = \Delta$, given by the
  hypothesis~\eqref{eq:hor-subst-rule}. Hence
  assume~\eqref{eq:hor-subst-rule-it} for a fixed~$n \ge 2$; we show
  it for~$n+1$. Using the definition of~$\Delta^{n+1}$, the base case
  and then the induction hypothesis yields
  \begin{align*}
    \Delta^{n+1} \cum = (\Delta^n \otimes \id) & (\cum \otimes \id) \,
    \Delta + (\Delta^n \otimes \id) (\evl \otimes \id) \, \Delta \cum
    = \Big( \big( (\cum \otimes \id^{\otimes (n-1)}) \, \Delta^n
    \big) \otimes \id \Big) \, \Delta\\
    & + \Big( \big( (\evl \otimes
    \id^{\otimes (n-1)}) \, \Delta^n \cum \big) \otimes \id \Big)
    \Delta + (\Delta^n \otimes \id) (\evl \otimes \id) \, \Delta \cum.
  \end{align*}
  Using~\eqref{eq:it-counit} and the alternative recursion
  for~$\Delta^n$ gives~$(\evl \otimes \id^{\otimes (n-1)}) \, \Delta^n
  = (1 \otimes \id^{\otimes (n-1)}) \, \Delta^{n-1}$ and~$(\Delta^n
  \otimes \id) (\evl \otimes \id) \, \Delta \cum = 1^{\otimes n}
  \otimes \cum$. Combining this with the definition of~$\Delta^{n+1}$
  in the first summand we get
  \begin{equation*}
    \Delta^{n+1} \cum = (\cum \otimes \id^{\otimes n}) \, \Delta^{n+1}
    + \Big( \big( (1 \otimes \id^{\otimes (n-1)}) \, \Delta^{n-1} \cum
    \big) \otimes \id \Big) \Delta + 1^{\otimes n} \otimes \cum.
  \end{equation*}
  Using~\eqref{eq:hor-subst-rule}, its middle summand is
  \begin{align*}
    \Big( \big( & (1 \otimes \id^{\otimes (n-1)}) \, \Delta^{n-1}
    \big) \otimes \id \Big) (\cum \otimes \id) \, \Delta = \Big( \big(
    (1 \otimes \id^{\otimes (n-1)}) \, \Delta^{n-1} \big) \otimes \id
    \Big) (\Delta \cum - (\evl \otimes \id) \,
    \Delta \cum)\\
    &= 1 \otimes \Delta^n \cum - \Big( \big( (1 \otimes \id^{\otimes
      (n-1)}) \, \Delta^{n-1} \evl \big) \otimes \id \Big) \Delta
    \cum = 1 \otimes \Delta^n \cum - 1^{\otimes n} \otimes \cum,
  \end{align*}
  where in the last step we have used~$\Delta^{n-1} \evl =
  \evl^{\otimes (n-1)}$ and~\eqref{eq:it-counit}. Substituting this
  into~$\Delta^{n+1} \cum$, it remains only to check that~$(\evl
  \otimes \id^{\otimes n}) \, \Delta^{n+1} \cum = 1 \otimes \Delta^n
  \cum$. Applying once again~\eqref{eq:it-counit},
  and~\eqref{eq:it-coassoc} with~$m_1=2, m_2=n-1$, one sees
  immediately that~$(\evl \otimes \id^{\otimes n}) \, \Delta^{n+1} = 1
  \otimes \Delta^n$.
\end{proof}

\begin{theorem}
  \label{thm:induced-hierarchy}
  Let~$(\balg, \Delta, \counit, \cum)$ be a scaled Rota-Baxter
  bialgebra. Then we obtain a Rota-Baxer hierarchy~$(\indhier{\balg},
  \cum^{x_n})_{n \in \NN}$ with the operators~$\cum^{x_n} =
  \id^{\otimes (n-1)} \otimes \cum$.
\end{theorem}

\begin{proof}
  It is easy to see that the~$\cum^{x_n}$ are commuting Rota-Baxter
  operators on~$\indhier{\balg}$. Hence let us now check
  conditions~\eqref{it:mapping-prop} to~\eqref{it:sub-rule} of
  Definition~\ref{def:hier-Rota-Baxter}:

  \eqref{it:mapping-prop} Let~$n \le m$. From the definition of the
  Rota-Baxter operators, $\cum^{x_n} \balg^{\otimes m} \subseteq
  \balg^{\otimes m}$ holds. For showing~$\cum^{x_n} \matfix{K}{m}^* =
  \matfix{K}{m}^* \cum^{x_n}$, choose any
  \begin{equation*}
    \tilde{M} = \begin{pmatrix} I_m & 0\\ 0 & M \end{pmatrix}
    \qquad \text{with $M \in \matmon{K}{r}$}
  \end{equation*}
  and~$f \in \indhier{\balg}$. We may choose~$k \ge m+r$ such that~$f
  \in \balg^{\otimes k}$ with~$f=f_1 \otimes \cdots \otimes
  f_k$. By~\eqref{eq:ind-mat-action} we have
  \begin{equation*}
    \tilde{M}^* f = (e_1^* f_1) \cdots (e_m^* f_m) \,
    (\tilde{M}_{m+1,\bullet}^* \, f_{m+1}) \cdots
    (\tilde{M}_{m+r,\bullet}^* \, f_{m+r})
  \end{equation*}
  with the rows~$\tilde{M}_{m+1,\bullet} = (0 \; M_{1\bullet}), \dots,
  \tilde{M}_{m+r,\bullet} = (0 \; M_{r\bullet})$ each having~$m$
  leading zeroes. From~\eqref{eq:it-counit} it follows that~$(e_1^*
  f_1) \cdots (e_m^* f_m) = f_1 \otimes \cdots \otimes f_m \in
  \balg^{\otimes k}$. Unfolding the definition of~$\tilde{M}_{m+j}^*
  \; (j = 1, \dots, r)$ and using
  coassociativity~\eqref{eq:it-coassoc} we have
  \begin{align*}
    \tilde{M}_{m+j}^* f_{m+j} &= (\evl^{\otimes m} \otimes
    M_{j\bullet}^\otimes) \, \Delta^{m+r} f_{m+j} = (\id^{\otimes m}
    \otimes M_{j\bullet}^\otimes) \, (\evl^{\otimes m} \Delta^m
    \otimes \Delta^r) \, \Delta f_{m+j}\\
    & = \evl(f_{m+j,(1)}) \, 1^{\otimes m} \otimes M_{j\bullet}^*
    f_{m+j,(2)},
  \end{align*}
  where the last step uses the relation~$\evl^{\otimes m} \Delta^m g =
  \evl(g) \in K \subseteq \balg_m$ for~$g \in \balg$, which follows
  from~\eqref{eq:it-counit}. Altogether we obtain
  \begin{equation*}
    \tilde{M}^* f = \Big( \sprod_{j=1}^r \evl(f_{m+j,(1)}) \Big)
    f_1 \otimes \cdots \otimes f_m \otimes \sbigotimes_{j=1}^r
    M_{j\bullet}^* f_{m+j,(2)},
  \end{equation*}
  from which it is clear that~$\cum^{x_n} \tilde{M}^* f = \tilde{M}^*
  \cum^{x_n} f$ since the Rota-Baxter operator~$\cum^{x_n}$ affects
  only the prefix~$f_1 \otimes \cdots \otimes f_m$.

  \eqref{it:ordinary-piece} We must show that~$(\balg_n, \cum^{x_n})$
  is an ordinary Rota-Baxter algebra over~$\balg_{n-1}$ with
  evaluation~$E_n^*$. Computing the tensor kernel
  via~\cite[Prop.~2.17]{Downie2012} we have
  \begin{equation*}
    \ker \cum^{x_n} = (\ker \id^{\otimes (n-1)}) \otimes \balg +
    \balg_{n-1} \otimes \ker \cum = 0 \otimes \balg + \balg_{n-1}
    \otimes 0
  \end{equation*}
  since~$\cum$ is injective; we conclude that~$\cum^{x_n}$ is
  injective as well. Since~$(\balg, \cum)$ is ordinary we have
  also~$\balg = K \dotplus \im \cum$, which implies
  \begin{equation*}
    \balg_n = \balg_{n-1} \otimes \balg = (\balg_{n-1} \otimes K)
    \dotplus (\balg_{n-1} \otimes \im \cum) \cong
    \balg_{n-1} \dotplus \im \cum^{x_n}
  \end{equation*}
  so~$(\balg_n, \cum^{x_n})$ is indeed ordinary
  over~$\balg_{n-1}$. The projector along the decomposition above is
  clearly~$\id^{\otimes (n-1)} \otimes \evl$, which sends~$f_1 \otimes
  \cdots \otimes f_{n-1} \otimes f_n$ to~$\evl(f_n) \, f_1 \otimes
  \cdots \otimes f_{n-1} \in \balg_{n-1} \subseteq \balg_n$. This
  agrees with
  \begin{equation*}
    E_n^* (f_1 \otimes \cdots \otimes f_{n-1} \otimes f_n) = (e_1^*
    f_1) \cdots (e_{n-1}^* f_{n-1}) \, (0^* f_n)      
  \end{equation*}
  since~$(e_1^* f_1) \cdots (e_{n-1}^* f_{n-1}) = f_1 \otimes \cdots
  \otimes f_{n-1}$ again by~\eqref{eq:it-counit} and~$0^* f_n =
  \evl^{\otimes n} \Delta^n f_n = \evl(f_n) \in K$ as in
  Item~\eqref{it:mapping-prop} of this proof.

  \eqref{it:perm-action} If~$\tau = (i \; j)$ is any transposition,
  the property~$\tau^* \cum^{x_i} = \cum^{x_j} \tau^*$ follows
  directly from the definition of the Rota-Baxter
  operators~$\cum^{x_n}$.

  \eqref{it:sub-rule} Finally, we have to show the three instances of
  the substitution rule. Among these, the diagonal substitution
  rule~\eqref{eq:scaling} is an immediate consequence of the diagonal
  substitution rule in~$(\balg, \Delta, \counit, \cum)$ as specified
  in Definition~\ref{def:Rota-Baxter-bialgebra}.

  For the horizontal substitution rule~\eqref{eq:transvec-spc} note
  that~$T_x(e_1) = (e_1 + e_2) \oplus e_2$ so that
  \begin{equation*}
    T_x(e_1)^* (f \otimes g) = (e_1+e_2)^* f \cdot e_2^* g
    = (\Delta f) (1 \otimes g) = f_{(1)} \otimes
    f_{(2)} g,
  \end{equation*}
  using again~\eqref{eq:it-counit} for the second
  factor. Since~$\cum^x (f \otimes g) = (\cum f) \otimes g$,
  using~\eqref{eq:hor-subst-rule} yields
  \begin{align*}
    T_x(e_1)^* & \cum^x (f \otimes g) = (\Delta \, \cum f) \, (1
    \otimes g) = \big( \cum^x \Delta f + 1 \otimes \cum f \big)
    (1 \otimes g)\\
    &= \cum f_{(1)} \otimes f_{(2)} g + 1 \otimes (\cum f) g =
    \cum^x \, T_x(e_1)^* (f \otimes g) + E_x^* T_x(e_1)^* \cum^x (f
    \otimes g),
  \end{align*}
  where the last summand comes from observing that~$E_x^* \,
  T_x(e_1)^* = \smallmat{0}{1}{0}{1}^*$ and hence
  \begin{equation*}
    E_x^* T_x(e_1)^* (\cum f) \otimes g = (e_2^* \cum f) \, (e_2^*
    g) = (\evl \otimes \id) \, \Delta \cum f \cdot (\evl \otimes
    \id) \, \Delta g = 1 \otimes (\cum f)g,
  \end{equation*}
  where~\eqref{eq:it-counit} was employed for the last step.

  It remains to show the vertical substitution
  rule~\eqref{eq:vert-subst-spc}, which will need a bit more
  effort. By Lemma~\ref{lem:vert-subst-alt}, we may assume that~$v'
  \in \{0,1\}^{n-2}$. Let us write~$w = e_1 + v \in \{0,1\}^{n-1}$,
  and let~$\Lambda \subset \nonzero{\NN}$ the set containing $1$ and
  all indices~$i$ with~$w_{i-1} = 1$. By the definition of~$\cum^x$
  and~$\cum^y$, it suffices to show~\eqref{eq:vert-subst-spc}
  on~$\balg_n$. Hence let~$f = f_1 \otimes \cdots \otimes f_n \in
  \balg_n$ be arbitrary but fixed. We can write this as~$\hat{f}^0
  \hat{f}^1$ with~$\hat{f}^j = \hat{f}^j_1 \otimes \cdots \otimes
  \hat{f}^j_n \; (j = 0,1)$ and
  \begin{equation*}
    \hat{f}^0_i =
    \begin{cases}
      1 & \text{if $i \in \Lambda$,}\\
      f_i & \text{otherwise,}
    \end{cases}
    \quad\text{and}\quad
    \hat{f}^1_i = 
    \begin{cases}
      f_i & \text{if $i \in \Lambda$,}\\
      1 & \text{otherwise.}
    \end{cases}
  \end{equation*}
  Observe that~$\hat{f}^0 = 1 \otimes 1 \otimes \cdots$ commutes
  with~$\cum^x$ and~$\cum^y$. Using~\eqref{eq:it-counit} we obtain
  also
  \begin{equation}
    \label{eq:Lxw-cum}
    L_x(w)^* \cum^x f = L_x(w)^* \, (\cum f_1) \otimes f_2 \otimes
    \cdots \otimes f_n = \hat{f}^0 \, e_1^* (\cum f_1) \, \sprod_{i \in
      \Lambda \setminus \{ 1 \}} (e_1 + e_i)^* f_i
  \end{equation}
  and hence~$\cum^x L_x(w)^* \cum^x f = \hat{f}^0 \cum^x L_x(w)^*
  \cum^x \hat{f}^1$. A similar calculation shows that~$\hat{f}^0$
  commutes also with~$L_y(v')^*$, $L_y^{-1}(v')^*$
  and~$L_x(e_1)^*$. Consequently we can pull out~$\hat{f}^0$ both from
  the left and right-hand side of~\eqref{eq:vert-subst-spc}, and we
  may thus assume without loss of generality that~$\hat{f}^0 = 1 \in
  \balg_n$. Let~$k$ be the cardinality of~$\Lambda$. If~$\tau \in S_n$
  is any permutation sending~$\Lambda$ to~$\{ 1, \dots, k\}$,
  conjugation by~$\tau^*$ will reduce~\eqref{eq:vert-subst-spc} to the
  case~$f \in \balg_k \subseteq \balg_n$. Hence we may also assume
  that~$k = n$ and~$\Lambda = \{ 1, \dots, n
  \}$. Then~\eqref{eq:Lxw-cum} shows that
  \begin{equation*}
    L_x(w)^* \cum^x f = e_1^* (\cum f_1) \, \sprod_{i>1} (e_1 +
    e_i)^* f_i
  \end{equation*}
  Using again~\eqref{eq:it-counit} we have~$e_1^* (\cum f_1) = (\cum
  f_1) \otimes 1 \otimes \cdots \otimes 1 \in \balg_n$. From the
  definition of the matrix action we have~$\Delta = (2 \: i)^*
  (e_1+e_i)^*$, which implies
  \begin{align}
    \cum^x L_x(w)^* & \cum^x f = \cum^x \Big( (\cum f_1) \sprod_{i>1}
    f_{i,(1)} \otimes\, \sbigotimes_{i>1} f_{i,(2)} \Big) = \big( \cum
    \, (\cum f_1) \sprod_{i>1} f_{i,(1)}
    \big) \otimes\, \sbigotimes_{i>1} f_{i,(2)}\nonumber\\
    &= \Big( (\cum f_1) \, (\cum \sprod_{i>1} f_{i,(1)}) \Big)
    \otimes\, \sbigotimes_{i>1} f_{i,(2)} - \cum \Big( f_1 \, \cum
    \sprod_{i>1} f_{i,(1)} \Big) \otimes \sbigotimes_{i>1}
    f_{i,(2)}\nonumber\\
    &= (\cum f_1) \, \cum^x \Big( \sprod_{i>1} f_{i,(1)} \otimes\,
    \sbigotimes_{i>1} f_{i,(2)} \Big) - \cum^x f_1 \Big( (\cum
    \sprod_{i>1} f_{i,(1)}) \otimes\, \sbigotimes_{i>1} f_{i,(2)}
    \Big)
    \label{eq:cum-Lxw-cum}
  \end{align}
  where in the last but one step the Rota-Baxter axiom for~$\balg_1$
  was applied, while the last step uses the embedding~$\balg_1 \subset
  \balg_n$. Turning now to the right-hand side
  of~\eqref{eq:vert-subst-spc}, we obtain
  \begin{equation*}
      L_y(v')^* f = f_1 \otimes f_2 \sprod_{i>2} f_{i,(1)} \otimes\,
      \sbigotimes_{i>2} f_{i,(2)}
  \end{equation*}

  \vskip-\lastskip\noindent
  and hence by~\eqref{eq:hor-subst-rule} also
  \begin{align*}
    L_x(e_1)^* & \cum^x \cum^y L_y(v')^* f = L_x(e_1)^* \Big( (\cum
    f_1) \otimes (\cum f_2 \sprod_{i>2}
    f_{i,(1)}) \otimes\, \sbigotimes_{i>2} f_{i,(2)} \Big)\\
    &= (\cum f_1) \, \Big( \Delta \big( \cum f_2 \sprod_{i>2} f_{i,(1)}
    \big) \otimes\, \sbigotimes_{i>2} f_{i,(2)}
    \Big)\\
    & = (\cum f_1) \, \cum^x \Big( f_{2,(1)} \sprod_{i>2} f_{i,(1)(1)}
    \otimes f_{2,(2)} \sprod_{i>2} f_{i,(1)(2)}
    \otimes\, \sbigotimes_{i>2} f_{i,(2)} \Big)\\
    & \qquad + (\cum f_1) \otimes \big( \cum f_2
    \sprod_{i>2} f_{i,(1)} \big) \otimes\, \sbigotimes_{i>2} f_{i,(2)}
    \Big) .
  \end{align*}
  Using the antipode~$S = (-1)^*$, we obtain further
  \begin{align}
    L_y^{-1}&(v')^* L_x(e_1)^* \cum^x \cum^y L_y(v')^* f\nonumber\\
    & = (\cum f_1) \, \cum^x \Big( f_{2,(1)} \sprod_{i>2} f_{i,(1)(1)}
    \otimes f_{2,(2)} \sprod_{i>2} f_{i,(1)(2)} \,
    Sf_{i,(2)(1)} \otimes \sbigotimes_{i>2} f_{i,(2)(2)} \Big)\label{eq:first-term}\\
    & \qquad + (\cum f_1) \otimes \big( \cum f_2 \sprod_{i>2}
    f_{i,(1)} \big) \sprod_{i>2} Sf_{i,(2)(1)} \otimes\,
    \sbigotimes_{i>2} f_{i,(2)(2)}\label{eq:second-term}
  \end{align}
  Let us first study the left summand of~\eqref{eq:first-term}, which
  we claim to be equal to the first term
  of~\eqref{eq:cum-Lxw-cum}. For this, it is sufficient to show that
  \begin{equation}
    \label{eq:intermed}
    \sprod_{i>2} f_{i,(1)(1)} \otimes \sprod_{i>2} f_{i,(1)(2)} \, S
    f_{i,(2)(1)} \otimes\, \sbigotimes_{i>2} f_{i,(2)(2)} = 
    \sprod_{i>2} f_{i,(1)} \otimes 1 \otimes\, \sbigotimes_{i>2} f_{i,(2)}.
  \end{equation}
  From~\eqref{eq:it-coassoc} we have~$(\Delta \otimes \Delta) \Delta =
  \Delta^4 = (\id \otimes \Delta \otimes \id) \Delta^3$, which may be
  used to rewrite the left-hand side as
  \begin{align*}
    \: & \sprod_{i>2} (3 \: i)^* (\id \otimes \nabla \otimes
    \id)(\Delta \otimes S \otimes \id) \, f_{i,(1)} \otimes
    f_{i,(2)(1)} \otimes f_{i,(2)(2)}\\
    &= \sprod_{i>2} (3 \: i)^* (\id \otimes \nabla \otimes
    \id) (\id^{\otimes 2} \otimes S \otimes \id) (\Delta \otimes
    \id) (\id \otimes \Delta) \Delta f_i,\\
    &= \sprod_{i>2} (3 \: i)^* (\id \otimes \nabla(\id
    \otimes S) \otimes \id) (\id \otimes \Delta \otimes \id)
    \Delta^3 f_i
    = \sprod_{i>2} (3 \: i)^* (\id \otimes \evl \otimes
    \id) \Delta^3 f_i\\
    &= \sprod_{i>2} (3 \: i)^* (\id \otimes \evl
    \otimes \id) (\Delta \otimes \id) \Delta f_i
    = \sprod_{i>2} (3 \: i)^* f_{i,(1)} \otimes 1 \otimes
    f_{i,(2)}
  \end{align*}
  where the third step applies the identity~$\nabla (\id \otimes S)
  \Delta = \evl$, which is true because~$(\balg, \counit, \Delta, S)$
  is a Hopf algebra. The last expression follows
  from~\eqref{eq:it-counit} and is clearly equal to the right-hand
  side of~\eqref{eq:intermed}.

  We determine now the second term~\eqref{eq:second-term} of the
  commutator on the right-hand side of~\eqref{eq:vert-subst-spc}. To
  start with, we compute
  \begin{align*}
    & L_x(e_1)^* \cum^y L_y(v')^* f = f_1 \, \Delta \big( \cum f_2
    \sprod_{i>2} f_{i,(1)} \big) \otimes\, \sbigotimes_{i>2} f_{i,(2)}\\
    &= f_1 \, \cum^x \Big( f_{2,(1)} \sprod_{i>2} f_{i,(1)(1)} \otimes
    f_{2,(2)} \sprod_{i>2} f_{i,(1)(2)} \otimes\, \sbigotimes_{i>2}
    f_{i,(2)} \Big)
    + f_1 \otimes \big( \cum f_2 \sprod_{i>2} f_{i,(1)}
    \big) \otimes\, \sbigotimes_{i>2} f_{i,(2)},
  \end{align*}
  using once again~\eqref{eq:hor-subst-rule}; hence the required
  commutator term~$L^{-1}_y(v')^* \cum^x L_x(e_1)^* \cum^y L_y(v')^*
  f$ is given by
  \begin{align*}
    \big( \cum & f_1 \cum f_{2,(1)} \sprod_{i>2} f_{i,(1)(1)} \big)
    \otimes f_{2,(2)} \sprod_{i>2} f_{i,(1)(2)} \, Sf_{i,(2)(1)}
    \otimes\, \sbigotimes_{i>2} f_{i,(2)(2)}\\
    & \qquad + (\cum f_1) \otimes \big( \cum f_2 \sprod_{i>2} f_{i,(1)}
    \big) \sprod_{i>2} Sf_{i,(2)(1)} \otimes\, \sbigotimes_{i>2}
    f_{i,(2)(2)} .
  \end{align*}
  We notice that the second summand cancels with the second summand
  of the first commutator term~\eqref{eq:first-term}. Therefore it
  remains to prove
  \begin{align*}
    \big( \cum & f_1 \cum f_{2,(1)} \sprod_{i>2} f_{i,(1)(1)} \big)
    \otimes f_{2,(2)} \sprod_{i>2} f_{i,(1)(2)} \, Sf_{i,(2)(1)}
    \otimes\, \sbigotimes_{i>2} f_{i,(2)(2)}\\
    &= \cum^x f_1 \Big( \cum f_{2,(1)} \sprod_{i>2} f_{i,(1)(1)} 
    \otimes f_{2,(2)} \sprod_{i>2} f_{i,(1)(2)} \, Sf_{i,(2)(1)}
    \otimes\, \sbigotimes_{i>2} f_{i,(2)(2)} \Big)\\
    &= \cum^x f_1 \Big( (\cum \sprod_{i>1} f_{i,(1)}) \otimes\,
    \sbigotimes_{i>1} f_{i,(2)} \Big),
  \end{align*}
  where the first equality uses only the definition of~$\cum^x$. But
  this clearly follows from
  \begin{equation*}
    f_{2,(1)} \sprod_{i>2} f_{i,(1)(1)} \otimes f_{2,(2)} \sprod_{i>2}
    f_{i,(1)(2)} \, Sf_{i,(2)(1)} \otimes\, \sbigotimes_{i>2}
    f_{i,(2)(2)} = \sprod_{i>1} f_{i,(1)} \otimes\, \sbigotimes_{i>1}
    f_{i,(2)},
  \end{equation*}
  which is a trivial consequence of the identity~\eqref{eq:intermed}
  that we have shown above.

  This concludes the proof that~$(\indhier{\balg}, \cum^{x_n})_{n \in
    \NN}$ is a Rota-Baxter hierarchy.
\end{proof}

\begin{example}
  \label{ex:Rota-Baxter-bialgebra}
  Let us now make sure that the scaled bialgebras of our standard
  model (Example~\ref{ex:scaled-bialg}) are in fact scaled Rota-Baxter
  bialgebras under the natural choice of Rota-Baxter operator. We do
  this for the exponential polynomials~$K[x,e^{Kx}]$, which includes
  the plain polynomials~$K[x]$ as a Rota-Baxter subhierarchy.

  The algebraic definition of the Rota-Baxter operator~$\cum$
  on~$K[x,e^{Kx}]$ can be given either in recursive or in summation
  form; the latter is more amenable for our purposes. Hence we have
  for~$k \in \NN, \alpha \in K$ the formulae
  \begin{align*}
    \cum x^k e^{\alpha x} &= \tfrac{(-1)^{k+1} k!}{\alpha^{k+1}} +
    \sum_{i=0}^k \tfrac{(-1)^i \fafac{k}{i}}{\alpha^{i+1}} x^{k-i}
    e^{\alpha x} \qquad (\alpha \neq 0),\\
    \cum x^k &= \tfrac{x^{k+1}}{k+1},
  \end{align*}
  where~$\fafac{k}{i} = k!/(k-i)!$ denotes the falling factorial. For
  verifying that~$(K[x,e^{Kx}], \cum)$ is a scaled Rota-Baxter
  bialgebra we have to check that it satisfies the diagonal
  substitution rule~$\cum \lambda^* = \lambda^{-1} \lambda^* \cum \;
  (\lambda \neq 0)$ and the horizontal substitution
  rule~\eqref{eq:hor-subst-rule}. The former is immediate, so let us
  verify the latter in the form~\eqref{eq:hor-subst-rule-alt}. For the
  $K$-basis vector~$f = x^k e^{\alpha x}$ we compute
  \begin{align*}
    \Delta \cum f &= \tfrac{(-1)^{k+1} k!}{\alpha^{k+1}} + \sum_{i=0}^k
    \sum_{j=0}^{k-i} \tbinom{k-i}{j} \tfrac{(-1)^i
      \fafac{k}{i}}{\lambda^{i+1}} \, x^j y^{k-i-j} e^{\alpha(x+y)},\\
    \cum^x \Delta f &= \sum_{l=0}^k \tbinom{k}{l} \tfrac{(-1)^{l+1}
      l!}{\alpha^{l+1}} \, y^{k-l} e^{\alpha y}
    + \sum_{l=0}^k \sum_{i=0}^l \tbinom{k}{l} \tfrac{(-1)^i \,
      \fafac{l}{i}}{\alpha^{i+1}} \, x^{l-i} y^{k-l} e^{\alpha(x+y)},\\
    (\cum f)(y) &= \tfrac{(-1)^{k+1} \, k!}{\alpha^{k+1}}
    - \sum_{l=0}^k \tfrac{(-1)^{l+1} \, \fafac{k}{l}}{\alpha^{l+1}} \,
    y^{k-l} e^{\alpha y}.
  \end{align*}
  We observe that the constant terms of~$\Delta \cum f$ cancels with
  the one of~$(\cum f)(y)$, likewise the first term of~$\cum^x \Delta
  f$ with the second term of~$(\cum f)(y)$; thus it remains to show
  the two double sums equal. But this follows immediately by
  transforming the index~$i$ of the outer sum in~$\Delta \cum f$ to~$l
  = i+j$ and then swapping the summations. The verification for the
  $K$-basis vector~$f = x^k$ is similar but simpler.
\end{example}

Since every ordinary Rota-Baxter algebra~$(\ogalg, \cum)$
contains~$K[x]$, its induced hierarchy contains the ascending algebra
of \emph{polynomial rings}~$K[x_1, \dots, x_n]$. These are closed
under all integrators and linear substitutions, so one can always
choose~$K[x_1, x_2, \dots]$ as the simplest coefficient domain.
Another natural choice for the classical
example~$C^\infty(\RR^\infty)$ is given by the exponential
polynomials discussed in Example~\ref{ex:Rota-Baxter-bialgebra}
above. Both of these are instances of admissible coefficient domains
in the following sense.

\begin{definition}
  \label{def:separated-algebra}
  A Rota-Baxter hierarchy is called \emph{separated} if it is of the
  form~$\indhier{\balg}$ for some scaled Rota-Baxter
  bialgebra~$\balg$.  If~$\galg$ is a fixed Rota-Baxter hierarchy, an
  \emph{admissible coefficient domain} for~$\galg$ is a separated
  Rota-Baxter subhierarchy~$\ogalg \le \galg$.
\end{definition}

As remarked above, the \emph{minimal choice} is to take~$\ogalg =
\indhier{K[x]}$, which is an admissible coefficient algebra for any
Rota-Baxter hierarchy~$(\galg, \cum^{x_n})_{n \in \NN}$. We may view
it as an analog of the prime field in any given field.

%=====================================================================
\section{The Ring of Partial Integral Operators}
\label{sec:pios}
%=====================================================================

We proceed now to the task of setting up a ring of partial integral
operators and substitutions acting on a given Rota-Baxter hierarchy,
and taking coefficients in any admissible coefficient domain. We will
do this in two steps: first identifying first an ideal of suitable
\emph{operator relations}, then constructing the operator ring as a
\emph{quotient algebra} of the free algebra modulo the relation ideal.

Note the parallel development for \emph{rings of differential
  operators}, where the $n$-th Weyl algebra~$A_n(K)$ is built as a
quotient of~$K[x_1, \dots, x_n; \der_1, \dots, \der_n]$
modulo~$[\der_i, x_j] = \delta_{ij}$. In the case of a \emph{single
  Rota-Baxter operator} (without linear substitutions), the
corresponding operator ring has been studied~\cite{GuoLin2015} in
connection with representations of Rota-Baxter algebras.

%---------------------------------------------------------------------
\subsection{Crucial Operator Relations}
\label{ssec:operator-relations}
%---------------------------------------------------------------------

Logically speaking, all the operator relations that we shall now
compile are in fact consequences of the axioms of Rota-Baxter
hierarchies (and admissible coefficient algebras). But we need the
operator relations in this special form since it facilitates the
algorithmic treatment of partial integral operators and linear
substitutions. Ultimately, we hope to have a \emph{noncommutative
  Gr\"obner basis} for the relation ideal (see the end of
Section~\ref{ssec:const-quot}).

First we need to clarify some \emph{notational
  conventions}. Let~$(\galg_n, \cum^{x_n})_{n \in \NN}$ be a fixed
Rota-Baxter hierarchy over~$K$. For avoiding confusion with
multiplication operators, we will from now on apply the alternative
notation~$f[M]$ for the action~$M^*(f)$. If~$M = (i \; j)$ is the
matrix of a transposition, we write~$f[i \; j]$ for~$f[M]$, and
parentheses are also dropped for embedded row vectors~$M$.  We may
identify all elements~$f \in \galg$ with their induced multiplication
operators~$\galg \to \galg, u \mapsto fu$. Thus the operator~$M^* f$
denotes the composition~$u \mapsto M^*(fu)$ rather than the
action~$f[M]$. The basic operator relation~$M^* f = f[M] \, M^*$
describes how substitutions interact with multiplication
operators. Note that~$f[M][\tilde M] = f[M \tilde M]$.

Let~$\ogalg$ be an \emph{admissible coefficient domain}
for~$\galg$. Then we have~$\ogalg = \indhier{\balg} \le \galg$ for a
scaled Rota-Baxter bialgebra~$\balg$. Since the latter is determined
by the coefficient domain as the first tensor grade~$\balg =
\ogalg_1$, we may omit reference to~$\balg$ altogether. If~$f \in
\ogalg_1$ we shall write~$f(x_i) \in \ogalg_{(1)}$ as an intuitive
shorthand for~$(1 \; i)^* f$.

The crucial property of an admissible coefficient domain is that it
provides \emph{tensor expansion} all multivariate functions~$g \in
\galg$ as per
\begin{equation}
  \label{eq:tensor-expn}
  g = \sum_{\mu=1}^r g_{1,\mu} \otimes \cdots \otimes g_{n,\mu}
\end{equation}
for suitable~$g_{1,\mu}, \dots, g_{n,\mu} \in \galg_1$, where we shall
consistently use Greek indices for the components. For definiteness,
we fix an ordered $K$-basis~$(b_i)$ for~$\galg_1$, and we use the
graded lexicographic ordering for the induced $K$-basis on~$\galg_n$.
Then each summand above has the form~$\lambda_\mu \, b_{i_1(\mu)}
\otimes \cdots \otimes b_{i_n(\mu)}$ so that we may choose~$g_{1,\mu}
= \lambda_\mu \, b_{i_1(\mu)}$ and~$g_{j,\mu} = b_{i_j(\mu)}$
for~$j>1$. Consequently, each~$g \in \galg_n$ is uniquely described by
the~$n$ \emph{component sequences}
\begin{equation*}
  g\{j\} := (1 \: j)^* \, (g_{j,1}, \dots, g_{j,r}) \in \ogalg_{(j)}^r
  \qquad (j=1, \dots, n),
\end{equation*}
where the number~$r \in \NN$ depends of course on the element~$g \in
\galg_n$; its minimial value is the tensor rank of~$g$. Note that we
recover~\eqref{eq:tensor-expn} by setting~$g_{j,\mu} =
\big(g\{j\}\big)_\mu$ and replacing the tensor product by the product
of~$\ogalg$. For~$\alpha \subseteq \{1, \dots, n\}$ we will also use
the extended sequence notation
\begin{equation*}
  g\{\alpha\} := \prod_{j \in \alpha} g\{j\} \in \ogalg_\alpha^r,
\end{equation*}
where the product on the right-hand side is taken
componentwise. If~$\alpha=(i_1, \dots, i_r)$ we denote this briefly
by~$g\{i_1, \dots, i_r\}$. For complements we write~$g\{\alpha\}' :=
g\{\alpha'\}$, where~$\alpha'$ is to be understood as~$\{1, \dots, n\}
\setminus \alpha$.

We will employ component sequences in conjunction with the following
variant of the \emph{Einstein summation convention}: Unless stated
otherwise, summation is implied over all Greek indices on a component
sequence (over the index range of the component
sequence).\footnote{Incidentally, there is a close relation to
  \emph{sumless Sweedler notation}: Writing~$\gamma := \Delta^n g \in
  \galg_n$ for the (iterated) coproduct, one obtains the
  relation~$\Delta^n g = \gamma\{1\}_\mu \cdots \gamma\{n\}_\mu =
  g_{(1)} \cdots g_{(n)}$ so that one could identify the ``formal
  symbol''~$g_{(i)}$ with the component sequence~$\gamma\{i\}$.} For
example, the expansion~\eqref{eq:tensor-expn} can be written as~$g =
g\{1\}_\mu \, g\{1\}'_\mu$. Similarly, we have equivalent expansions
like~$g = g\{ 1 \}_\mu \, g\{ 1, 2 \}'_\mu \, g\{ 2 \}_\mu$. The
practical value of this convention will become apparent when component
sequences are used in noncommutative operator expressions, as in the
following result generalizing Lemma~\ref{lem:simple-prop}. Here and
henceforth we employ the following notation for sequence substitution:
If~$\gamma \in \galg_\alpha^r$ is an extended operator sequence and~$M
\in \matmon{K}{n}$ any substitution, we write~$\gamma[M]$ for the
sequence with components~$\gamma_1[M], \dots, \gamma_r[M]$. In the
sequel, indices~$i,j$ range over~$\{1, \dots, n\}$.

\begin{lemma}
  \label{lem:after-pivot}
  Let~$(\galg_n, \cum^{x_n})_{n \in \NN}$ be a Rota-Baxter hierarchy
  over a field~$K$, and~$\ogalg$ an admissible coefficient domain
  for~$\galg$. Let~$M \in \matmon{K}{n}$ have~$j$-th column~$e_i$ and $g
  \in \ogalg_1$. Then
  \begin{equation*}
    \cum^{x_j} g(x_j) \, M^* = \gamma'_\mu \, (1-E_j^*) \, M^*
    \cum^{x_i} \, \gamma[i \: j]_\mu
  \end{equation*}
  for~$\gamma' = \tilde{g}\{j\}'$, $\gamma = \tilde{g}\{j\}$
  and~$\tilde{g} = g[-M_{i,1}, \dots, -M_{i,j-1}, 1, -M_{i,j+1}, \dots,
  -M_{i,n}]$.
\end{lemma}
\begin{proof}
  It suffices to consider~$i=j=1$ since the general case follows then
  by multiplying with~$(1 \: j)^*$ from the left and~$(1 \; i)^*$ from
  the right. Hence write~$M = \smallmat{1}{v}{0}{A}$ with~$v \in
  K_{n-1}$ and~$A \in \matmon{K}{n-1}$. Note that~$M =
  \smallmat{1}{0}{0}{A} \smallmat{1}{v}{0}{I}$ with~$I := I_{n-1}$ so
  that
  \begin{align*}
    \cum^{x_1} & g(x_1) \, M^* = \cum^{x_1} g \,
    \smallmat{1}{v}{0}{I}^* \smallmat{1}{0}{0}{A}^* = \cum^{x_1}
    \smallmat{1}{v}{0}{I}^* \tilde{g} \, \smallmat{1}{0}{0}{A}^* =
    (1-E_1^*) \, \smallmat{1}{v}{0}{I}^* \cum^{x_1} \tilde{g} \,
    \smallmat{1}{0}{0}{A}^*\\
    &= (1-E_1^*) \, \smallmat{1}{v}{0}{I}^* \gamma'_\mu \, \cum^{x_1}
    \gamma_\mu \, \smallmat{1}{0}{0}{A}^*
    = (1-E_1^*) \, \smallmat{1}{v}{0}{I}^* \gamma'_\mu \,
    \smallmat{1}{0}{0}{A}^* \cum^{x_1} \gamma_\mu\\
    &= \gamma'_\mu \, (1-E_1^*) \, M^* \cum^{x_1} \gamma_\mu,
  \end{align*}
  where the third equality follows from Lemma~\ref{lem:gen-transvec},
  the fourth employs tensor expansion and Item~\eqref{it:trivial-int}
  of Lemma~\ref{lem:simple-prop}, the fifth follows from~$\gamma \in
  \ogalg_{(1)}^{r}$ and Item~\eqref{it:mapping-prop} of
  Definition~\ref{def:hier-Rota-Baxter}, the sixth from~$\gamma' \in
  \ogalg_{(1)'}^{r}$ and straightness.
\end{proof}

Lemma~\ref{lem:after-pivot} allows us to put arbitrary one-dimensional
integrators~$\cum^{x_i} f(x_i) M^*$ into normal form. All we have to
do is to apply one sweep of Gaussian elimination to~$M$ for creating
zeroes underneath the first pivot element in the~$j$-th column
of~$M$. If the pivot is~$M_{ij}$ so that~$M_{1j} = \cdots = M_{i-1,j}
= 0$, one uses the eliminant
\begin{equation}
  \label{eq:l-matrix}
  L_i(-l) =
  \begin{pmatrix}
    1&&& 0\\
    &\ddots && \vdots\\
    &&1& 0\\
    &&&1\\
    &&&-l_{i+1} & 1\\
    &&&\vdots && \ddots\\
    &&&-l_n &&& 1
  \end{pmatrix},
\end{equation}
where~$l = (l_{i+1}, \dots, l_n) \in K_{n-i}$ with~$l_k :=
M_{kj}/M_{ij}$. Applying this elimination yields
\begin{equation}
  \label{eq:red-mat}
  \tilde{M} := L_i(-l) \cdot M =
  \begin{pmatrix}
    M_{11} & \cdots &  0 & \cdots & M_{1n}\\
    \vdots && \vdots && \cdots\\
    M_{i-1,1} & \cdots & 0 & \cdots & M_{i-1,n}\\
    M_{i1} & \cdots & M_{ij} & \cdots & M_{in}\\
    \tilde{M}_{i+1,1} & \cdots & 0 & \cdots & \tilde{M}_{i+1,n}\\
    \vdots && \vdots && \vdots\\
    \tilde{M}_{n1} & \cdots & 0 & \cdots & \tilde{M}_{nn}
  \end{pmatrix}
\end{equation}
with~$M_{ij} e_i$ in its~$j$-th column.

\begin{proposition}[Normalization of One-Dimensional Integrators]\ \\
  \label{prop:1d-subst-rule}%
  Let~$(\galg_n, \cum^{x_n})_{n \in \NN}$ be a Rota-Baxter hierarchy
  over a field~$K$, and let~$\ogalg$ be an admissible coefficient
  domain for~$\galg$. Then for~$M \in \matmon{K}{n}$ and~$g \in \ogalg_1$
  we have
  \begin{equation}
    \label{eq:1d-subst-rule}
    \cum^{x_j} g(x_j) M^* =
    \begin{cases}
      M_{ij}^{-1} \, \gamma'_\mu (1-E_j^*) \tilde{M}^*
      \cum^{x_i} \gamma[\mmt{j}{1/M_{ij}}]_\mu \, L_i(l)^* &
      \text{if $i \neq \infty$},\\[0.4ex]
      \big(\cum^{x_j} g(x_j)\big) M^* & \text{otherwise},
    \end{cases}
  \end{equation}
  where~$i = \min \{ k \mid M_{kj} \neq 0 \}$, and~$\tilde{M} \in
  \matmon{K}{n}$, $l \in K_{n-i}$ as in~\eqref{eq:l-matrix},
  \eqref{eq:red-mat} if~$i \neq \infty$. Moreover, we set~$\gamma' =
  \tilde{g}\{j\}'$, $\gamma = \tilde{g}\{j\}$ and~$\tilde{g} =
  g[-\tfrac{M_{i,1}}{M_{ij}}, \dots, -\tfrac{M_{i,j-1}}{M_{ij}}, 1,
  -\tfrac{M_{i,j+1}}{M_{ij}}, \dots, -\tfrac{M_{i,n}}{M_{ij}}]$.
\end{proposition}
\begin{proof}
  Assume first~$i=\infty$. Again it suffices to consider~$j=1$ since
  the general case follows by multiplying with~$(1 \: j)^*$ from the
  left. We have to show~$\cum^{x_1} g M^* f = \big(\cum^{x_1} g\big)
  M^* f$ for all~$f \in \galg$. We may assume~$f \in \galg_k$ for
  some~$k > n$, and we may also view~$M \in
  \matmon{K}{k}$. Setting~$\tau = (1 \: k)$, we obtain a matrix~$M
  \tau \in \matmon{K}{k}$ with~$k$-th column zero, and
  Item~\eqref{it:zero-col} of Lemma~\ref{lem:simple-prop} implies
  that~$(M\tau)^* f = \tau^* M^* \! f \in \galg_{(k)}'$. Then we
  obtain~$\cum^{x_k} (\tau^* g) (\tau^* \!  M^* f) = (\tau^* \!  M^*
  \! f) \, (\cum^{x_k} \tau^* g)$ from Item~\eqref{it:trivial-int} of
  the same lemma. Multiplying this relation by~$\tau^*$ from the left
  yields the desired identity.

  Now assume the minimum exists. We use the
  decomposition~\eqref{eq:red-mat} to obtain
  \begin{equation*}
    M^* = (\mmt{i}{1/M_{ij}} \tilde{M})^* \mmt{i}{M_{ij}}^* L_i(l)^*,
  \end{equation*}
  such that the $j$-th column of~$\mmt{i}{1/M_{ij}} \tilde{M}$
  is~$e_i$. Applying Lemma~\ref{lem:after-pivot}
  and~\eqref{eq:scaling} yields
  \begin{align*}
    \cum^{x_j} & g(x_j) \, M^* = \cum^{x_j} g \, (\mmt{i}{1/M_{ij}}
    \tilde{M})^* \mmt{i}{M_{ij}}^* L_i(l)^*\\
    &= \gamma'_\mu \, (1-E_j^*) (\mmt{i}{1/M_{ij}}
    \tilde{M})^* \cum^{x_i} \gamma[i \: j]_\mu \, \mmt{i}{M_{ij}}^* L_i(l)^*\\
    &= \gamma'_\mu \, (1-E_j^*) (\mmt{i}{1/M_{ij}} \tilde{M})^*
    \cum^{x_i} \mmt{i}{M_{ij}}^* \gamma[\mmt{j}{1/M_{ij}}]_\mu \, L_i(l)^*\\
    &= M_{ij}^{-1} \, \gamma'_\mu \, (1-E_j^*) (\mmt{i}{1/M_{ij}}
    \tilde{M})^* \mmt{i}{M_{ij}}^* \cum^{x_i}
    \gamma[\mmt{j}{1/M_{ij}}]_\mu \, L_i(l)^*\\
    &= M_{ij}^{-1} \, \gamma'_\mu (1-E_j^*) \tilde{M}^*
    \cum^{x_i} \gamma[\mmt{j}{1/M_{ij}}]_\mu \, L_i(l)^*
  \end{align*}
  which is indeed~\eqref{eq:1d-subst-rule} for~$i \neq \infty$.
\end{proof}

Before we proceed to the ordering of one-dimensional integrators, let
us note the following natural commutativity result for integrals and
evalutions along distinct axes.

\begin{corollary}
  \label{cor:eval-int-comm}
  For~$i \neq j$ the operators~$\cum^{x_i}$ and~$E_j^*$ commute.
\end{corollary}
\begin{proof}
  Applying Lemma~\ref{lem:after-pivot} with~$M=E_j$ and~$g=1$
  yields~$\cum^{x_i} E_j^* = (1-E_i^*) \, E_j^* \, \cum^{x_i}$, which
  implies commutativity because of~$E_i E_j = E_j E_i$
  and~\eqref{it:eval-van} of Lemma~\ref{lem:simple-prop}.
\end{proof}

By Proposition~\ref{prop:1d-subst-rule}, the basic building blocks are
one-dimensional inegrators of the form~$\cum^{x_i} h(x_i) L_i(v)^*$
along a coordinate axis~$x_i$. We shall now investigate how two such
integrators interact with each other, starting with the case of
different coordinate axes. It turns out that we can always reorder
them. In this paper, we choose to put them in \emph{ascending order}.

\begin{proposition}[Ordering of One-Dimensional Integrators]\ \\
  \label{prop:order-int}%
  Let~$(\galg_n, \cum^{x_n})_{n \in \NN}$ be a Rota-Baxter hierarchy
  over a field~$K$, and~$\ogalg$ be an admissible coefficient
  domain for~$\galg$. Then for~$i < j$, vectors~$v \in K^{n-i}$,~$w
  \in K^{n-j}$ and functions~$g,h \in \ogalg_1$ we have
  \begin{equation}
    \label{eq:order-int}
    \cum^{x_j} h(x_j) L_j(w)^* \cum^{x_i} g(x_i) L_i(v)^* = (1-E_j^*)
    \cum^{x_i} \eta'_\mu \, L_i(v')^*
    \cum^{x_j} \eta_\mu \, L_j(w)^*,
  \end{equation}
  where we set~$v' = L_{j-i}^{-1}(w) \, v = \trp{(v_{i+1}, \cdots,
    v_j, v_{j+1} - w_{j+1} v_j, \dots, v_n - w_n v_j)} \in K^{n-i}$,
  as well as~$\eta' = g(x_i) \, \tilde{h}\{i\}$, $\eta =
  \tilde{h}\{j\}$ and~$\tilde{h} = h[e_j - v_j e_i] \in
  \ogalg_{(i,j)}$.
\end{proposition}
\begin{proof}
  First note that Lemma~\ref{lem:after-pivot} with~$g=1$ yields
  $\cum^{x_i} L_j(w)^* = L_j(w)^* \cum^{x_i} - \big(L_j(w) \,
  E_i\big)^* \cum^{x_i}$, but the last summand vanishes
  by~\eqref{it:eval-van} of Lemma~\ref{lem:simple-prop} since~$L_j(w)
  \, E_i = E_i \, L_j(w)$. Hence
  \begin{align*}
    \cum^{x_j} & h(x_j) L_j(w)^* \cum^{x_i} g(x_i) L_i(v)^* =
    \cum^{x_j} h(x_j) \cum^{x_i} L_j(w)^* g(x_i) L_i(v)^*\\
    &= \cum^{x_i} g(x_i) \cum^{x_j} h(x_j) L_i(v')^* L_j(w)^* =
    \cum^{x_i} g(x_i) \, \tilde{h}\{i\}_\mu \, (1-E_j^*) \,
    L_i(v')^* \cum^{x_j} \tilde{h}\{j\}_\mu \, L_j(w)^*\\
    &= (1-E_j^*) \cum^{x_i} g(x_i) \, \tilde{h}\{i\}_\mu
    L_i(v')^* \cum^{x_j} \tilde{h}\{j\}_\mu \, L_j(w)^*
  \end{align*}
  where the second step follows from~\eqref{eq:lmat-desc} and
  Item~\eqref{it:trivial-int} of Lemma~\ref{lem:simple-prop}, from
  $L_j(w)_{\lrcorner i} = I_i$ and the commutativity of Rota-Baxter
  operators, the third from Lemma~\ref{lem:after-pivot}
  with~$\tilde{h} = L_i(v')_{j\bullet}^* h$ and~$L_i(v')_{j\bullet} =
  -v_j' e_i + e_j = e_j - v_j e_j$, the last step from
  Corollary~\ref{cor:eval-int-comm} and straightness.
\end{proof}

\begin{lemma}
  \label{lem:coalesc}%
  Let~$(\galg_n, \cum^{x_n})_{n \in \NN}$ be a Rota-Baxter hierarchy
  over a field~$K$, and~$\ogalg$ an admissible coefficient domain
  for~$\galg$. Then for~$i \in \NN$, $h \in \ogalg_1$ and~$w \in
  K^{n-i} \setminus \{0\}$ we have
  \begin{equation}
    \label{eq:coalesc-int}
    \begin{aligned}
      \cum^{x_i} h(x_i) \, L_i(w)^* \cum^{x_i} = L_k^{-1}(w')^*
      \, \eta'[k \; n\!+\!1]_\mu \, \Big( L_i(\bar{w})^* \,
      \cum^{x_i} - \cum^{x_i} L_i(\bar{w})^* \Big) \cum^{x_k}
      \eta_\mu \, L_k(w')^*
    \end{aligned}
  \end{equation}
  as an identity on~$\galg_n$, where~$k = \min \{j > i \mid w_j \neq 0
  \}$ and~$\eta' = \tilde{h}\{n+1\}$, $\eta = \tilde{h}\{k\}$ with
  \begin{equation*}
    \tilde{h} := h[e_k/w_k-e_{n+1}/w_k]/w_k \in \galg_{(k,n+1)}.
  \end{equation*}
  Moreover, for~$w = \trp{(w_{i+1}, \dots, w_n)}$ we have defined
  \begin{equation*}
    \bar{w} = w_k e_{k-i} \in K^{n-i}, \qquad
    w' = (w_{k+1}, \dots, w_n)/w_k \in K^{n-k}
  \end{equation*}
  as auxiliary vectors.
\end{lemma}
\begin{proof}
  It suffices to assume~$i=1$; the general case follows by conjugation
  with~$\tau^* = (i \; 1)^*$ since~$\tau L_k(w') \tau = L_k(w')$
  and~$\tau L_1(\bar{w}) \tau = L_i(\bar{w})$. We write~$J := I_n
  \oplus e_k$ as an abbreviation. Using the relation~$L_x(w)^* = w_k^*
  \, L_x(e_{k-1} + \trp{(0, w')}) \, (1/w_k)^*$, we obtain
  \begin{align*}
    & \cum^x h(x) L_x(w)^* \cum^x = w_k^{-2} \, w_k^* \, \cum^x
    h[d_{1,1/w_k}] \, L_x(e_{k-1} + \trp{(0, w')})^* \cum^x (\tfrac{1}{w_k})^*\\
    &= w_k^{-1} \, w_k^* \, L_k^{-1}(w')^* J^* \Big( L_x(e_{k-1})^* \cum^x
    - \cum^x L_x(e_{k-1})^* \Big) \cum^{x_k} \tilde{h} \, L_k(w')^* \, (\tfrac{1}{w_k})^*\\
    &= w_k^{-1} \, L_k^{-1}(w')^* J^* \Big(w_k^* \, L_x(e_{k-1})^*
    \cum^x - w_k^* \, \cum^x L_x(e_{k-1})^* \Big) \cum^{x_k} \tilde{h}
    \, L_k(w')^* \, (\tfrac{1}{w_k})^*\\
    &= L_k^{-1}(w')^* J^* \Big(L_x(\bar{w})^* \, \cum^x  - \cum^x
    L_x(\bar{w})^* \Big) \cum^{x_k} \tilde{h} \, w_k^* L_k(w')^* \, (\tfrac{1}{w_k})^*\\
    &= L_k^{-1}(w')^* J^* \Big(L_x(\bar{w})^* \, \cum^x
    - \cum^x L_x(\bar{w})^* \Big) \cum^{x_k} \tilde{h} \, L_k(w')^*
  \end{align*}
  where the first equality uses~\eqref{eq:scaling} twice, the second
  employs~\eqref{eq:vert-subst}, the third the commutation of~$w_k
  e_{11} + e_{22} + \cdots e_{nn}$ with~$L_k^{-1}(w')$ and~$J$, the
  fourth applies again~\eqref{eq:scaling} together with~$w_k^* \,
  L_x(e_{k-1})^* = L_x(\bar{w})^* w_k^*$, and the last~$w_k^* \,
  L_k(w')^* (1/w_k)^* = L_k(w')^*$. Finally, \eqref{eq:coalesc-int}
  follows by substituting the expansion of~$\tilde{h}$, noting that
  the components~$\eta' = \tilde{h}\{n+1\} \in \galg_{(n+1)}^r$
  commute with~$\cum^x$ and~$L_x(\bar{w})^*$ and that~$J^*$ acts
  on them as~$(k \; n\!+\!1)^*$.
\end{proof}

For stating the last operator relation, we need the following
technical result for \emph{composing two substitutions}. A general
result could be formulated but for our present purposes the following
special case will be sufficient, where the first substitution is~$x_1
\mapsto a x_k + b x_l$, and the second is~$x_k \mapsto c x_i + d x_j$.

\begin{lemma}
  \label{ref:comp-subs}
  Let~$(\ogalg_n, \cum^{x_n})_{n \in \NN}$ be a separated Rota-Baxter
  hierarchy over~$K$. Fix~$h \in \ogalg_1$, indices~$i < j \le k < l$,
  and scalars~$a,b,c,d \in K$. Assume~$h[a e_k + b e_l] = u$ is
  expanded into~$u = \eta'_\mu \eta_\mu$ with~$\eta = u\{k\}
  \in \ogalg_{(k)}^r$ and~$\eta' = u\{l\} \in
  \ogalg_{(l)}^r$. Setting~$\eta_\mu = u_\mu(x_k)$, expand
  also~$u_\mu[c e_i + d e_j] = \tilde{u}_\mu$ with~$\gamma'_\mu =
  \tilde{u}_\mu\{j\} \in \ogalg_{(j)}^{s_\mu}$ and~$\gamma_\mu =
  \tilde{u}\{i\} \in \ogalg_{(i)}^{s_\mu}$ into~$\tilde{u}_\mu =
  \gamma'_{\mu\nu} \gamma_{\mu\nu}$, where summation is only
  over~$\nu$. Then in~$\ogalg_{(i,j,l)}$ we have the standard tensor
  expasion
  \begin{equation}
    \label{eq:comp-subs}
    h[ace_i + ade_j + be_l] = \eta'_\mu \, \gamma'_{\mu\nu} \, \gamma_{\mu\nu},
  \end{equation}
  with the sum over~$(\mu, \nu) \in \{1, \dots, s_1 \} \times \cdots
  \times \{1, \dots, s_r\} \cong \{ 1, \dots, s_1 + \cdots + s_r\}$.
\end{lemma}

\begin{proof}
  Choosing~$n \ge l$, write~$M = a e_k + b e_l \in \matmon{K}{n}$
  and~$\tilde{M} = I_{k-1} \oplus (c e_i + d e_j) \in \matmon{K}{n}$ in
  the sense of the embedding. Operating~$\tilde{M}^*$ on~$h[M] =
  \eta'_\mu \eta_\mu$ yields
  \begin{equation*}
    h[ace_i + ade_j + be_l] = h[M \tilde{M}] = \eta'_\mu \,
    \eta_\mu[\tilde{M}]
  \end{equation*}
  since~$\tilde{M}^*$ is an algebra homomorphism which
  leaves~$\eta'_\mu \in \ogalg_{(l)}$ invariant (this follows easily
  from straightness and~$\tilde{M}_{l\bullet} = e_l$). But~$u_\mu =
  \eta_\mu[1 \: k]$ implies~$\tilde{u}_\mu = \eta_\mu[\tilde{M}]$
  since the matrices~$\tilde{M}$ and~$(1 \: k) \, \big(
  (ce_i+de_j)\oplus I_{n-1} \big)$ have the same~$k$-th row (again the
  conclusion follows by straightness). Now~\eqref{eq:comp-subs}
  follows by substituting the tensor expansion of~$\tilde{u}_\mu$.
\end{proof}

We can now state the last important relation, which will enable us to
put one-dimensional integrators into \emph{strictly ascending order}
since two adjacent integrators along the same coordinate axes can be
reduces as the following propostion makes precise.

\begin{proposition}[Coalescence of One-Dimensional Integrators]\ \\
  \label{prop:coalesc}%
  Let~$(\galg_n, \cum^{x_n})_{n \in \NN}$ be a Rota-Baxter hierarchy
  over a field~$K$, and~$\ogalg$ an admissible coefficient domain
  for~$\galg$. Then for~$i \in \NN$, and~$v,w \in K^{n-i}$ with~$w
  \neq 0$ we have on~$\galg_n$ the identity
  \begin{equation}
    \label{eq:contr-int}
    \begin{aligned}
      \cum^{x_i} & h(x_i) \, L_i(w)^* \cum^{x_i} g(x_i) \, L_i(v)^* =
      L_k^{-1}(w')^* \chi''[k \; n\!+\!1]_\mu \times {}\\
      & {} \times \Big( L_i(\bar{w})^* \, \cum^{x_i}
      g(x_i) \, \chi'_\mu \, L_i(v')^* \,
      - \cum^{x_i} g(x_i) \, \chi'_\mu \, L_i(v'+\bar{w})^* \Big) \cum^{x_k}
      \chi_\mu L_k(w')^*,
    \end{aligned}
  \end{equation}
  where we have set
  \begin{align*}
    \tilde{h} &= h[-(v_k'/w_k)e_i + e_j/w_k - e_{n+1}/w_k]\\
    v' &= L_{k-i}^{-1}(w') \, v =
    \trp{(v_{i+1}, \dots, v_k, v_{k+1} - w_{k+1}' v_k, \dots, v_n - w_n'
      v_k)} \in K^{n-i}.
  \end{align*}
  with component vectors~$\chi'' = \tilde{h}\{n+1\}$, $\chi' =
  \tilde{h}\{i\}$, $\chi = \tilde{h}\{k\}$, and~$w'$, $\bar{w}$ from
  Lemma~\ref{lem:coalesc}.
\end{proposition}
\begin{proof}
  Multiplying~\eqref{eq:coalesc-int} by~$g(x_i) \, L_i(v)^*$ from the
  right, the right-hand side of~\eqref{eq:contr-int} becomes
  \begin{equation*}
    L_k^{-1}(w')^*
    \, \eta'[k \; n\!+\!1]_\mu \, \Big( L_i(\bar{w})^* \,
    \cum^{x_i} - \cum^{x_i} L_i(\bar{w})^* \Big) \cum^{x_k}
    \eta_\mu \, L_k(w')^* g(x_i) \, L_i(v)^*.
  \end{equation*}
  Since~$L_k(w')^*$ commutes with~$g(x_i)$ and~\eqref{eq:lmat-desc}
  swaps eliminants, the right-hand factor is
  \begin{align*}
    \cum^{x_k} \eta_\mu \, & L_k(w')^* g(x_i) \, L_i(v)^*
    = g(x_i) \, \cum^{x_k} \eta_\mu \, L_i(v')^* L_k(w')^*\\
    &= g(x_i) \, \gamma'_{\mu\nu} (1-E_k^*) \, L_i(v')^* \cum^{x_k}
    \gamma_{\mu\nu} \, L_k(w')^*,
  \end{align*}
  where the last expression (summation only over~$\nu$) follows from
  Lemma~\ref{lem:after-pivot} by writing $\eta_\mu = u_\mu(x_k)$ and
  defining $\tilde{u}_\mu = u_\mu[e_k - v_k' e_i]$ with component
  vectors~$\gamma'_\mu = \tilde{g}\{i\}$, $\gamma_\mu =
  \tilde{g}\{k\}$. By Lemma~\ref{eq:comp-subs} we can determine~$s \in
  \NN$ and make the identifications~$\eta' = \chi'' \in
  \ogalg_{(n+1)}^s$, $\gamma' = \chi' \in \ogalg_{(i)}^s$, and~$\chi =
  \gamma \in \ogalg_{(k)}^s$. Comparing the above equations
  with~\eqref{eq:contr-int} and noting that~$g(x_i) \, \chi'_\mu \in
  \ogalg_{(i)}$ commutes with~$L_i(\bar{w})^*$, it is clear that it
  suffices now to prove that
  \begin{equation*}
    \Big( L_i(\bar{w})^* \cum^{x_i} - \cum^{x_i} L_i(\bar{w})^* \Big)
    g(x_i) \, \chi'_\mu \, E_k^* = \Big( L_i(\bar{w})^* \cum^{x_i} E_k^* -
    \cum^{x_i} L_i(\bar{w})^* E_k^* \Big) g(x_i) \, \chi'_\mu
  \end{equation*}
  vanishes, where we have used that~$g(x_i) \, \chi'_\mu E_k^* = E_k^*
  \, g(x_i) \, \chi'_\mu$ by~$g(x_i) \, \chi'_\mu \in
  \ogalg_{(i)}$. But the parenthesis on the right-hand side above is
  zero by~$L_i(\bar{w})^* E_k^* = E_k^*$ and
  Corollary~\ref{cor:eval-int-comm}.
\end{proof}

We have now collected all essential operator relations for building up
the partial integral operators as a quotient ring modulo these
relations.

%---------------------------------------------------------------------
\subsection{Construction via Quotient}
\label{ssec:const-quot}
%---------------------------------------------------------------------

The \emph{raw material} for building the operator ring over an
admissible coefficient domain~$\ogalg$ for some Rota-Baxter
hierarchy~$\galg$ comes from the three base algebras: the coefficient
algebra~$\ogalg$; the substitutions from~$\matma$ in the form of a
monoid algebra, and the commuting Rota-Baxter operators (a polynomial
ring in infinitley many variables). These three algebras are put
together by a free product, then we take a suitable quotient modulo
the ideal of operator relations.

We denote the \emph{free product}~\cite[\S III.6]{Cohn1981}
\cite[\S29]{Graetzer2008} \cite[p.~85]{Jacobson1989} of two
$K$-algebras~$A$ and~$B$ by $A \amalg_K\! B$. Rougly speaking, this
consists of $K$-linear combinations of alternating products of
elements of~$A$ and~$B$; for a short introduction see
also~\cite[Example~3.28]{Franz2006} where~$K$ is the complex field
(though the construction works for general fields) or the historical
reference~\cite{Hungerford1968}. One has natural isomorphisms $A
\amalg_K\! (B \amalg_K\!  C) \cong (A \amalg_K\! B) \amalg_K\! C$ so
that we may employ the notation~$A \amalg_K\! B \amalg_K\!
C$. Incidentally, note that the free product is commutative in the
sense that~$A \amalg_K\! B \cong B \amalg_K\!  A$, despite having a
noncommutative multiplication.

Recall that~$\matrices := \substmon{K}$ denotes the \emph{substitution
  monoid} over~$K$, which contains the semigroup~$\nonzero{\matrices}
= \substmon{K} \setminus \{ I \}$. Since its action on~$\galg$
and~$\ogalg \le \galg$ is contravariant, we form the opposite monoid
algebra~$\matma^*$ and the associated opposite semigroup
algebra $\matsa^*$. Its basis elements will be written as~$M^*$
with~$M \in \matma$. For avoiding confusion, we shall write the
\emph{integrals} in the prospective operator ring as~$A_n \; (n \in
\NN)$; they form the commutative polynomial ring~$K[A] = K[A_n \mid n
> 0]$. After setting up an appropriate action~$\odot$ on the given
Rota-Baxter hierarchy~$(\galg_n, \cum^{x_n})_{n \in \NN}$, we shall
have~$A_n \odot f = \cum^{x_n} f$ and~$M^* \odot f = f[M]$ and~$f
\odot g = fg$.

\def\mc{\multicolumn{2}{|l|}}
\begin{table}[h]
  \centering
  \renewcommand{\baselinestretch}{1.5}
  \normalsize
  \begin{equation*}
  \begin{array}[h]{|ll|}
    \hline
    M^* g \cong g[M] \, M^* & M^* \! A_i \cong 0 \quad\text{if $M_{i\bullet} = 0$}\\
    A_j \, g(x_i) \cong g(x_i) A_j & A_i \, g(x_j) \cong g(x_j) A_i\\
    A_j \, g(x_j) M^* \cong
    M_{ij}^{-1} \, \gamma'_\mu (1-E_j^*) \tilde{M}^*
    \cum^{x_i} \gamma[\mmt{j}{1/M_{ij}}]_\mu \, L_i(l)^* &
    \text{if $i := \min \{ k \mid M_{kj} \neq 0 \} \neq \infty$}\\
    A_j \, g(x_j) M^* \cong \big(\cum^{x_j} g(x_j)\big) M^* &
    \text{if $i := \min \{ k \mid M_{kj} \neq 0 \} = \infty$}\\
        \mc{A_j h(x_j) L_j(w)^* A_i \, g(x_i) L_i(v)^* \cong (1-E_j^*)
    A_i \, \eta'_\mu \, L_i(v')^* A_j \, \eta_\mu \, L_j(w)^*}\\
    \mc{A_i \, h(x_i) \, L_i(w)^* A_i \, g(x_i) \, L_i(v)^* \cong
      L_k^{-1}(w')^* \chi''[k \; n\!+\!1]_\mu \times {}}\\
    \mc{\qquad \times \Big( L_i(\bar{w})^* \, \cum^{x_i}
    g(x_i) \, \chi'_\mu \, L_i(v')^* \, - \cum^{x_i} g(x_i) \,
    \chi'_\mu \, L_i(v'+\bar{w})^* \Big) \cum^{x_k} \chi_\mu
    L_k(w')^*}\\
    \mc{A_j \, g(x_j) A_j \cong \big(\cum^{x_j} g(x_j)\big) A_j
      - A_j \big(\cum^{x_j} g(x_j)\big)}\\
    \hline
  \end{array}
  \end{equation*}
  \caption{Relations for Partial Integral Operators and Substitutions}
  \label{fig:red-rules}
\end{table}

\begin{definition}
  \label{def:ring-partial-intops}
  Let~$(\ogalg, \cum)$ be a separated Rota-Baxter hierarchy over a
  field~$K$. Then the \emph{ring of partial integral operators}
  over~$\ogalg$ is defined as the quotient $K$-algebra
  \begin{equation}
    \label{eq:free-prod}
    \ogalg[\cum] := \ogalg \amalg_K\! \matma^* \amalg_K\! K[A] \Bigm/
    \mathord{\cong}
  \end{equation}
  with the congruence~$\cong$ generated by Table~\ref{fig:red-rules}
  for the implicit ranges~$i,j \in \NN$ such that~$i<j$ and~$M \in
  \matma$, and~$g, h \in \ogalg_1$. For all other notations and
  conditions, we refer to the Propositions~\ref{prop:1d-subst-rule},
  \ref{prop:order-int}, and~\ref{prop:coalesc}.
\end{definition}

Note that Table~\ref{fig:red-rules} is meant to contain various
\emph{implied congruences} as special instances of the ones given,
obtained by substituting~$g(x_k) = 1$ in the first rule; $g(x_j) = 1$
in the rules~4, 5, 6; and~$h(x_j) = 1$ or~$g(x_i) = 1$ or both in the
last two rules. Moreover, $v=0$ and hence~$L_i(v) = I$ is allowed in
the last rule. These implicit rewrite rules are also important for the
implementation.

The definition of the free product~$H = H_1 \amalg_K\!  \cdots
\amalg_K\!  H_k$ yields the coproduct in the category of $K$-algebras,
with insertions~$H_k \hookrightarrow H$. This requires for each factor
a \emph{unitary splitting}~$H_i \cong K \oplus \tilde{H}_i$ as
$K$-vector spaces such that the alternating products are built
from~$\tilde{H}_i$ rather than~$H_i$, with~$1_H$ being the empty
product. In our case, the implicit rules will remove the
units~$1_\ogalg \in \ogalg$ and~$I^* \in \matma^*$ and~$1_{K[A]} \in
K[A]$ in favor of the new unit in~$K \subset \ogalg[\cum]$. We use the
unitary splittings
\begin{align*}
  \ogalg &= K \oplus \cup_{n > 0} \, \ogalg_n,\\
  \matma^* &= K \oplus \matsa^*,\\
  K[A] &= K \oplus \{ p \in K[A] \mid \deg(p) > 0\}
\end{align*}
for the three factors in the free product~\eqref{eq:free-prod}.

A detailed description of the \emph{computational realization}
of~$\ogalg[\cum]$ will be the subject of another paper. For our
present purposes it suffices to observe that we can start from the
free $K$-algebra generated by a basis~$(b_i)_{i\in\NN}$ of~$\ogalg$
such that~$b_0 = 1$, the substitutions~$M^* \in \matma$ and the
integrators~$(A_n)_{n > 0}$. Then one has the following extra rules in
addition to the explicit and implicit rules of
Table~\ref{fig:red-rules}:
\begin{itemize}
\item Every occurrence of~$b_0$, $I^*$, and $1_{K[A]}$ is removed,
  while the empty word is treated as the true unit~$1 \in
  \ogalg[\cum]$.
\item After operating on a basis function~$b_i \in \ogalg$ by
  multiplication, integration or substitution, the resulting element
  of~$\ogalg$ is again expanded with respect to the basis~$(b_i)_{i
    \in \NN}$.
\item Within monomials of~$K[A]$, one enforces the ordering~$A_1 > A_2
  > \cdots$, thus achieving normal forms~$A^\alpha = A_1^{\alpha_1}
  A_2^{\alpha_2} \cdots$ for~$\alpha\colon \nonzero{\NN} \to
  \NN$. In view of Table~\ref{fig:red-rules}, reduction
  ensures~$\alpha_i \in \{0, 1\}$.
\end{itemize}
For the subtle interplay between canonical forms and basis expansion,
we refer to~\cite[\S4.4]{RosenkranzRegensburgerTecBuchberger2012},
where this is discussed for ordinary integro-differential operators
(but the same applies here \emph{mutatis mutandis}).

\begin{proposition}
  \label{prop:action}
  Let~$(\galg_n, \cum^{x_n})_{n \in \NN}$ be a Rota-Baxter hierarchy
  over a field~$K$, and let~$\ogalg$ be an admissible coefficient
  domain for~$\galg$. Then the action~$\odot\colon \ogalg[\cum] \times
  \galg \to \galg$ induced by~$g \odot f = gf$, $M^* \odot f = f[M]$
  and~$A_i \odot f = \cum^{x_i} f$ is well-defined.
\end{proposition}
\begin{proof}
  As indicated, one may interpret~$\ogalg[\cum]$ as a quotient of the
  free~$K$-algebra generated by a~$K$-basis~$(b_i)_{i \in \NN}$
  of~$\ogalg$, the substitutions~$M^* \in \matma$ and the
  integrators~$(A_n)_{n > 0}$. Since the implicit rules and the three
  extra ones listed above are obvious consequences, it suffices to
  verify Table~\ref{fig:red-rules} for basis functions~$g, g_j, g_k$.
  Rule~1 is valid because~$\matma$ acts on~$\ogalg \subseteq \galg$
  via $K$-algebra homomorphisms, by hypothesis
  (Definition~\ref{def:hierarchy}). Rule~2 follows from~$M = E_i M$
  and Item~\eqref{it:eval-van} of Lemma~\ref{lem:simple-prop},
  Rules~3, 4 by Item~\eqref{it:trivial-int}. Rules~5, 6 correspond to
  Proposition~\ref{prop:1d-subst-rule}, Rule~7 to
  Proposition~\ref{prop:order-int}, Rule~8 to
  Proposition~\ref{prop:coalesc}. Finally, Rule~9 follws since
  each~$(\ogalg_n, \cum^{x_n})$ is a Rota-Baxter algebra in its own
  right.
\end{proof}

We orient the congruences of Table~\ref{fig:red-rules} from left to
right. Moreover, we agree that on the right-hand side all matrix
products, substitutions and coefficient multiplications are
immediately carried out. For example, the result of Rule~3 is a
$K$-linear combination of monomials having the form~$\ogalg \cdot
\matma^* \cdot A_i \cdot \ogalg_{(i)} \cdot \matma^*$. For showing
termination of the induced rewrite system (including the implicit and
extra rules), we use the following \emph{term order} on the underlying
free algebra. Every monomial is segmented into~$W = w_0 A_{i_1} w_1
A_{i_2} \cdots A_{i_n} w_n$ such that the words~$w_0, w_1, \dots, w_n$
do not contain any~$A$. Note that every word~$w_k$ with~$k>0$ has a
corresponding \emph{integration index}~$i_k$, and by convention we
set~$i_0 := 0$. For any subword~$w_k$, let~$\sigma(w_k)$ denote~$w_k$
viewed as an element in the following Noetherian partial order on the
word monoid over the alphabet~$\ogalg \uplus \matma^*$. We impose no
order amongst the~$g \in \ogalg$ or amongst the~$M^* \in \matma^*$
but we stipulate that~$g < M^*$ whenever~$M \neq I$; this partial
order on the alphabet is extended to a graded lexicographic partial
order on the corresponding words. Furthermore, let~$N(w_k)$ be the
total number of nonzero matrix entries~$M_{r,i_k}$ with~$r<i_k$ and
nonunit matrix entries~$M_{i_k,i_k}$, for all~$M^*$ occurring in the
word~$w_k$. If~$u_l$ is another word, we put now~$w_k < u_l$
iff~$(N(w), -i_k, \sigma(w_k)) < (N(u), -j_l, \sigma(u_l))$ in the
pure lexicographic sense. Finally, let~$U = u_0 A_{j_1} u_1 A_{j_2}
\cdots A_{j_m} u_m$ be a second monomial in the free algebra. Then we
define~$W < U$ iff~$(w_n, \dots, w_1, w_0) < (u_m, \dots, u_1, u_0)$
in the sense of the graded lexicographic order induced by the previous
order given on the components.

\begin{theorem}
  \label{thm:termination}
  Let~$(\ogalg, \cum)$ be a separated Rota-Baxter hierarchy over a
  field~$K$. Orienting the rules of Table~\ref{fig:red-rules} from
  left to right, one obtains a Noetherian reduction system.
\end{theorem}
\begin{proof}
  We note first that all the rules of Table~\ref{fig:red-rules}
  respect the above term order in the sense that all monomials on the
  right-hand side are smaller than the monomial on the left-hand
  side. For Rules~1--4 this is evident. Rule~5 is only applicable
  if~$M$ does not have the form~$L_j(v)$ so that~$N(w)>0$ for the word
  on the left-hand side while clearly~$N(u) = 0$ for all rightmost
  words of the right-hands side monomials. In Rule~6 we have a drop in
  the number of integrators, so the left-hand side is greater due to
  our choice of graded order on subword sequences. Looking at Rule~7
  we see that the last subword~$w$ has~$N(w) = 0$ before and after
  reduction but its negated integration index drops from~$-i$
  to~$-j$. The same is true for Rule~8, where the index drops
  from~$-i$ to~$-k$. Finally, Rule~9 decreases the number of
  integrators and hence drops in the graded order on subword
  sequences.

  Next let us convince ourselves that~$>$ is a term order. If~$1$
  denotes the empty word, we have~$1 < W$ for all non-empty
  words. This is because the empty word has~$N(1) = 0$ and integration
  index~$0$, and~$\sigma(1)$ is the mimimum among all
  $\sigma$-words. Now we assume~$W < U$ and we must show~$WV < UV$
  and~$VW < VU$ for all words~$V$ in the free algebra. But this is
  clear since multiplication by~$V$ on either side equally
  affects~$N(w)$ and the integration index of~$w$ of each segmental
  word~$w$ within~$V$, and the order on the~$\sigma(w)$ respects
  concatenation.

  It remains to prove that~$>$ is Noetherian. By way of contradiction,
  assume an infinite descending chain exists. Since~$>$ is graded on
  subword sequences, there must then be an infinite descending chain
  of words~$w_0 A_{i_1} w_1 A_{i_2} \cdots A_{i_n} w_n$ with fixed
  sequence length~$n$. Since the~$N(w_k)$ and the integration indices
  cannot drop infinitely often, we obtain an infinite descending chain
  of~$\sigma(w_k)$ instances. But this is impossible because the order
  on the~$\sigma(w_k)$ is Noetherian.
\end{proof}

In a Noetherian reduction system, every element has at least one
normal form. Each normal form may be taken to be a representative of
the corresponding congruence class of the quotient
algebra~\eqref{eq:free-prod}. In our case, we can go one step further:
We can specify \emph{monomial normal forms}, such that every element
of the quotient algebra can be written as a $K$-linear combination of
these (congruence classes of) monomomial normal forms. This can be
done as follows.

Let us call a word (of the underlying free algebra) a \emph{line
  integrator of index $i$} if is has the form~$A_i b(x_i) L_i(v)^*$
with a $K$-basis element~$b \in \ogalg_1$, where again~$b = 1$
or~$L_i(v) = I$ means absence. A \emph{volume integrator} is word of
the form~$K = b \, M^* J_1 \cdots J_r$, where~$b \in \ogalg$ is a
basis element and~$M^* \in \matma^*$ with~$M_{i_1\bullet} \neq 0$
for~$r>0$, and where~$J_1, \dots, J_r$ are line integrators of
indices~$i_1 < \cdots < i_r$.

\begin{theorem}
  \label{thm:confluence}
  Let~$(\ogalg, \cum)$ be a separated Rota-Baxter hierarchy over a
  field~$K$. Then every element of~$\ogalg[\cum]$ is a $K$-linear
  combination of terms~$b \, M^* J_1 \cdots J_r$, where~$b \in \ogalg$
  is a basis element and~$M^* \in \matma^*$ with~$M_{i_1\bullet} \neq
  0$ for~$r>0$, and where~$J_1, \dots, J_r$ are integrators of
  indices~$i_1 < \cdots < i_r$.
\end{theorem}
\begin{proof}
  We have to show that every word can be reduced to a linear
  combination of the form stated in the Theorem. Using Rule~1 will
  ensure all subwords~$w_k$ in the canonical segmentation~$w_0 A_{i_1}
  w_1 A_{i_2} \cdots A_{i_n} w_n$ have the form~$b_k \, M_k^*$ for
  some~$b_k \in \ogalg$ and~$M_k \in \matma$. By Rules~3 and~4 we can
  achieve that~$b_k \in \ogalg_{(i_k)}$ and by Rules~5/6 that~$M_k =
  L_{i_k}(v)^*$ for some~$v$; again~$v=0$ means absence. Hence all
  monomials can be reduced to the form~$b \, M^* J_1 \cdots J_r$,
  where~$b \in \ogalg$ is any $K$-basis elements and~$M^* \in
  \matma^*$ is arbitrary. If the integration index drops between any
  neighboring integrators~$J_i, J_{i+1}$, application of Rule~7 fixes
  this. Hence we can ensure that the sequence of integration indices
  is increasing. For achieving a stricly increasing sequence, it
  remains to apply Rule~8 if~$w \neq 0$ and Rule~9 otherwise. Finally,
  Rule~2 is used to guarantee~$M_{i_1\bullet} \neq 0$ if~$r>0$.
\end{proof}

In fact, it is our conjecture that the reduction system induced by
Table~\ref{fig:red-rules} (or perhaps a slight variation of it) is not
only Noetherian but also \emph{confluent} (modulo the~$K$-algebra
axioms, as usual). In other words, we have a noncommutative Gr{\"o}bner
basis for the relation ideal defined by the noncommutative polynomials
corresponding to these rules. As a consequence, the volume integrators
would be a $K$-basis of~$\ogalg[\cum]$.

\begin{conjecture}
  \label{conj:confluence}
  Let~$(\ogalg, \cum)$ be a separated Rota-Baxter hierarchy over a
  field~$K$. Orienting the rules of Table~\ref{fig:red-rules} from
  left to right, one obtains a convergent (i.e.\@ Noetherian and
  confluent) reduction system.
\end{conjecture}

The confluence proof (of which we have completed significant parts)
appears to be rather lengthy and laborious, at least if one follows
the well-known strategy of the \emph{Diamond
  Lemma}~\cite{Bergman1978}. Moreover, such a proof will have a much
more computational flavour than the current article. Therefore we plan
to present the confluence result in a separate publication.

%=====================================================================
\section{Conclusion}
%=====================================================================

As pointed out earlier, our construction of~$\ogalg[\cum]$ is
originally motivated from \emph{computational algebra}: We want to
build up algorithmic support for calculations with integral operators,
substitutions and function expansions. In particular, we aim to
represent, compute and manipulate (e.g.\@ factor) Green's operators
for suitable classes of LPDE boundary problems, as explained in the
Introduction. This necessitates the following further tasks:
\begin{itemize}
\item Adjoining derivations to form partial integro-differential
  operators with linear substitutions.
\item For both operator rings, prove confluence to ensure unique
  normal forms and hence deciding equality (relative to deciding
  equality in the given coefficient domain~$\ogalg$).
\item Develop methods for computing and factoring LPDE Green's
  operators, as for example in~\cite{RosenkranzPhisanbut2013}.
\item Study how the classical Green's function can be extracted from
  the Green's operator of certain LPDE boundary problems.
\item Provide an implementation in a computer algebra package.
\end{itemize}
We plan to address these and related items in future work.

Apart from these rather obvious further developments, the material
presented in this paper also contains some more intrinsic topics that
might deserve a detailed investigation. Our notion of Rota-Baxter
hierarchy~$(\galg, \cum^{x_n})_{n \in \NN}$ is a first attempt to
provide an algebraic framework to study integral operators and linear
substitutions. But one could also try \emph{laying the foundations
  deeper}, for example by using operads with a scaling action. The
operad of Rota-Baxter
algebras~\cite{BaiBellierGuoNi2013,BaiGuoPei2015} has played a key
role in understanding splitting of operads initiated by
Loday~\cite{Loday2002}. In our case, note that the operation $f(x_1)
\mapsto f(x_1 + x_2)$ in the hierarchy~$\galg$ is similar to a
coproduct except that it is of type~$\mathcal{F}_1 \to \mathcal{F}_2$
rather than~$\mathcal{F}_1 \to \mathcal{F}_1 \otimes
\mathcal{F}_2$. Working on multivariate functions~$f(x_1, \dots, x_n)$
and using similar substitutions
\begin{equation}
  \label{eq:operad}
  \Delta_i := \Delta_{i,n}\colon\quad
  x_j \mapsto x_j \: (j<i), \quad x_i \mapsto x_i+x_{i+1}, \quad
  x_j \mapsto x_{j+1} \: (j>i)
\end{equation}
one gets a family of operations $\Delta_{i,n}\colon \galg_n \to
\galg_{n+1}$ generalizing the coproduct of a bialgebra. One may check
that they satisfy the commutation rule~$\Delta_j \, \Delta_i =
\Delta_{i+1} \, \Delta_j$ for~$i \ge j$. If one writes~$\alpha\colon
\galg_2 \to \galg_1$ for the addition map, one has~$\Delta_{i,n}(f) =
f \circ_{n,2,i} \alpha$ for~$f \in \galg_n$. Using the language of
PROPs~\cite{Loday2008,Markl2008,Vallette2007}, the commutation rule
follows from the axioms and the associativity of~$\alpha$. In a
similar vein, we can express the scaling action
by~$\mmt{i}{\lambda}^*(f) = f \circ_{n,1,i} \lambda^*$
where~$\lambda^*\colon \galg_1 \to \galg_1$ is the substitution~$x
\mapsto \lambda x$. From the~$\Delta_{i,n}$ and the~$\mmt{i}{\lambda}$
one can then build a Rota-Baxter hierarchy by decomposing matrices
into permutation matrices, diagonal matrices and evaluations.

The advantage of this approach is that it would lead naturally to a
generalization where~$f \circ_{n,m,i} g$ is allowed for any~$g \in
\galg_m$, not only for~$g=\alpha$ and~$g = \lambda^*$. For the
Rota-Baxter operators we would then require the \emph{general
  substitution rule} of multivariable calculus (and for the
derivations the general chain rule), as it is satisfied
for~$C^\infty(\RR^\infty)$.

Another possible generalization is more geometrically motivated: As
mentioned earlier (Remark~\ref{rem:geometric}),
for~$C^\infty(\RR^\infty)$ one can interpret our formulation of the
chain rule as integration over certain parallelepipeds. However, this
is not sufficient to provide a genuine theory of \emph{integration
  over simplices}. In fact, it cannot: The definition of simplices
requires some notion of orientation (e.g.\@ coming from the order
structure in the case~$K = \RR$) while we have only assumed algebras
over a general field~$K$. Nevertheless, it seems what we are missing
is ``not much'' in a certain sense. Perhaps \emph{oriented
  sets}~\cite[\S2.6, Expl.~9b]{BergeronLabelleLeroux1998} could be
used for endowing the vertices of a simplex with an orientation. On a
more fundamental level, \emph{abstract simplicial complexes} might
provide just the right structure. In either case, the simplex would
have one ``variable vertex'' (running along a fixed~$1$-cell), to
ensure an indefinite integral and hence some kind of Rota-Baxter
operator.

This in turn could be compared to the integration of differential
forms over simplices, a species of \emph{definite integral} that is
crucial in de Rham cohomology (where $k$-chains are built from
$k$-simplices). On the manifold~$\RR$, the connection between the
Rota-Baxter operator of indefinite integration and the definite
integration of differential forms is given by the Fundamental Theorem
of Calculus. Roughly speaking, its generalization to higher dimensions
is Stokes' Theorem. It would be interesting to explore this connection
in the light of Rota-Baxter hierarchies.

\section*{Acknowledgments}

This work was supported by the National Science Foundation of US
(Grant No. DMS 1001855), the Engineering and Physical Sciences
Research Council of UK (Grant No. EP/I037474/1) and the National
Natural Science Foundation of China (Grant No. 11371178).

%=====================================================================
% Bibliography
%=====================================================================

% \bibliographystyle{abbrv}
% \bibliography{../../../../../SVN/Sciop/Bibliography/Master}

\begin{thebibliography}{10}

\bibitem{AlbrecherConstantinescuPirsicEtAl2009}
H.~Albrecher, C.~Constantinescu, G.~Pirsic, G.~Regensburger, and M.~Rosenkranz.
\newblock An algebraic approach to the analysis of gerber-shiu functions.
\newblock {\em Insurance: Mathematics and Economics}, Special Issue on
  Gerber-Shiu Functions:accpeted, 2009.

\bibitem{BaiBellierGuoNi2013}
C.~Bai, O.~Bellier, L.~Guo, and X.~Ni.
\newblock Splitting of operations, {M}anin products, and {R}ota-{B}axter
  operators.
\newblock {\em Int. Math. Res. Not. IMRN}, (3):485--524, 2013.

\bibitem{BaiGuoPei2015}
C.~Bai, L.~Guo, and J.~Pei.
\newblock Splitting of operads and rota-baxter operators on operads.
\newblock Preprint on arXiv:1306.3046., 2015.

\bibitem{Baxter1960}
G.~Baxter.
\newblock An analytic problem whose solution follows from a simple algebraic
  identity.
\newblock {\em Pacific J. Math.}, 10:731--742, 1960.

\bibitem{BeckerWeispfenning1993}
T.~Becker and V.~Weispfenning.
\newblock {\em Gr\"obner bases}, volume 141 of {\em Graduate Texts in
  Mathematics}.
\newblock Springer, New York, 1993.
\newblock A computational approach to commutative algebra, In cooperation with
  Heinz Kredel.

\bibitem{BergeronLabelleLeroux1998}
F.~Bergeron, G.~Labelle, and P.~Leroux.
\newblock {\em Combinatorial species and tree-like structures}, volume~67 of
  {\em Encyclopedia of Mathematics and its Applications}.
\newblock Cambridge University Press, Cambridge, 1998.
\newblock Translated from the 1994 French original by Margaret Readdy, With a
  foreword by Gian-Carlo Rota.

\bibitem{Bergman1978}
G.~M. Bergman.
\newblock The diamond lemma for ring theory.
\newblock {\em Advances in {M}athematics}, 29(2):179--218, August 1978.

\bibitem{BuchbergerRosenkranz2012}
B.~Buchberger and M.~Rosenkranz.
\newblock Transforming problems from analysis to algebra: a case study in
  linear boundary problems.
\newblock {\em J. Symbolic Comput.}, 47(6):589--609, 2012.

\bibitem{Cartier2007}
P.~Cartier.
\newblock A primer of {H}opf algebras.
\newblock In {\em Frontiers in number theory, physics, and geometry. {II}},
  pages 537--615. Springer, Berlin, 2007.

\bibitem{Cohn1981}
P.~M. Cohn.
\newblock {\em Universal algebra}, volume~6 of {\em Mathematics and its
  Applications}.
\newblock D. Reidel Publishing Co., Dordrecht-Boston, Mass., second edition,
  1981.

\bibitem{ConnesKreimer2000}
A.~Connes and D.~Kreimer.
\newblock Renormalization in quantum field theory and the {R}iemann-{H}ilbert
  problem. {I}. {T}he {H}opf algebra structure of graphs and the main theorem.
\newblock {\em Comm. Math. Phys.}, 210(1):249--273, 2000.

\bibitem{Downie2012}
R.~W. Downie.
\newblock An introduction to the theory of quantum groups.
\newblock Master's thesis, Eastern Washington University, June 2012.

\bibitem{Duffy2001}
D.~G. Duffy.
\newblock {\em Green's functions with applications}.
\newblock Studies in Advanced Mathematics. Chapman \& Hall, Boca Raton, FL,
  2001.

\bibitem{EvansGariepy1992}
L.~C. Evans and R.~F. Gariepy.
\newblock {\em Measure theory and fine properties of functions}.
\newblock Studies in Advanced Mathematics. CRC Press, Boca Raton, FL, 1992.

\bibitem{Franz2006}
U.~Franz.
\newblock L\'evy processes on quantum groups and dual groups.
\newblock In {\em Quantum independent increment processes. {II}}, volume 1866
  of {\em Lecture Notes in Math.}, pages 161--257. Springer, Berlin, 2006.

\bibitem{Graetzer2008}
G.~Gr{\"a}tzer.
\newblock {\em Universal algebra}.
\newblock Springer, New York, second edition, 2008.
\newblock With appendices by Gr{\"a}tzer, Bjarni J{\'o}nsson, Walter Taylor,
  Robert W. Quackenbush, G{\"u}nter H. Wenzel, and Gr{\"a}tzer and W. A. Lampe.

\bibitem{Guo2012}
L.~Guo.
\newblock {\em An Introduction to Rota-Baxter Algebras}.
\newblock International Press, 2012.

\bibitem{GuoKeigher2000}
L.~Guo and W.~Keigher.
\newblock Baxter algebras and shuffle products.
\newblock {\em Adv. Math.}, 150(1):117--149, 2000.

\bibitem{GuoLin2015}
L.~Guo and Z.~Lin.
\newblock Representations and modules of rota-baxter algebras.
\newblock Preprint, 2015.

\bibitem{Hoffman1997}
M.~E. Hoffman.
\newblock The algebra of multiple harmonic series.
\newblock {\em J. Algebra}, 194(2):477--495, 1997.

\bibitem{Hoffman2005}
M.~E. Hoffman.
\newblock Algebraic aspects of multiple zeta values.
\newblock In {\em Zeta functions, topology and quantum physics}, volume~14 of
  {\em Dev. Math.}, pages 51--73. Springer, New York, 2005.

\bibitem{Hungerford1968}
T.~W. Hungerford.
\newblock The free product of algebras.
\newblock {\em Illinois J. Math.}, 12:312--324, 1968.

\bibitem{Jacobson1989}
N.~Jacobson.
\newblock {\em Basic algebra. {II}}.
\newblock W. H. Freeman and Company, New York, second edition, 1989.

\bibitem{JianZhang2014}
R.-Q. Jian and J.~Zhang.
\newblock Rota-baxter coalgebras.
\newblock Preprint, September 2014.

\bibitem{Kolchin1973}
E.~Kolchin.
\newblock {\em Differential algebra and algebraic groups}, volume~54 of {\em
  Pure and Applied Mathematics}.
\newblock Academic Press, New York-London, 1973.

\bibitem{Loday2002}
J.-L. Loday.
\newblock Dialgebras.
\newblock In {\em Dialgebras and related operads}, volume 1763 of {\em Lect.
  Notes in Math.}, pages 7--66. Springer, 2002.

\bibitem{Loday2008}
J.-L. Loday.
\newblock Generalized bialgebras and triples of operads.
\newblock {\em Ast\'erisque}, 320:x+116, 2008.

\bibitem{Manetti2012}
M.~Manetti.
\newblock Differerential graded coalgebras.
\newblock Lecture notes from a course on deformation theory held at UniRoma 1,
  2012.

\bibitem{Markl2008}
M.~Markl.
\newblock Operads and {PROP}s.
\newblock In {\em Handbook of algebra. {V}ol. 5}, volume~5 of {\em Handb.
  Algebr.}, pages 87--140. Elsevier/North-Holland, Amsterdam, 2008.

\bibitem{RegensburgerRosenkranz2009}
G.~Regensburger and M.~Rosenkranz.
\newblock An algebraic foundation for factoring linear boundary problems.
\newblock {\em Ann. Mat. Pura Appl. (4)}, 188(1):123--151, 2009.
\newblock DOI:10.1007/s10231-008-0068-3.

\bibitem{Ritt1932}
J.~F. Ritt.
\newblock {\em Differential Equations from the Algebraic Standpoint}.
\newblock American Mathematical Society, New York, 1932.

\bibitem{Ritt1966}
J.~F. Ritt.
\newblock {\em Differential algebra}.
\newblock Dover Publications Inc., New York, 1966.

\bibitem{Rosenkranz2005}
M.~Rosenkranz.
\newblock A new symbolic method for solving linear two-point boundary value
  problems on the level of operators.
\newblock {\em {J}. {S}ymbolic {C}omput.}, 39(2):171--199, 2005.

\bibitem{RosenkranzPhisanbut2013}
M.~Rosenkranz and N.~Phisanbut.
\newblock A symbolic approach to boundary problems for linear partial
  differential equations: Applications to the completely reducible case of the
  {Cauchy} problem with constant coeffcients.
\newblock In {\em CASC}, pages 301--314, 2013.

\bibitem{RosenkranzRegensburger2008a}
M.~Rosenkranz and G.~Regensburger.
\newblock Solving and factoring boundary problems for linear ordinary
  differential equations in differential algebras.
\newblock {\em Journal of Symbolic Computation}, 43(8):515--544, 2008.

\bibitem{RosenkranzRegensburgerTecBuchberger2009}
M.~Rosenkranz, G.~Regensburger, L.~Tec, and B.~Buchberger.
\newblock A symbolic framework for operations on linear boundary problems.
\newblock In V.~P. Gerdt, E.~W. Mayr, and E.~H. Vorozhtsov, editors, {\em
  Computer Algebra in Scientific Computing. Proceedings of the 11th
  International Workshop (CASC 2009)}, volume 5743 of {\em LNCS}, pages
  269--283, Berlin, 2009. Springer.

\bibitem{RosenkranzRegensburgerTecBuchberger2012}
M.~Rosenkranz, G.~Regensburger, L.~Tec, and B.~Buchberger.
\newblock Symbolic analysis of boundary problems: {F}rom rewriting to
  parametrized {G}r{\"o}bner bases.
\newblock In U.~Langer and P.~Paule, editors, {\em Numerical and Symbolic
  Scientific Computing: Progress and Prospects}, pages 273--331. Springer,
  2012.

\bibitem{Rota1969}
G.-C. Rota.
\newblock Baxter algebras and combinatorial identities ({I}, {II}).
\newblock {\em Bull. Amer. Math. Soc.}, 75:325--334, 1969.

\bibitem{Rota1995}
G.-C. Rota.
\newblock Baxter operators, an introduction.
\newblock In {\em Gian-Carlo Rota on Combinatorics, Introductory papers and
  commentaries}. Birkh\"{a}user, Boston, 1995.

\bibitem{Stakgold1979}
I.~Stakgold.
\newblock {\em Green's functions and boundary value problems}.
\newblock John Wiley \& Sons, New York, 1979.

\bibitem{Sweedler1969}
M.~E. Sweedler.
\newblock {\em Hopf algebras}.
\newblock Mathematics Lecture Note Series. W. A. Benjamin, Inc., New York,
  1969.

\bibitem{Traynor1993}
T.~Traynor.
\newblock Change of variable for {H}ausdorff measure.
\newblock Presented at the Workshop di Teoria della Misura e Analisi Reale,
  Grado, Italy, September/October 1993.

\bibitem{Vallette2007}
B.~Vallette.
\newblock A {K}oszul duality for {PROP}s.
\newblock {\em Trans. Amer. Math. Soc.}, 359(10):4865--4943, 2007.

\bibitem{Wu1987}
W.~T. Wu.
\newblock A constructive theory of differential algebraic geometry based on
  works of {J}. {F}. {R}itt with particular applications to mechanical
  theorem-proving of differential geometries.
\newblock In {\em Differential geometry and differential equations ({S}hanghai,
  1985)}, volume 1255 of {\em Lecture Notes in Math.}, pages 173--189.
  Springer, Berlin, 1987.

\bibitem{Zagier1994}
D.~Zagier.
\newblock Values of zeta functions and their applications.
\newblock In {\em First {E}uropean {C}ongress of {M}athematics, {V}ol.\ {II}
  ({P}aris, 1992)}, volume 120 of {\em Progr. Math.}, pages 497--512.
  Birkh\"auser, Basel, 1994.

\end{thebibliography}

\end{document}